\documentclass[aap,preprint]{imsart}

\RequirePackage{amsthm,amsmath,amsfonts,amssymb}
\RequirePackage[numbers]{natbib}
\RequirePackage[colorlinks,citecolor=blue,urlcolor=blue]{hyperref}%% uncomment this for coloring bibliography citations and linked URLs
\RequirePackage{graphicx}%% uncomment this for including figures

\startlocaldefs
\numberwithin{equation}{section}
\theoremstyle{plain}
\newtheorem{theorem}{Theorem}[section]
\newtheorem{proposition}[theorem]{Proposition}
\newtheorem{corollary}{Corollary}[theorem]
\newtheorem{lemma}[theorem]{Lemma}
\theoremstyle{definition}
\newtheorem{assumption}{Assumption}

\newtheorem*{remark}{Remark}
\usepackage{algorithmicx}
\usepackage{listings}
\usepackage{xspace}
\usepackage{algorithm}
\usepackage{algpseudocode}
\newcommand{\R}{\mathbb{R}}
\newcommand{\E}{\mathbb{E}}
\renewcommand{\L}{\mathcal{L}}
\newcommand{\half}{\frac{1}{2}}

\usepackage[dvipsnames]{xcolor}
\newcommand{\mhc}[1]{}%{[{\color{Cerulean} #1}]} % mhc comments
\newcommand{\todan}[1]{}%{{\bf[{\color{BurntOrange}Dan: #1}]}} % comments for dan to address
\newcommand{\todo}[1]{}%{[\textit{{\color{CadetBlue} #1}}]} % 
\newcommand{\redan}[1]{}%{{\itshape\color{ForestGreen}[Re: #1]}}
\endlocaldefs

\begin{document}

\begin{frontmatter}
%%%%%%%%%%%%%%%%%%%%%%%%%%%%%%%%%%%%%%%%%%%%%%
%%                                          %%
%% Enter the title of your article here     %%
%%                                          %%
%%%%%%%%%%%%%%%%%%%%%%%%%%%%%%%%%%%%%%%%%%%%%%
\title{A modified Euler-Maruyama method to simulate a one-dimensional sticky diffusion}
\runtitle{Simulating 1D sticky diffusions}

\begin{aug}
%%%%%%%%%%%%%%%%%%%%%%%%%%%%%%%%%%%%%%%%%%%%%%%
%% Only one address is permitted per author. %%
%% Only division, organization and e-mail is %%
%% included in the address.                  %%
%% Additional information such as            %%
%% identifying the corresponding author must %%
%% be included in in the Acknowledgments     %%
%% section if necessary.                     %%
%% ORCID can be inserted by command:         %%
%% \orcid{0000-0000-0000-0000}               %%
%%%%%%%%%%%%%%%%%%%%%%%%%%%%%%%%%%%%%%%%%%%%%%%
\author[A]{\fnms{Chenqi}~\snm{Jiang}\ead[label=e1]{chenqijiang2028@u.northwestern.edu}}
\author[B]{\fnms{Miranda}~\snm{Holmes-Cerfon}\ead[label=e2]{holmescerfon@math.ubc.ca}}
%%%%%%%%%%%%%%%%%%%%%%%%%%%%%%%%%%%%%%%%%%%%%%
%% Addresses                                %%
%%%%%%%%%%%%%%%%%%%%%%%%%%%%%%%%%%%%%%%%%%%%%%
\address[A]{Department of Industrial Engineering and Management Sciences, Northwestern University\printead[presep={ ,\ }]{e1}}

\address[B]{Department of Mathematics, University of British Columbia\printead[presep={,\ }]{e2}}
\end{aug}

\begin{abstract}
    A sticky diffusion is a process that can stick to and detach from a lower-dimensional boundary. A challenge in simulating such a process is in capturing the change in dimension in a dynamically consistent way. We introduce a numerical algorithm to simulate a one-dimensional sticky diffusion, which sticks to and detaches from a point. Our method is a simple modification of the standard Euler-Maruyama scheme, which chooses with some probability between a reflected Euler-Maruyama update and a jump to the sticky point. We show how to choose this probability to be consistent with the generator of the desired dynamics, and we prove that our scheme converges weakly to a sticky diffusion with order 1. 
\end{abstract}

%\begin{keyword}[class=MSC]
%\kwd[Primary ]{???}
%\kwd{???}
%\kwd[; secondary ]{???}
%\end{keyword}
%
%\begin{keyword}
%\kwd{???}
%\kwd{???}
%\end{keyword}

\end{frontmatter}
%%%%%%%%%%%%%%%%%%%%%%%%%%%%%%%%%%%%%%%%%%%%%%
%% Please use \tableofcontents for articles %%
%% with 50 pages and more                   %%
%%%%%%%%%%%%%%%%%%%%%%%%%%%%%%%%%%%%%%%%%%%%%%
\tableofcontents

%%%%%%%%%%%%%%%%%%%%%%%%%%%%%%%%%%%%%%%%%%%%%%
%%%% Main text entry area:

\newpage 

\section{Introduction}

Sticky diffusions are diffusion processes that can spend a nonzero amount of time on a lower-dimensional boundary. The simplest example is a sticky Brownian motion, which behaves as a Brownian motion away from the origin, but is slowed down at the origin in such a way that it spends a positive measure of time there. A more general sticky diffusion solves a stochastic differential equation (SDE) in an ambient space, but it can stick to a boundary, where it can move according to an SDE that can be different from that in the ambient space. 

Sticky diffusions have been studied mathematically since the pioneering work of Feller in the 1950s \cite{Feller.1952,Feller.1957}, who aimed to classify all possible boundary behaviours of a one-dimension diffusion process. His work spurred subsequent developments in both probability and PDE theory \cite{Venttsel.1959,Ikeda.1961,Ito.1963,Stroock.1971,Watanabe.1971}; see also the  mathematical history by Peskir \cite{Peskir.2015}. 

Somewhat neglected for a few decades, sticky diffusions have regained attention recently because they have been found to be useful models in a wide variety of applications. In finance, they are used to model financial variables whose values may linger near certain threshholds \cite{Anagnostakis.2025},  such as Japanese interest rates during 1995-2004 \cite{Kabanov.2007}, US interest rates following the 2008 economic crisis \cite{Nie.2020}, and US Treasury bill yields in the 1930s  \cite{Longstaff.1992}, which all spent long times hovering near zero. In physics and chemistry, sticky diffusion processes can model the dynamics of small  particles in a fluid with short-ranged adhesive interactions, where particles may form transient bonds \cite{Baxter.1968,Stell.1991,Miller.2004mws,Meng.2010,Holmes-Cerfon.2013,Perry.2015}. 
Biology has found a range of applications of sticky diffusions, including in cell biology, where they have been used to study the capture and release of molecules diffusing on a cell surface \cite{Gandolfi.1985,Bressloff.2023}; epidemiology, where the evolution of pathogen load can have a sticky boundary at zero to capture the nonzero probability of an individual being not infected \cite{Calsina.2012,Pugliese.2018}; and ecology, where they have recently been evoked as a mechanism to sustain the distribution of biomass and in particular the abundance of rare species \cite{Nes.2024}. 
Operations researchers have used sticky diffusions in queuing models and storage processes \cite{Harrison.2016}, and have recently generalized such ideas to networks, modelling e.g. traffic flow  or data transmission, where sticky boundary conditions at the vertices of a network can capture congestion effects \cite{Berry.2025}.
Sticky diffusions have also found a range of applications in mathematical coupling techniques, used to prove convergence rates of sampling algorithms and solutions to  stochastic differential equations \cite{Howitt:2007,Eberle.2019,Durmus.2024}.

Given their burgeoning areas of application, it is natural to ask how to simulate a sticky diffusion. 
Yet, doing so is a challenge because of the need to capture a singular measure on the boundary. Standard methods for simulating SDEs, such as the Euler-Maruyama method \cite{Kloeden.1992}, are not straightforward to adapt, as they take random steps in space so they will never lie exactly on the boundary. Inspired by interacting particle systems, one can approximate a sticky boundary condition by applying a strong, short-ranged force to keep a process near the boundary \cite{Peters.2002qkj,Fantoni.2006,Bou-Rabee.2020}, but such a method is not efficient as resolving the force requires time steps much smaller than the timescales of interest.

An alternative idea, and the one which has seen the most development recently, is to take random steps in \emph{time}, and move on a set of pre-defined points in space, leading to a continuous-time Markov chain  approximation \cite{Bou-Rabee.2018}. This idea was first proposed to simulate a sticky Brownian motion in  \cite{Bou-Rabee.2020}, and it was later extended to more general one-dimensional \cite{Meier.2021fag}  and multidimensional \cite{Meier.2023} sticky diffusions with flat boundaries. While this method works effectively for low-dimensional problems, there remain a few drawbacks, particularly for higher-dimensional problems. One is that to choose a direction to move in, the method requires evaluating a number of rate constants which scales at least linearly but often quadratically or higher with dimension. Two, it can be hard to find a discretization of space which leads to a valid continuous-time Markov chain. The authors in \cite{Meier.2023} propose to compute the discrete moves using the eigenvectors of the diffusion covariance matrix, but in high-dimensional problems such an eigendecomposition can become infeasible. Three, it is not clear how feasible such a method will be when applied to problems with curved boundaries, such as arise when considering distance constraints between point particles. Finally, such a method is significantly different from how diffusion processes and in particular interacting particle systems are normally simulated, making it hard to incorporate into existing simulation methodologies \cite{Anderson.2020,Thompson.2022,frenkel2023understanding}. 

A recent study made steps toward developing a perhaps more natural discretization based on the Euler-Maruyama method \cite{Sharma.2025}. This method takes a (discrete) step based on a discrete version of an Euler-Maruyama proposal step, but if a boundary is crossed then it reflects the step as for a reflecting diffusion \cite{Leimkuhler.2020er8} while simultaneously adjusting the time step to account for the slowdown near the sticky boundary. This method can simulate a sticky diffusion  where the drift and diffusion are the same in the interior and on the boundary, because such a process can be constructed from a time-change of a reflected diffusion process. Unfortunately, it does not seem easy to generalize to allow for different drift and diffusion on the boundary, because such a process cannot be constructed as an aforementioned time-changed process. It also seems a disadvantage that the method never moves exactly to the boundary, as it fails to leverage the many methods developed for simulating processes with constraints \cite{Barth.1995,Vanden-Eijnden.20069yk,Ciccotti.2007,Bou-Rabee.2014,Leimkuhler.2016}.

\emph{Our contribution.}
In this paper we make an important step toward developing a general method to simulate arbitrary sticky diffusions, by proposing a modified Euler-Maruyama method to simulate a one-dimensional sticky reflecting diffusion process, which can jump exactly to and from a boundary. The method is only a small modification of the standard reflected Euler-Maruyama scheme, but it allows the process to lie exactly on the boundary, without requiring it to be a time-change of a reflecting diffusion.  
The method works as follows: when the process is far enough away from its sticky point (which we take to be at $x=0$), the algorithm proposes updates as for the regular Euler-Maruyama (EM) algorithm with reflection. When the process is close enough to 0, it performs a reflected EM update with some probability $\lambda(x)$ depending on its current position $x$, and it moves exactly to 0 with some probability $1-\lambda(x)$.
The function $\lambda$ is chosen to make the transition operator over one step of the discrete scheme, a good approximation for the transition operator over a timestep $h$ of the continuous dynamics. 

We prove our method converges weakly with order $O(h)$. The global error arises from a combination of the local error, i.e. the error over one timestep, and the number of points in each state: near the boundary, or away from the boundary. Specifically, away from the boundary, the local error of the EM update is $O(h^2)$; the number of points in such a state is $O(h^{-1})$ leading to $O(h)$ global convergence. Near the boundary, the local error of our scheme is $O(h^{3/2})$, however the number of points near the boundary is $O(h^{-1/2})$ contributing to the same global error. The most technical part of the proof is bounding the number of points near the boundary. 

While we develop the method in detail here 
for a one-dimensional process, we anticipate the ideas can be extended to a multi-dimensional process by following the heuristics and rigorous convergence tools that we describe in some detail here. An advantage of such an extension compared with other methods, is that we propose standard moves, with the only modification being that we choose from the discrete set of moves with probabilities consistent with a sticky diffusion process. Thus, we expect such a method to be easy to incorporate into existing simulation methodologies, which allow for such discrete choices of dynamics, for example when incorporating chemical reactions.

\emph{Overview of the paper.} In Section \ref{sec:introsticky} we briefly recall the properties of one-dimensional sticky diffusion processes. We also outline some formal tools for deriving their associated forward and backward equations and for understanding their properties. In Section \ref{sec:SBM} we present our method for a sticky Brownian motion (SBM). This section gives an overview of the method, a formal overview of how to derive the function $\lambda$, and some numerical experiments, before presenting our proofs of convergence. We focus on an SBM first, because the ideas are easier to follow, and because the more general case reuses some of the results developed for an SBM. A general one-dimensional sticky diffusion is considered in Section \ref{sec:generalsticky}, following the same structure as for SBM. The main convergence results and the lemmas used to prove them are presented at the end of the section, but the proofs are saved for the Supplement as they are technically involved but use standard techniques.

%%%%%%%%%%%%%%%%%%%%%%
%%%% Intro to sticky diffusions      %%%%
%%%%%%%%%%%%%%%%%%%%%%

\section{A brief introduction to one-dimensional sticky diffusion processes}\label{sec:introsticky}

In this section we give an overview of the key equations characterizing one-dimensional sticky diffusions. Our overview will be mainly formal, with the goal of providing accessible ways of writing down the equations characterizing with such diffusions. 

\subsection{Stochastic differential equation}
A sticky diffusion process is the weak solution to a particular stochastic differential equation (SDE). Though not the main tool we use to characterize such processes, we present the SDE here for completeness. 

We consider a process which, on domain $D=(0,\infty)$, behaves as a diffusion process with drift $b(x)$ and diffusion $\sigma(x)$. At the boundary, $\partial D = \{0\}$, the process is delayed before it is reflected, in such a way that it cumulatively spends a finite, nonzero amount of time there. The process can be constructed as the solution to the following system of stochastic differential equations \cite{Ikeda.1981}:
\begin{subequations}\label{eq:SDE}
\begin{align}
dX_t &= b(X_t)1_D(X_t)dt + \sigma(X_t)1_D(X_t) dW_t + dL_t,\label{sde_a}\\
\frac{\sigma^2(X_t)}{2}1_{\partial D}(X_t)dt &= \kappa dL_t. \label{sde_b}
\end{align}
\end{subequations}
Here $1_A(\cdot)$ is the indicator function for set $A$, $W_t$ is a one-dimensional Brownian motion, and $L_t$ is the local time process, a nondecreasing process which increases only when $X_t$ is on the boundary, as
\[
L_t = \int_0^t 1_{\partial D}(X_s)dL_s. 
\]
The constant $\kappa>0$ characterizes how ``sticky'' the origin is. Larger $\kappa$ implies the origin is stickier, while $\kappa=0$ is not at all sticky and corresponds to a purely reflected process. 

The existence and uniqueness of weak solutions to \eqref{eq:SDE} is shown in \cite[Theorem 7.2]{Ikeda.1981}, where a sufficient condition is shown to be that $\sigma,b$ are bounded and Lipschitz continuous on $(0,\infty)$ with $\sigma(0)^2>0$. (This reference shows existence and uniqueness for a general $d$-dimensional sticky diffusion on a half space but we specialize the equation and remarks here to a one-dimensional process.) 

\subsection{Backward equation}
The key object we will work with to show weak convergence of our numerical scheme  is the backward equation associated with \eqref{eq:SDE}. 
We start by recalling the generator for \eqref{eq:SDE}  is the differential operator $\L$ given by 
\begin{equation}\label{eq:L}
    \L f(x) = b(x)\partial_x f + a(x) \partial_{xx}f.
\end{equation}
As usual $a(x) = \half \sigma^2(x)$. The sticky boundary implies a restriction on the domain of the generator, formulated as the sticky boundary condition \cite{Ikeda.1981}
\begin{equation}\label{eq:BC}
    %a(x)\partial_x f = \kappa \L f \qquad \text{at } x=0.
    a\partial_x f = \kappa \L f \qquad \text{at } x=0.
\end{equation}
The backward equation is thus characterized by the following. 

\begin{proposition}
    Let $\phi:[0,\infty)\to \R$ be a bounded function, let $T>0$, and suppose the backward equation 
    \begin{subequations}\label{eq:back}
    \begin{align}
        \partial_s u + \mathcal L u &= 0, & (x\geq 0) \label{back_sde}\\
        a\partial_x u &= \kappa \L u & (x=0) \label{back_bc}\\
        u(T,x) &= \phi(x). \label{back_terminal}
    \end{align}
    \end{subequations}
    has a solution $u:[0,T]\times[0,\infty)\to \R$ which is such that $u,\partial_tu,\partial_xu,\partial_{xx}u$ are continuous and bounded on $[0,T]\times[0,\infty)$, and $\sigma\partial_xu$ is bounded on that same domain. 
    Then $u$ has the representation $u(t,x) = \E[\phi(X_T)|X_t=x]$, where $(X_t)_{t\geq 0}$ is a weak solution to \eqref{eq:SDE}. 
\end{proposition}

\begin{proof}
    Let $u$ solve \eqref{eq:back}. Using It\^{o}'s formula, rewriting $1_D\mathcal L u = \mathcal L u-1_{\partial_D}\mathcal L u$, and using \eqref{sde_b}, we evaluate: 
    \begin{multline*}
    u(T,X_T) \\= u(t,X_t) + \int_t^T (\partial_s u + 1_D\mathcal L u)(s,X_s)ds + \int_t^T (\sigma 1_D\partial_x u)(s,X_s)dW_s
                + \int_t^T \partial_x u(s,X_s)dL_s\\
              =   u(t,X_t) + \int_t^T (\partial_s u + \mathcal L u)(s,X_s)ds + \int_t^T (\sigma 1_D\partial_x u)(s,X_s)dW_s 
                + \int_t^T 1_{\partial D}(-\mathcal Lu + \frac{a}{\kappa}\partial_x u)(s,X_s))ds.
    \end{multline*}
    The deterministic integrals exist because of the regularity assumptions on $u$ and its derivatives, and the stochastic integral is a martingale because of the regularity assumption on $\sigma \partial u$.
    
    The first integrand is identically zero, by \eqref{back_sde}. The last integrand is identically zero, by the sticky boundary condition \eqref{back_bc}. Taking $\E[\cdot | X_t = x]$ shows that $\E[u(X_T,T) | X_t = x] = u(x,t)$. Then we apply the terminal condition \eqref{back_terminal} to obtain the result.
\end{proof}

% \todan{please check you are ok with this proof}
% \redan{I am okay with this proof, modulo a few small but important cleanups. We shall use $u(t,X_t)$ instead of $u(X_t,t)$ to be consistent with \eqref{eq:back}. We should fix the typo in the stochastic-integral limits, use the argument order $u(t,x)$ consistently, and make the regularity assumptions explicit. Bounded $\phi$ is not enough for Itô proof. I propose we let $\phi:[0,\infty)\to\mathbb R$ be bounded. Assume that the backward problem
% \eqref{eq:back} admits a solution $u$ with terminal condition $u(T,x)=\phi(x)$,
% such that $u,\partial_tu,\partial_xu,\partial_{xx}u$ are continuous and bounded (so that deterministic integrals in ito's lemma are valid)
% on $[0,T]\times[0,\infty)$ and that $\sigma\partial_xu$ is bounded (the stochastic integral is a martingale with expectation 0).}
% \todan{I changed all of this, except the typo mentioned -- can you change this? I couldn't see a typo. Please check you are happy with the statement as is. }

\subsection{Stationary distribution}

One way to see the role of parameter $\kappa$ in determining the stickiness of the process \eqref{eq:SDE} is to compute the stationary distribution. 
We will show (formally) that a locally stationary distribution is
\begin{equation}\label{eq:pi}
    \pi(x) = \tilde\pi(x)(1 + \kappa \delta(x)),
\end{equation}
where $\tilde\pi$ is a locally stationary distribution for a Markov process with generator \eqref{eq:L} assuming reflecting boundary conditions. 
That is, $\tilde\pi$ satisfies $\L^*\tilde\pi = 0$, where $\L^*$ is the formal adjoint of $\L$:
\begin{equation}\label{eq:Lstar}
    \L^* p = \partial_x(-b(x)p + \partial_x(a(x)p)). 
\end{equation}
Furthermore we have $-b\pi +\partial_x(a\pi)|_{x=0,\infty}=0$ by the reflecting boundary condition on $\tilde\pi$. 
Hence, the $\kappa$ is directly proportional to the atom of probability on the boundary.

We use a weak formulation of the usual condition $\mathcal L^*\pi=0$ for a stationary distribution, which asks to show that $\langle \pi,\L f\rangle=0$ for all $f$ in the domain of the generator, where $\langle f ,g\rangle=\int_0^\infty fgdx$ is the $L_2$ inner product on $\bar D$. %(This is the weak form of the equation $\partial_t p= 0$, for a probability density $p$ evolving as $\partial_tp = \L^*p$.) 
We make use of an easy calculation: 
\begin{equation}\label{eq:LLstar}
    p\L f = f\L^* p + \partial_x\Big(f\big(pb - \partial_x(pa)\big) + pa\partial_x f\Big).
\end{equation}
Now we may calculate: 
\begin{align*}
    \langle \pi,\L f\rangle &= \tilde \pi\kappa \L f\big|_{x=0} 
    + \Big(\tilde \pi a \partial_x f + f(\tilde\pi b - \partial_x(\tilde \pi a) \Big)\Big|_{x=0}^\infty
    + \int_0^\infty f\L^*\tilde\pi dx\\
    &= \tilde \pi(\kappa \L f-a \partial_x f )\\
    &= 0.
\end{align*}
The last two terms in the first line were zero because of the condition that $\tilde\pi$ was a locally stationary distribution for $\mathcal L^*$. 
The final terms vanished by the boundary condition \eqref{eq:BC}.

\subsection{Derivation via conservation of probability}

We'll now consider an alternative formulation of the evolution equations for a sticky diffusion, in terms of the evolution of its probability measure. This formulation can be derived from simple considerations of conservation of probability. Consider a process that has probability density $p_1(x,t)$ for $x>0$, and which has an atom of probability $p_0(t)$ at 0, so the total density of the process can be written as 
\[
\rho(x,t) = p_1(x,t)1(x>0) + p_0(t)\delta(x).
\]
Suppose $p_1$ evolves with the diffusion dynamics described by the operator $\mathcal L^*$ in \eqref{eq:Lstar}:
\[
\partial_t p_1 = \mathcal L^* p_1.
\]
Let's write $j = bp_1 - \partial_x(ap_1)$ for the flux associated with $\mathcal L^*$, so that $\mathcal L^*p_1 = -\partial_x j$. 

Probability is exchanged with the atom via flux from the exerior:
\[
\partial_t p_0 = -j|_{x=0}.
\]
One can verify this condition conserves the total probability $\int_0^\infty \rho dx$. 

We need one additional condition to close this system and obtain a well-posed system of equations for $p_0,p_1$. One such condition is 
\[
p_0(t) = \kappa p_1(0,t)
\]
for some constant $\kappa>0$.
This condition was shown in \cite{Holmes-Cerfon.2013,Bou-Rabee.2020} to be consistent with the stickiness near zero arising from the limit of a strong, short-ranged force pushing the diffusion process towards 0. 

Putting this together and using that $\partial_t p_0=\kappa\partial_t p_1|_{x=0}  = \kappa \mathcal L^* p_1|_{x=0}$, shows the probability density is 
\begin{equation}\label{eq:rho}
\rho(x,t) = p_1(x,t)(1+\kappa \delta(x))
\end{equation}
where $p_1$ solves
\begin{equation}\label{eq:FP}
\partial_t p_1 = \mathcal L^* p_1 \quad (x\geq 0), \qquad \kappa \mathcal L^* p_1 = -j \quad (x=0),
\end{equation}
with $j = bp_1 - \partial_x(ap_1)$.

\subsection{Equivalence of the forward and backward equations}

Finally, let's show that the operator $\mathcal L^*$ in \eqref{eq:FP} with the given boundary condition, is the (formal) adjoint of $\mathcal L$ with boundary condition \eqref{eq:BC}. We wish to show that 
\begin{equation}\label{eq:adjoint}
\langle f,\partial_t \rho\rangle = \langle \L f,\rho\rangle
\end{equation}
for all test functions $f$ satisfying \eqref{eq:BC}, and for all $\rho$ of the form \eqref{eq:rho} where $p_1$ is a test function (infinitely differentiable with compact support). 
This is a weak form of the usual condition $\langle f,\L^*\rho\rangle = \langle \L f,\rho\rangle$, and is more suited to handling $\delta$-functions. 
We calculate
\[
\langle f,\partial_t \rho\rangle = \int_0^\infty f\partial_t p_1 dx \;\;+ f\kappa\partial_t p_1\Big|_{x=0},
\]
and, using \eqref{eq:LLstar},
\begin{align*}
\langle \L f,\rho\rangle &= \int_0^\infty f\L^* p_1dx \;\; - \Big(fj + p_1a\partial_x f  \Big)\Big|_{x=0}
\; + \kappa p_1\L f\Big|_{x=0}\\
&= \int_0^\infty f\L^* p_1dx - fj\Big|_{x=0}\;\; ,
\end{align*}
by the sticky boundary condition \eqref{eq:BC}.
Thus, we can see that \eqref{eq:adjoint} holds, exactly when $p_1$ satisfies \eqref{eq:FP}.

%%%%%%%%%%%%%%%%%%%%%%
%%%%              SBM          %%%%
%%%%%%%%%%%%%%%%%%%%%%

\section{Sticky Brownian motion}\label{sec:SBM}

In this section we present our algorithm for simulating a sticky Brownian motion (SBM), which has drift $b(x)=0$ and diffusion $\sigma(x)=1$.  We first give a high-level description of the algorithm and our numerical experiments (Section \ref{sec:SBMoverview}), then we show how to formally derive the algorithm, using simple Taylor expansions (Section \ref{sec:lambdaSBM}). In Section \ref{sec:converge1} we present our main convergence results, with proofs in Section \ref{sec:Convergence proofs_SBM}. Section \ref{sec:sufficient_condition_assumptions} proves a sufficient condition for regularity of the solution to the associated backward equation.

\subsection{Overview of the simulation algorithm}\label{sec:SBMoverview}

Consider a SBM  $\{Y_t\}_{t\in [0,T]}$, where $T$ is the total length of time.  We construct a numerical approximation $\{X_k\}_{k=1,\ldots,N}$, where $X_k$ is the value of the approximation at time $t_k=kh$, and $h$ is our time step, so that $N=T/h$ is the total number of steps. We always assume that $h$ is chosen such that $N$ is an integer. 

The numerical scheme works as follows. We divide the domain into two regions: 
\begin{align*}
    \text{boundary layer:}\quad  & \Omega_b=[0,\sqrt{3h}],\\
    \text{exterior:}\quad & \Omega_e = (\sqrt{3h},\infty).
\end{align*}
The algorithm proposes different updates, depending on whether $X_k$ is in the boundary layer or the exterior. 

Given a point in the \emph{exterior}, $X_k \in \Omega_e$,  it is updated with a standard Euler-Maruyama (EM) increment. 
Specifically, 
\begin{equation}
\label{eqn:cont_scheme_outlayer}
    X_{k+1} = X_k + Z_{k},    
\end{equation}
where $\{Z_k\}_{k=0,\ldots,N-1} \sim \mathcal{U}([-\sqrt{3h}, \sqrt{3h}])$ is a sequence of \textit{iid} uniform random variables, which have mean 0 and variance $\text{Var}(Z_k) = h$. We choose uniform random variables rather than normal random variables to simplify our later analysis. 
Notice that, because of the choice of uniform steps,  if $X_k\in \Omega_e$, then $X_{k+1} > 0$, so we do not need to consider reflection at this stage. 
%, however we expect that, as is typical for such schemes, the algorithm will perform equally well with any random variable with mean 0 and variance $h$. 

\begin{figure}
    \centering
    \includegraphics[width=0.85\linewidth]{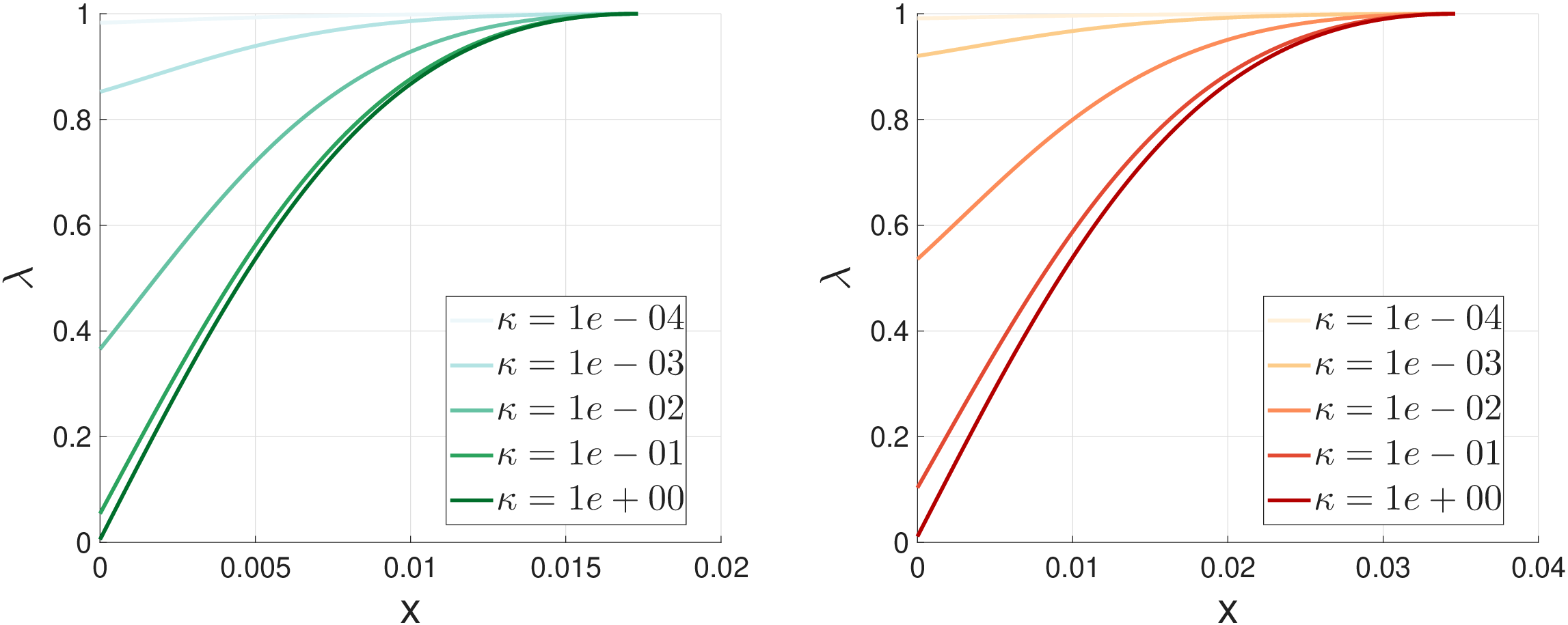}
    \caption{Plot of $\lambda(x)$ for $x\in [0,\sqrt{3h}]$ with various $\kappa$ (shown in the legends) and time step $h=10^{-4}$ (left) and $h=4\times 10^{-4}$ (right). We always have $\lambda(\sqrt{3h}) = 1$, while $\lambda(0)$ is a positive constant for all $\kappa$ which decreases to 0 as $\kappa$ increases. Increasing $h$ has the effect of increasing $\lambda(\cdot)$.
    %Notice that as $\kappa\to 0$, $\lambda(\cdot)\to 1$, and as $\kappa \to \infty$ goes to $+\infty$, $\lambda(\cdot)$ converges to a concave function.
==
    }
    \label{fig:lambda_plot}
\end{figure}

\begin{algorithm}
\caption{Approximation of SBM $X_T$ starting from $x \in [0,\infty)$.}
\label{alg:cont_scheme}
\begin{algorithmic}[1]
\State \textbf{Set} $X_0=x$.
\For{$k = 0$ \textbf{to} $N-1$}
    \State Generate $Z_k \sim \mathcal{U}([-\sqrt{3h},\sqrt{3h}])$.
    \If{$X_k > \sqrt{3h}$}
        \State Update $X_{k+1}$ using \eqref{eqn:cont_scheme_outlayer}.
    \Else \quad ($X_k \in [0,\sqrt{3h}]$)
        \State Calculate $\lambda(X_k)$ using \eqref{eqn:lambda_cont_scheme}.
        \State Generate $u \sim \mathcal{U}([0,1])$.
        \If{$u \leq \lambda(X_k)$}
            \State Update $X_{k+1} = |X_k + Z_k|$.
        \Else
            \State Set $X_{k+1} = 0$.
        \EndIf
    \EndIf
\EndFor
\end{algorithmic}
\end{algorithm}

Given a point $X_k$ in the \emph{boundary layer}, $X_k\in \Omega_b$, it does one of two things: with probability $\lambda(X_k)$ it is updated with a reflected EM increment \cite{Leimkuhler.2020er8}, and with probability $1-\lambda(X_k)$ it jumps exactly to $0$:
\begin{equation}
\label{eqn:cont_scheme_inlayer}
    X_{k+1} =
    \begin{cases} 
    \left|X_k + Z_{k}\right|, & \text{with probability  } \lambda(X_k) \\
    0, & \text{with probability  } 1-\lambda(X_k). \\
   \end{cases}
\end{equation}
We remark that this same update is performed when $X_k=0$; if the second case in \eqref{eqn:cont_scheme_inlayer} is chosen, the process simply stays at 0. 

We will show momentarily that we achieve good convergence properties when $\lambda$ is the following function: 
\begin{equation}\label{eqn:lambda_cont_scheme}
    \lambda(x) = \frac{x^2 + 2\kappa x + h}{\left(\frac{\kappa}{\sqrt{3h}}+1\right)x^2 + h + \kappa\sqrt{3h}}.
\end{equation}
Figure \ref{fig:lambda_plot} shows $\lambda(x)$ for different values of $\kappa$ and $h$. It has many sensible properties: one, $\lambda$ is an increasing function of $x$, with $\lambda(\sqrt{3h}) =1$; this is expected because when $X_k$ is further away from 0, it should have a smaller probability to jump there. Two, as $\kappa$ increases, $\lambda$ decreases, which also makes sense because when the boundary is stickier, there should be a greater chance of jumping there. 

Our algorithm is summarized in algorithm \ref{alg:cont_scheme}. Some realizations from our algorithm are illustrated in Figure \ref{fig:X1_SBM_realization}.  We can see that as sticky parameter $\kappa$ increases, realizations spend more time within the boundary layer.

\begin{figure}
    \centering
    \includegraphics[width=0.99\linewidth]{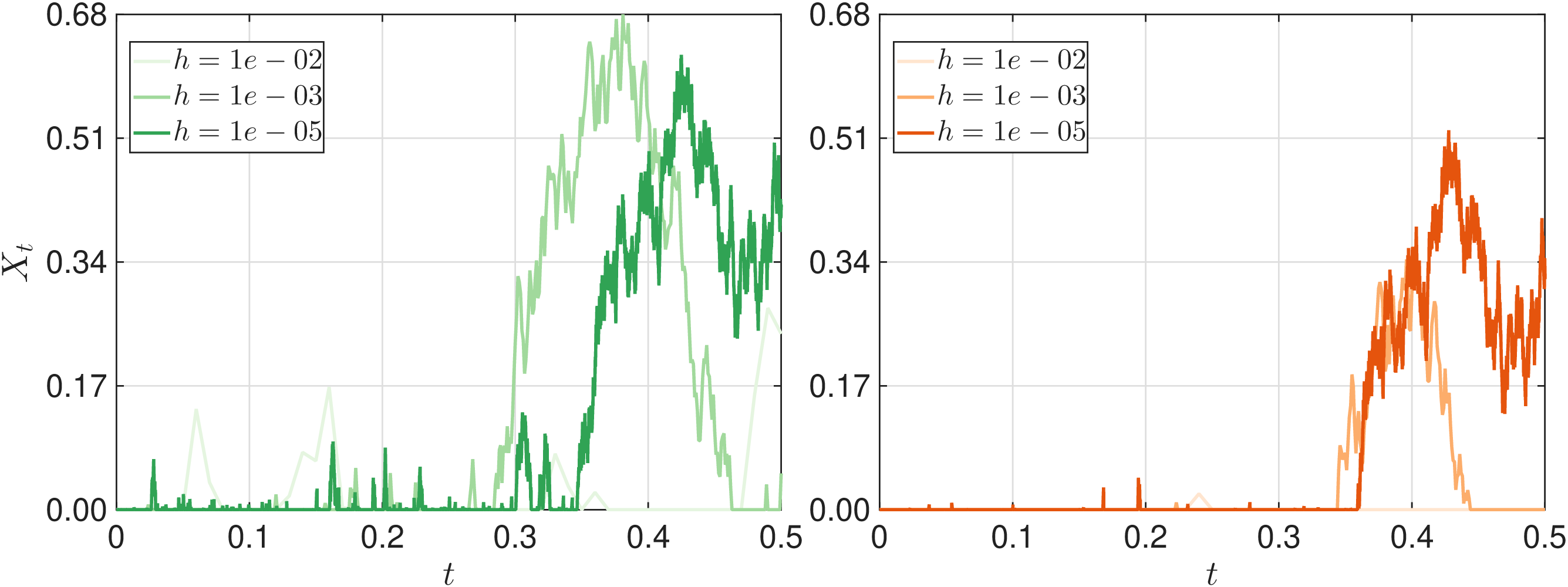}
    \caption{Simulations of an SBM with $X_0=0$ using Algorithm \ref{alg:cont_scheme} with sticky parameter  $\kappa=0.6$ (left) and $\kappa=3$ (right), and with time steps $h \in \{10^{-2}, 10^{-3}, 10^{-5}\}$ as shown in the legend.
    The same realizations for the random variables are used for both plots for the same $h$. 
    Readers can observe how stickiness causes trajectories to spend more time near the sticky boundary, before beginning excursions.
    }
    \label{fig:X1_SBM_realization}
\end{figure}

Note that it is essential that  $\lambda$ be a valid probability, i.e. that $0\leq \lambda(x)\leq 1$. 
This is not guaranteed in the derivation to follow, but is something that must be checked a posteriori. %Let's prove it for the formula \eqref{eqn:lambda_cont_scheme} above. 

\begin{lemma}
\label{lemma:bound_on_lambda_SBM}
The jump probability $\lambda$ defined in \eqref{eqn:lambda_cont_scheme} is a valid probability: 
    $\lambda(x) \in [0,1]$ for all $x \in [0, \sqrt{3h}]$.
\end{lemma}
\begin{proof}
By factoring part of the denominator, we can write $\lambda$ as 
\[
\lambda(x) = \frac{x^2+h+2\kappa x}{x^2+h + 2\kappa x+ \frac{\kappa}{\sqrt{3h}}(x-\sqrt{3h})^2 }.
\]
From this is it clear that 
$\lambda(x)> 0$, since every term in the expression is positive. It is also clear that $\lambda(x)\leq 1$, with equality only when $x=\sqrt{3h}$ or when $\kappa=0$. 
\end{proof}

\subsection{How to choose $\lambda(x)$}\label{sec:lambdaSBM}

The key step of this algorithm is in choosing the function $\lambda$ which governs the probability of jumping directly to 0, instead of performing an EM update. This is done by solving for a $\lambda$ that makes the transition operator of the discrete scheme a good approximation to the transition operator of SBM over one timestep. In this section we derive a formula for $\lambda$ formally, using Taylor expansions of the generator for SBM. 
Our intent is to show the key steps one must follow to apply these ideas more generally; we prove rigorous convergence statements in subsequent sections. 

% \mhc{we use notation $\E_x$ for expectation with initial point $x$? consistent throughout?}\todan{will think about this}
% \redan{I think we can use $\E_x$ here, but we have to explicitly say that we write $\E_x(\cdot):=\E(\cdot\mid X_0=x)$ for expectation under the numerical scheme started from $x$.
% Since the scheme is Markov and time-homogeneous, the same notation may be used in one-step calculations at later times through the identity
% $\E\bigl(g(X_{k+1})\mid X_k=x\bigr)=\E_x\bigl(g(X_1)\bigr).$
% In the global-error analysis, however, we fix a deterministic initial condition $X_0=Y_0=x_0$ and suppress the subscript $x_0$ on unconditional expectations for notational simplicity. See comments/edit at beginning of Sec 3.3 and \eqref{eq:err_decompose}}

Let the transition operator of our numerical scheme with timestep $h$ be denoted by $P^h$; it
characterizes the evolution of statistics via 
\[
(P^h f)(x)  = \E_{x} f(X_1).
\]
We write $\mathbb{E}_x(\cdot):=\mathbb{E}(\cdot|X_{0}=x)$. 
We wish to find $\lambda$ such that $P^h$ is a good approximation for the transition operator of the true dynamics; i.e. such that 
\begin{equation}\label{eq:genapprox0}
(P^h f)(x) -f(x) \approx  \E_x \int_0^h \L f(Y_s)ds\,,
\end{equation}
where $\L f = \frac{1}{2}\partial_{xx}f$ for a SBM. 
We have that 
\[
\E_x\int_0^h  \L f(Y_s)ds\approx h(\L f)(x), 
\]
and
\begin{equation*}
\begin{split}
     (P^h f)(x) - f(x) &=  \int_0^{+\infty} p_h(y|x)\left(f(t,y)-f(t,x)\right)\;dy\\
    &\approx \partial_xf(t,x)\cdot \mathbb{E}_x(\Delta X) + \frac{1}{2}\partial_{xx}f(t,x)\cdot \mathbb{E}_x(\Delta X^2).
\end{split}
\end{equation*}
Here $p_h(y|x)$ is the transition probability to move from $x$ to $y$ in one step of the scheme, and $\Delta X := X_{1} - X_0$.
Thus, asking that the approximations for both sides of \eqref{eq:genapprox0} are within some tolerance say $\epsilon$, our task is to find a $\lambda$ such that 
\begin{equation}\label{eq:genapprox}
    |\mathbb{E}_x(\Delta X)(\partial_xf) + \frac{1}{2} \mathbb{E}_x(\Delta X^2)(\partial_{xx}f) - \frac{h}{2}(\partial_{xx}f)| \leq \epsilon.
\end{equation}
Later, we will show that choosing $\epsilon =O(h^2+hx)$ gives good convergence properties. 
% \[
% \frac{P^h f - f}{h}\approx  \L f,
% \]

% Specifically, we will look for a $\lambda$ such that 
% \begin{equation}\label{eq:genapprox}
%     |P^h f - f - h\L f| \leq C(h^2+hx)
% \end{equation}
% where $C$ is some constant, independent of $x$. The error on the right-hand side is related to the local error of the scheme, i.e. the error it makes in approximating statistics at each timestep.  

% Let us first approximate $P^hf-f$. Writing $p_h(y|x)$ for the transition probability to move from $x$ to $y$, and defining $\Delta X_k := X_{k+1} - X_k$, $\mathbb{E}_x(\Delta X_k):=\mathbb{E}(\Delta X_k|X_{k}=x)$,  we have 
% \begin{equation*}
% \begin{split}
%      P^h f - f &=  \int_0^{+\infty} p_h(y|x)\left(f(t,y)-f(t,x)\right)\;dy\\
%     &\approx \partial_xf(t,x)\cdot \mathbb{E}_x(\Delta X) + \frac{1}{2}\partial_{xx}f(t,x)\cdot \mathbb{E}_x(\Delta X^2).
% \end{split}
% \end{equation*}
% We will later show the neglected terms are $O(h^2+hx)$.

We must choose $\lambda$ such that the coefficients of $\partial_x f, \partial_{xx}f$ on the left-hand side of \eqref{eq:genapprox} are sufficiently small. This is not yet possible since there is only one term involving $\partial_xf$. Thus, our next step is to approximate $\partial_xf$ in terms of $\partial_{xx}f$, using the sticky boundary condition $\partial_xf(t,0) = \kappa \partial_{xx}f(t,0)$. Taylor-expanding $\partial_xf, \partial_{xx}f$ around $x$ shows that
\begin{align*}
        \partial_xf(t,0) &= \partial_xf(t,x) - x\partial_{xx}f(t,x) + \frac{x^2}{2}f^{(3)}(t,z_x),\\
        \partial_{xx}f(t,0) &= \partial_{xx}f(t,x) - xf^{(3)}(t,y_x) ,
\end{align*}
for some $z_x, y_x \in [0,x]$. From the sticky boundary condition, we have
\begin{equation}\label{eq:taylorsticky}
    \partial_xf(t,x) = \kappa\partial_{xx}f(t,x) +x\partial_{xx}f(t,x)- \kappa xf^{(3)}(t,y_x) - \frac{x^2}{2}f^{(3)}(t,z_x).
\end{equation}
Therefore, an approximation of $\partial_xf(t,x)$ using only $\partial_{xx}f(t,x)$ is 
\begin{equation}\label{eq:stickyapprox}
    \partial_xf(t,x) = (\kappa+x)\partial_{xx}f(t,x) + O(h+x).
\end{equation}
%
%\todan{seems like error should be $O(x+h)?$ do you agree? Maybe we should just delete the error term in \eqref{eq:stickyapprox}}
%\redan{Yes, I agree should be $O(x+h)?$, I think it would be fine to delete the big-O and replace $=$ by $\approx$, will be estimated precisely in following proofs}
We remark that to obtain an error of order $O(h+x)$, we do not need the additional term $x\partial_{xx}f(t,x)$; in fact, this term has the same magnitude as some higher-order terms that we are dropping. However, we found that without including this term, we obtained a function $\lambda$ that was not always a valid probability, i.e. $\lambda\notin[0,1]$. %It seems in principle that it would be possible to add terms of the form $\alpha(x)\partial_{xx}f(t,x)$ to the right-hand side of \eqref{eq:stickyapprox} where $\alpha \sim O(h+x)$ and still obtain a valid approximation, however we have not explored whether doing so improves the scheme. 

Substituting for $\partial_xf$ into  \eqref{eq:genapprox}, shows that we must find a function $\lambda$ such that 
\begin{equation}\label{eq:genapprox2}
    |(\kappa+x)\mathbb{E}_x(\Delta X) + \half \mathbb{E}_x(\Delta X^2) - \frac{h}{2} |\approx 0.
\end{equation}
%for some sufficiently small $\epsilon'$.

At this point we may evaluate $\mathbb{E}_x(\Delta X), \mathbb{E}_x(\Delta X^2)$, and then solve \eqref{eq:genapprox2} for $\lambda$, treating \eqref{eq:genapprox2} as an equality. We have that 
\[
p_h(y|x) = (1-\lambda)1_{\{0\}}(y) + \frac{\lambda}{\sqrt{3h}}1_{[0,x+\sqrt{3h}]}(y) 
+ \frac{\lambda}{\sqrt{3h}}1_{[0,|x-\sqrt{3h}|]}(y).
\]
From this, we can calculate
\begin{align}
\mathbb{E}_x(\Delta X_k) &=-x + \frac{\lambda\sqrt{3h}}{2} + \frac{\lambda  x^2}{2\sqrt{3h}}\,,  \nonumber\\
\mathbb{E}_x(\Delta X_k^2) &= -\frac{\lambda x^3}{\sqrt{3h}} + (1+\lambda)x^2 - \sqrt{3h}\lambda x + \lambda  h\,.  \label{eq:DelX}
\end{align}
Substituting into \eqref{eq:genapprox2} and solving for $\lambda$ gives 
gives Formula \eqref{eqn:lambda_cont_scheme}.
% \[
% \frac{\lambda}{2\sqrt{3h}}\left( (\kappa+x)(x^2+3h)  + h\sqrt{3h}  + \sqrt{3h}x^2 - 3hx- x^3\right)
% = \epsilon' + \kappa x + x^2 - \frac{x^2}{2}.
% \]
%Formula \eqref{eqn:lambda_cont_scheme} comes from setting $\epsilon=0$ and solving for $\lambda$. 

\medskip

Our calculations above have demonstrated two Lemmas, which we will use shortly in our convergence proofs. 

\begin{lemma}
\label{lemma:sticky_bc@x}
Consider a function $u(t,x)$ satisfying the sticky boundary condition at $x=0$, i.e. such that $\frac{1}{2}\partial_xu(t,0) = \kappa \partial_{xx}u(t,0)$.  
    Given $x \in \Omega_b$ and $t\in [0,+\infty)$, there exists $z_x,y_x \in [0,x]$ such that 
    \begin{equation}
    \label{eqn:sticky_bc@x}
       \partial_xu(t,x) = (\kappa + x)\cdot \partial_{xx}u(t,x) - \kappa x u^{(3)}(t,y_x) - \frac{x^2}{2}u^{(3)}(t,z_x).
    \end{equation}
\end{lemma}
\begin{proof}
This follows from our calculations deriving \eqref{eq:taylorsticky}. 
    % We know by Taylor expansion with remainder, there exists $z_x, y_x \in [0,x]$
    % \begin{equation}
    % \label{eqn:1st_derivative@0_Taylor}
    %     \partial_xu(t,0) = \partial_xu(t,x) - \partial_{xx}u(t,x) + \frac{x^2}{2}u^{(3)}(t,z_x))
    % \end{equation}
    % \begin{equation}
    % \label{eqn:2nd_derivative@0_Taylor}
    %     \partial_{xx}u(t,0) = \partial_{xx}u(t,x) - xu^{(3)}(t,y_x)
    % \end{equation}
    % Utilizing the sticky boundary condition: $\partial_xu(t,0) = \kappa \partial_{xx}u(t,0)$, I know
    % $$\partial_xu(t,x) - x\partial_{xx}u(t,x) + \frac{x^2}{2}u^{(3)}(t,z_x)) = \kappa \cdot \left( \partial_{xx}u(t, x) - xu^{(3)}(t,y_x) \right)$$
    % which gives
    % $$\partial_xu(t,x) = (\kappa + x)\cdot \partial_{xx}u(t,x) - \kappa xu^{(3)}(t,y_x) - \frac{x^2}{2}u^{(3)}(t,z_x)$$
\end{proof}

\begin{lemma}\label{lem:moments}
For $\lambda$ chosen as \eqref{eqn:lambda_cont_scheme}, and given $x\in \Omega_b$,
    \begin{equation}
    (\kappa+x)\mathbb{E}_x(\Delta X) + \half \mathbb{E}_x(\Delta X^2) - \frac{h}{2} =0.
\end{equation}
\end{lemma}

\begin{proof}
    This follows from our calculations following \eqref{eq:genapprox2}, which showed that we chose $\lambda$ such that the above equation holds. 
\end{proof}

%\todo{From Dan: The only comment I have is about whether we can add extra terms like $\alpha(x)\partial_{xx}f(t,x)$ to RHS of (16). Perhaps I have not explained clearly in my previous email. It is fine to have $\partial_x f(t,x) =\kappa\partial_{xx}f(t,x)$, but we may have $\lambda$ not a valid probability if time step-size $h$ is not small enough. In this particular case, we can have a valid lambda as probability when $h$ is small enough (at least $h < 3\kappa^2$, could be much worse) and prove the same order of local error. I think we can deliver the following message: we can have different $\lambda$ with the same order of local error and prove they are valid eventually/$h$ small enough. They might be impractical to simulate in certain cases (in our case, when $\kappa$ is small). We show our scheme is valid for all $h$ because we derive it by keeping track of $\partial_x f(t,x)$ and $\partial_{xx}f(t,x)$ to our best.}

\subsection{Primary convergence results}\label{sec:converge1}

Now we give an overview of our main convergence results. 
We wish to show weak convergence of the numerical scheme with order $h$. That is, we wish to show that, for $h$ sufficiently small, there is a constant $C>0$ such that 
\begin{equation}
    |\E \phi(X_N) - \E\phi(Y_T)| \leq Ch,
\end{equation}
where $\phi$ is from a suitable class of functions. This is shown by decomposing the error into a local, one-step error, which is different for the boundary and the interior, and a contribution related to the number of steps accumulating each type of error. Our proof strategy closely follows that used to show convergence of diffusions satisfying Robin boundary conditions \cite{Leimkuhler.2020er8}.

At this point we make a small but important remark about the assumed initial condition. 
Throughout the global-error analysis in this subsection, we fix a deterministic initial condition, and assume
$$
X_0=Y_0=x_0.
$$
All expectations and conditional expectations are taken under this initial condition. 
We make this assumption mainly in order to simplify notation and the readability of proofs, as it allows us to suppress the dependence on the initial condition in our notation. Thus, we suppress the dependency of $\E$ on $x_0$, except when the dependence on the starting point needs to be emphasized.

The only place that the initial condition could affect our estimates is in the global error analysis, where we could have a bounding constant depending on $x_0$. 
%This convention is mainly relevant for the constants appearing in the global-error bounds. 
For an SBM, the constants obtained below are uniform in $x_0$. For the sticky diffusion scheme in Section \ref{sec:generalsticky}, the present argument may produce constants that depend on $x_0$. If one instead allows a random initial condition $X_0$, the same argument may be applied conditionally on $X_0$, after which one integrates with respect to the law of $X_0$; in that case one must assume that the resulting $x_0$-dependent bound is integrable.

\medskip

As is typical in proofs of weak convergence \cite{Milstein.2021,Leimkuhler.2020er8}, our main tool will be the backward equation \eqref{eq:back}, repeated here for an SBM for ease of reference:
\begin{gather}\label{eq:backSBM}
    \partial_t u + \frac{1}{2} \partial^2_{xx}u = 0 \qquad (x\geq 0), \\
    \partial_x u(t,0) = \kappa \cdot \partial^2_{xx}u(t,0), \quad
    u(T,x) = \phi(x).\nonumber
\end{gather}
% with boundary condition and terminal condition \mhc{include in prev. eqn?}
% \begin{equation}\label{eq:backbc}
%     \partial_x u|_{x=0} = \kappa \cdot \partial^2_{xx}u|_{x=0}, \quad
%     u(T,x) = \phi(x).
% \end{equation}
 %Let us write 
%$\E_{t,x}\phi(Y_T) \equiv \E[\phi(Y_T) | Y_t=x]$. Thus, 
We may observe that for $X_0=x_0$,
\[
\E(\phi(Y_T)) = u(0,x_0), \qquad \E(\phi(X_N)) = \E(u(T,X_N)),
\]
where the latter comes from observing that
$u(T,X_N) =  \E[\phi(Y_T) | Y_T=X_N] = \phi(X_N)$.
Therefore, writing $t_k=kh$, $u_k=u(t_k,X_k)$, we have 
\begin{equation}\label{eq:telescope}
\E\left(\phi(X_N)\right) - \E\left(\phi(Y_T)\right) = \E\left(\sum_{k=0}^{N-1}u_{k+1}-u_k\right).
\end{equation}
This sum gives us a way to decompose the global error into a sum of local errors, via: %\mhc{need subscript $x$ on outer $\E$? must be careful with different expectations, what initial conditions are}\todan{check conditional expectation}
%\redan{I remove the subscript because the subscript used to represent the initial condition, but I think we can just state that $X_0=x_0$ is a constant random variable for our error analysis so that we do not confuse the subscript used later to represent conditional on $X_k=x$.}
\begin{align}
|\E \phi(X_N) - \E\phi(Y_T)|
&= \left|\mathbb{E}\left(\sum_{k=0}^{N-1} \mathbb{E}\left(u_{k+1}-u_k|X_k\right)\right)\right|\nonumber\\
&\leq \mathbb{E}\left(\sum_{k=0}^{N-1} \left|\mathbb{E}\left(u_{k+1}-u_k|X_k\right)\right|\cdot \mathbf{1}_{X_k \in [0,\sqrt{3h}]}\right) \nonumber\\&\quad  
        + \mathbb{E}\left(\sum_{k=0}^{N-1} \left|\mathbb{E}\left(u_{k+1}-u_k|X_k\right)\right|\cdot \mathbf{1}_{X_k \notin [0,\sqrt{3h}]}\right).
        \label{eq:err_decompose}
\end{align}
The global error thus arises from two things: the local errors $\left|\E\left(u_{k+1}-u_k\mid X_k\right)\right|$ in the boundary  and exterior, and the number of steps in each of these states. 

\medskip

To present our main  results we need some assumptions. %\mhc{see if these assumptions can be reduced to smoothness assumptions on $\phi$. we also need to clarify the space that $\phi$ lives in. }

\begin{assumption}
\label{assump:4th_spatial_derivative_bound}
The solution $u(t,x)$ of \eqref{eq:backSBM} has uniformly bounded spatial derivatives up to order $4$. That is,
$$
\|\partial_x^j u\|_\infty
:=
\sup_{(t,x)\in [0,T]\times[0,\infty)}
\left|\partial_x^j u(t,x)\right|
<\infty,
\qquad j=1,2,3,4.
$$
\end{assumption}

\begin{assumption}
\label{assump:2nd_time_derivative_bound}
    The solution $u(t,x)$ of \eqref{eq:backSBM} has time derivative up to order 2 where $\left\|\partial_{tt} u\right\|_{\infty} < \infty$. %\todan{do you also need time deriative to be bounded?} 
    %\redan{I do not think we need to assume a separate uniform bound on the first time derivative in either the SBM proof or the sticky diffusion proof. For error analysis, we replace the first order time derivative by spatial derivatives. For SBM, $\partial_tu = -\frac{1}{2}\partial_{xx}u$. For sticky diffusion, $\partial_tu = -b(x)\partial_xu - \frac{1}{2}\sigma^2(x)\partial_{xx}u$}
\end{assumption}

We further suppose that $h$ is always chosen such that $N := T/h \in \mathbb{Z}$. 

\begin{theorem}
\label{thm:local_err_cont_scheme}
    (Local Error of $\{X_k\}_{k\in\mathbb{N}}$ in Algorithm \ref{alg:cont_scheme})\\
    %\mhc{add: assumption on form of local error, ie \eqref{eq:genapprox}, or \eqref{eq:genapprox2}, and what size $\epsilon$ must be, for this statement to hold.}
    %\todan{Re--I make similar derivation for sticky diffusion, \eqref{eq:genapprox2_diffusion} shows that $\epsilon \leq C(hx+h^2)$ is the error we can tolerate for \eqref{eq:genapprox2}.}
    Under Assumptions \ref{assump:4th_spatial_derivative_bound},  \ref{assump:2nd_time_derivative_bound}, there exists a constant $M$ depending on $\left\| u^{(3)} \right\|_{\infty}$ $\left\| u^{(4)} \right\|_{\infty}$ and $\left\|\partial_{tt} u\right\|_{\infty}$, such that 
    \begin{equation}
    \label{eqn:local_err_cont_scheme}
        \left|\mathbb{E}\left(u_{k+1}-u_k|X_k=x\right)\right| \leq \begin{cases}
            M( h^2 + hx), & \forall x\in [0,\sqrt{3h}]\\
            Mh^2 & \forall x > \sqrt{3h}.
        \end{cases}
    \end{equation}
\end{theorem}

Thus, the local error of the scheme in the boundary layer is $O(h^{3/2})$ and in the exterior is $O(h^2)$. 

A naive estimate for the global error would be $O(h^{1/2})$, which comes from taking the worst-case local error of $O(h^{3/2})$, and multiplying by the total number of steps which is $O(h^{-1})$. This however is too pessimistic, as will turn out that the number of steps in the boundary layer is $O(h^{-1/2})$, leading to a global error of $O(h^{3/2}\cdot h^{-1/2}) + O(h^2\cdot h^{-1}) \sim O(h)$.  

The specific type of estimate we must construct on the number of steps in the boundary layer arises from this corollary. %, which follows immediately from Theorem \ref{thm:local_err_cont_scheme} by taking the expectation. 

\begin{corollary} Under the assumptions of Theorem \ref{thm:local_err_cont_scheme}, 
\label{corollary:local_err}
    \[\mathbb{E}\left(\sum_{k=0}^{N-1} \left|\mathbb{E}\left(u_{k+1}-u_k|X_k\right)\right|\cdot \mathbf{1}_{X_k \in [0,\sqrt{3h}]}\right) \leq M_1\left(h +  h \cdot \mathbb{E}\left(\sum_{k=0}^{N-1} X_k \cdot \mathbf{1}_{X_k \in [0,\sqrt{3h}]} \right)\right).\]
\end{corollary}

%\todan{should the first $h$ be an $h^2$?}
%\redan{$N = T/h$, so $N\cdot h^2 = Th$}

\begin{proof}
    This is simply the expectation applied to the first case of \eqref{eqn:local_err_cont_scheme}, and summed over all steps. 
\end{proof}

% \begin{proof}
%     Trivial by linearity of expectation and application of \textbf{Theorem} \ref{thm:local_err_cont_scheme}.
% \end{proof}

Our main lemma to estimate the number of points in the boundary layer is the following. 

\begin{lemma}
\label{lemma:key}
    For sufficiently small $h$, there exists a  constant $C>0$, independent of $h$ and initial position $x_0\ge0$, such that
   \[\mathbb{E}\left(\sum_{k=0}^{N-1} X_k \cdot \mathbf{1}_{X_k \in [0,\sqrt{3h}]} \right) < C.\]
\end{lemma}

This lemma implies the number of points in the boundary layer is $O(h^{-1/2})$, because $|X_k| =O(h^{1/2})$ in the boundary layer. It is the most technical lemma to prove. %We explain the key tool for proving it below. \mhc{we do? where?}

Applying this Lemma to \eqref{eq:err_decompose} gives the global error bound, our main convergence theorem. 

% \begin{theorem}
% \label{thm:global_err_cont_scheme}
%     (Global Error of $\{X_k\}_{k\in\mathbb{N}}$ in algorithm \ref{alg:cont_scheme})
%    Under Assumptions \ref{assump:4th_spatial_derivative_bound},  \ref{assump:2nd_time_derivative_bound}, there exists a constant $C$, depending on \mhc{insert--what?}\todan{or doesn't depend on?}, such that \todan{conditions on $\phi$ here?}\redan{See updated version at \ref{thm:global_err_cont_scheme} }
%     \begin{equation}
%     \label{eqn:global_err_cont_scheme}
%         |\mathbb{E}\left(\phi(X_N) - \phi(Y_T)\right)| \leq Ch.
%     \end{equation}
% \end{theorem}

\begin{theorem}
\label{thm:global_err_cont_scheme}
    (Global Error of $\{X_k\}_{k\in\mathbb{N}}$ in Algorithm \ref{alg:cont_scheme})\\
    Let $\phi:[0,\infty)\to \mathbb{R}$ be such that the associated backward solution of \eqref{eq:backSBM},
    $$
    u(t,x):=\mathbb{E}\!\left[\phi(Y_T)\mid Y_t=x\right],
    $$
    satisfies Assumptions~\ref{assump:4th_spatial_derivative_bound} and \ref{assump:2nd_time_derivative_bound}.
    Then there exist constants $\delta>0$ and $C>0$, independent of $h$ and initial position $x_0\ge0$, such that for all $h<\delta$ and $x_0 \geq 0$,
    \begin{equation}
    \label{eqn:global_err_cont_scheme}
        \left|\mathbb{E}\bigl(\phi(X_N)-\phi(Y_T)\bigr)\right|\le Ch.
    \end{equation}
\end{theorem}
\begin{remark}
The main convergence theorems are stated in terms of regularity assumptions on the backward solution $u$, rather than directly on the terminal datum $\phi$. 
A sufficient condition on $\phi$ ensuring Assumptions~\ref{assump:4th_spatial_derivative_bound} and \ref{assump:2nd_time_derivative_bound} is given later in Section~\ref{sec:sufficient_condition_assumptions}; see in particular Proposition~\ref{prop:sufficient_condition_compact_reflecting}. 
For simplicity, that sufficient condition is stated on a compact interval $[0,L]$ with a reflecting boundary condition at $x=L$ (general Wentzell conditions).
In fact, at $0$ and $L$, we can generalize the proof of Proposition \ref{prop:sufficient_condition_compact_reflecting}  to any Wentzell condition.
%If we ensure the compatibility conditions, it is sufficient to make the PDE solution $u$ satisfy Assumptions~\ref{assump:4th_spatial_derivative_bound} and \ref{assump:2nd_time_derivative_bound}.
\end{remark}

\begin{proof}
    This follows from \eqref{eq:err_decompose}, applying Theorem \ref{thm:local_err_cont_scheme}, Corollary \eqref{corollary:local_err}, and Lemma \ref{lemma:key}. 
\end{proof}

Thus, our scheme is weakly convergent with order 1. 

\medskip

\emph{Numerical experiments.}
The convergence results are illustrated numerically in Figure \ref{fig:SBMsim}. 
We run our algorithm with $X_0=0, T=1$ and estimate  $\mathbb{E}[X_T]$. In our algorithm we additionally impose a reflecting boundary condition at $x=L$, which is implemented in our algorithm using a reflected EM update (similar to the update in our algorithm at $x=0$). This is known to converge weakly  with order 1 \cite{Leimkuhler.2020er8}.
We compare our estimated mean, to the solution $u(0,0)$ to the backward equation \eqref{eq:backSBM} with $\phi(x) = x$. This solution is computed using a finite-difference scheme with with extremely small space and time steps; specifically, $\Delta x = 10^{-3}$ and  $\Delta t = \frac{\Delta x^2}{8*5} = 2.5\times10^{-8}$. 

Figure \ref{fig:SBMsim} demonstrates $O(h)$ convergence for two different values of the sticky parameter. It is notable that our choice of $\phi$ does not satisfy the compatibility conditions given in Proposition \ref{prop:sufficient_condition_compact_reflecting}, which are used to show that Assumptions~\ref{assump:4th_spatial_derivative_bound} and \ref{assump:2nd_time_derivative_bound} hold for the solution to the backward equation. Nevertheless, this does not appear to affect our empirical order of convergence.

%\todan{could also show some realizations of the method, eg for the two values of $\kappa$ shown in the Figure. Using different greyscale values, and one can show either (a) multiple realizations in one figure, or (b) realizations with different $h$'s in one figure. }

%\begin{figure}
%    \centering
%    \includegraphics[width=0.48\linewidth]{X1_L2_K5.eps}
%    \includegraphics[width=0.48\linewidth]{X1_L2_K20.eps}
%    \caption{Log-log plot for the weak error $|\mathbb{E}(X_T|X_0=0)-u(T,0)|$ of our simulation scheme
%   \ref{alg:cont_scheme} using $\kappa = 5$ (left) and $\kappa=20$ (right). Both comparisons used $L=2,T=1$. 
%   Error bars are $\pm 1$ empirical standard deviation.\todan{describe how computed -- 1 std dev or 2?} 
%   We used \mhc{insert} realizations to estimate each data point. \todan{please insert number of realizations.}
%   \todan{please compute slope of best-fit line through last 4 data points}
%   \todan{left plot looks like it might need a smaller value of $h$, to see order 1 convergence -- ie to go %down as far as the right plot -- but see what the slope of the best-fit line is.}
%   \todan{could the red line on the right be off-set, eg shited downward?}
%   \todan{fonts need to all be significantly bigger. Also, it would be better if the plots were the same size when generated.}
%    }
%    \label{fig:SBMsim}
%\end{figure}

\begin{figure}
    \centering
    \includegraphics[width=0.98\linewidth]{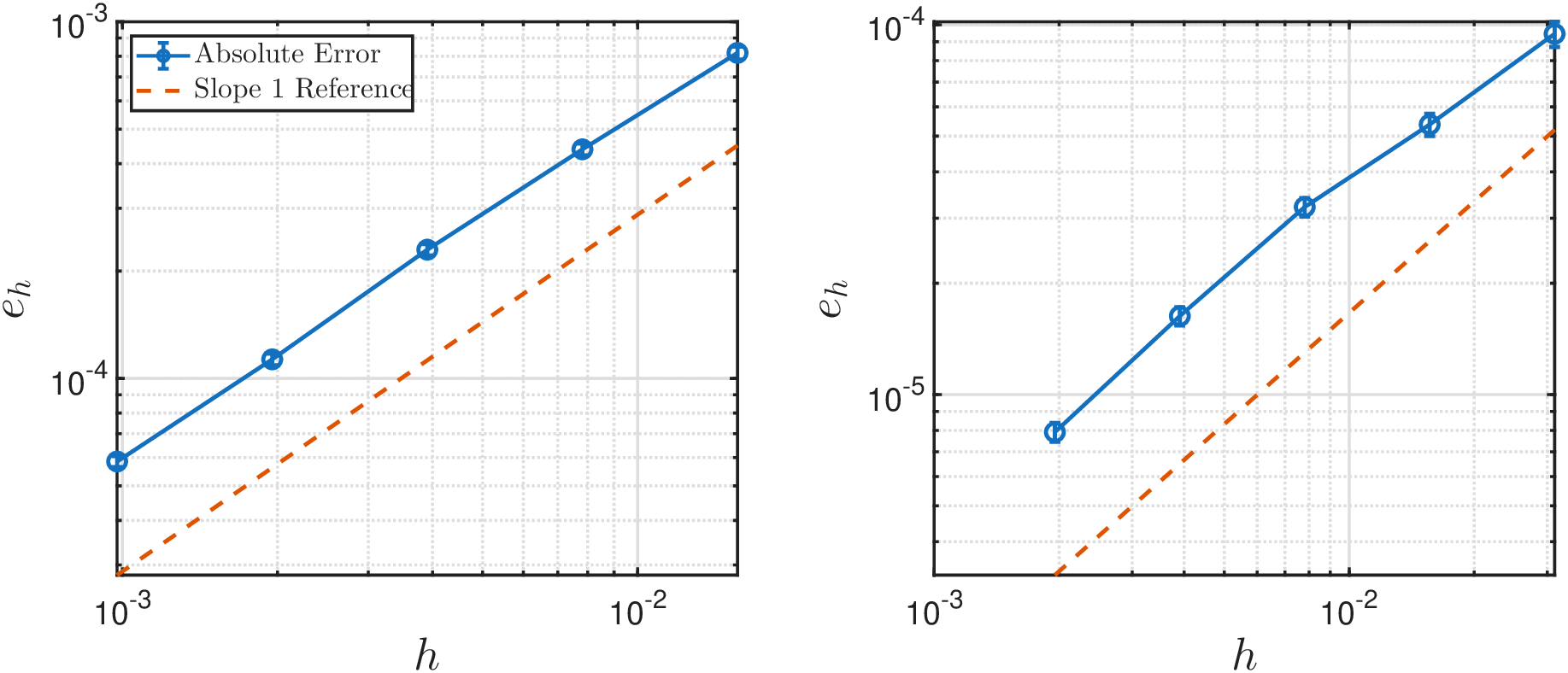}
    \caption{Log-log plot for the weak error $e_h = |\mathbb{E}(X_T|X_0=0)-u(0,0)|$ of our simulation scheme
   \ref{alg:cont_scheme} using $\kappa = 1$ (left) and $\kappa=10$ (right). Both comparisons used $L=2,T=1$. 
   Error bars are $\pm 1$ empirical standard deviation.
   The left plot takes time step-size $h \in \{1/64, 1/128, 1/256, 1/512, 1/1024\}$. To keep the relative error bar similar, we quadruple the number of realizations every time we halve the time step size starting from $2\times 10^8$ realization.
   The least square fit gives the slope of 0.956615.
   The right plot takes time step-size $h \in \{1/32, 1/64, 1/128, 1/256, 1/512\}$. Similarly, to keep the relative error bar similar, we start with $8 \times 10^8$ realizations and we quadruple the number of realizations every time we halve the time step size.
   The least square fit for the last three data points has the slope of 1.011378.
}
    \label{fig:SBMsim}
\end{figure}

\subsection{Convergence proofs}
\label{sec:Convergence proofs_SBM}

It remains to prove Theorem \ref{thm:local_err_cont_scheme} and Lemma \ref{lemma:key}. 
We start with a couple of lemmas regarding the magnitudes of the moments of $\Delta X_k$. 
In this section, we use notation $\E_xf(\Delta X_k) \equiv \E[f(\Delta X_k)|X_k = x]$ for a function $f$ of the increments. 

\begin{lemma} 
\label{lemma:1st_mom_order}
    Given the sequence $\{X_k\}$ generated from algorithm \ref{alg:cont_scheme} with $\kappa>0$, there is a constant $C$ such that 
    \[
    \sup_{x\in [0,\infty)} |\E_x(\Delta X_k^2)| \leq Ch\,, \qquad
    \sup_{x\in [0, \sqrt{3h}]} |\E_x(\Delta X_k)| \leq Ch.\]
\end{lemma}

The second part of the lemma is not trivial, as we might expect $|\E_x(\Delta X_k)|\sim O(h^{1/2})$, since this is the size of a typical jump to 0. 

\begin{proof}
    The bound on $\E_x(\Delta X_k^2)$ follows trivially from the fact that $|\Delta X_k^2|\leq 3h$. 
    The bound on $\E_x(\Delta X_k)$ then follows from Lemma \ref{lem:moments}, which shows that $\E_x(\Delta X_k)$ is a sum of terms which are all $O(h)$ or higher provided $\kappa >0$. 
\end{proof}

%We remark that the second bound would also hold provided $\epsilon \leq Ch$.

%There is a more subtle bound on the third moment of $\Delta X_k$.
\begin{lemma}
\label{lemma:3rd_mom_bound}
    There exists a constant $C$ such that
    \[\sup_{x\in [0, \sqrt{3h}]}\mathbb{E}_x(|\Delta X_k|^3)\leq C(h^2 + h x)\,.
    \]
\end{lemma}

%\mhc{check what is needed -- if only need $h^{3/2}$ then the proof is much easier.}

%\mhc{is there a simpler proof? eg, can we bound the third moment, knowing bounds on the 1st and 2nd moments, and the fact that the distribution has compact support? chebyshev, markov, ... ?}

\begin{proof}
    Fix $x\in[0,\sqrt{3h}]$, an increment $\Delta X_k$ is supported on $[-x,\sqrt{3h}]$.
    We decompose $\Delta X_k$ into its positive and negative parts,
    \[
    \Delta X_k^+=\max\{\Delta X_k,0\},\qquad 
    \Delta X_k^-=\max\{-\Delta X_k,0\},
    \]
    so that $|\Delta X_k|^3=(\Delta X_k^+)^3+(\Delta X_k^-)^3$ and therefore
    \[
    \mathbb{E}_x(|\Delta X_k|^3)
    =\mathbb{E}_x[(\Delta X_k^+)^3]+\mathbb{E}_x[(\Delta X_k^-)^3].
    \]
    
    \medskip\noindent
    \textit{Positive part.}
    Whenever $\Delta X_k>0$ the increment comes from the continuous branch $Z_k$ in Algorithm \ref{alg:cont_scheme}.
    This means the largest positive increment is $\sqrt{3h}$, hence
   \begin{equation}\label{eq:3nd_mom_pospart}
    \mathbb{E}_x[(\Delta X_k^+)^3]\le \lambda(x)\,(3h)^{3/2}.
    \end{equation}
    To bound $\lambda(x)$ in \eqref{eqn:lambda_cont_scheme}, recall that for $x\in[0,\sqrt{3h}]$ we have $x^2\le 3h$.  The
    denominator is
    \[
    \Bigl(\frac{\kappa}{\sqrt{3h}}+1\Bigr)x^2+h+\kappa\sqrt{3h}
    \;\ge\;\kappa\sqrt{3h},
    \]
    The numerator obeys $x^2+2\kappa x+h\le 4h+2\kappa x$.  Therefore there exist constants
    $A,B>0$, independent of $h$ and $x$, such that
    \begin{equation}\label{eq:3nd_mom_lambdabound}
    \lambda(x)\leq \frac{4\sqrt{h}}{\sqrt{3}\,\kappa} + \frac{2x}{\sqrt{3h}}\le A \sqrt{h} + B\frac{x}{\sqrt{h}}.
    \end{equation}
    Combining \eqref{eq:3nd_mom_lambdabound} with \eqref{eq:3nd_mom_pospart} %gives
    % \[
    % \mathbb{E}_x[(\Delta X_k^+)^3]
    % \le (A \sqrt{h} + B\frac{x}{\sqrt{h}})(3h)^{3/2}
    % = 3^{3/2}\cdot A\cdot h^2 + 3^{3/2}\cdot B\cdot hx.
    % \]
    %Hence 
    shows 
    there exist constants $C_1,C_2>0$ for which
    \begin{equation}\label{eq:3nd_mom_pos-final}
    \mathbb{E}_x[(\Delta X_k^+)^3]\le C_1 h^2 + C_2 h x.
    \end{equation}
    
    \medskip\noindent
    \textit{Negative part.}
    Because $\Delta X_k\ge -x$, we have $0\le \Delta X_k^-\le x$, and therefore
    \[
    (\Delta X_k^-)^3 \le x^3.
    \]
    Since $x\in [0,\sqrt{3h}]$,
    \begin{equation}\label{eq:3nd_mom_neg-final}
    \mathbb{E}_x[(\Delta X_k^-)^3] \le x^3 \le 3hx.
    \end{equation}
    
    \medskip\noindent
    Combining \eqref{eq:3nd_mom_pos-final} and \eqref{eq:3nd_mom_neg-final} gives
    \[
    \mathbb{E}_x(|\Delta X_k|^3)
    \;\le\;
    C_1 h^2 + (C_2+3)hx.
    \]
    The bound holds for
    $|\mathbb{E}_x(\Delta X_k^3)|$ with $C=\max\{C_1,C_2+3\}$.  This proves the lemma.
    % \todan{where can I find this proof? also, do you think there is a more direct proof, using e.g. the bounds on the 1st and 2nd moments or the equations they satisfy? Re: Original proof at the top of Supplements on page 30 (Lemma 3.6). I just figured out a more direct proof (not necessarily shorter but cleaner). It is shown as above, I would be appreciated if a double check can be made. The main idea is to exploit the support of increment $\Delta X_k$. I think this proof can be generalized to sticky diffusion, so computation gets much simplified if this works.}
    % \mhc{looks good to me!}
\end{proof}

We are now ready to prove the local error.

\begin{proof}[Proof of Theorem \ref{thm:local_err_cont_scheme}]
    %In this proof, constants $M_1$ and $M_2$ are independent of $x,t,h$ and can vary in different lines.

    Let $u$ be the solution to the backward equation \eqref{eq:backSBM}.
    We perform a Taylor expansion of $u$ to 3rd order centered at $(t_k, X_k)$. We write $u_k=u(t_k,X_k), \partial_xu_k = \partial_xu(t_k,X_k)$ (and similarly for $\partial_{xx}u_k, \partial_{xxx}u_k$), and we write $Y_k$ for a point such that $Y_k\in[X_k,X_{k+1}]$ or $Y_k\in[X_{k+1},X_{k}]$, and $s_k$ will be a point in interval $s_k\in[kh, (k+1)h]$). Expanding gives
    \begin{equation}\label{eq:taylor4}
    \begin{split}
        u_{k+1} - u_k
        &= \partial_t u_k \cdot h  +  \partial_xu_k\cdot \Delta X_k + \frac{1}{2}\partial_{xx}u_k\cdot \Delta X_k^2\\
        &\quad + \frac{1}{2}\partial_{tt} u(s_k, X_k) \cdot h^2 + \frac{1}{6}\partial_{xxx}u(t_k, Y_{k})\cdot \Delta X_k^3.
        %&\quad + \frac{1}{2}\partial_{tt} u(s_k, X_k) \cdot h^2 + \frac{1}{6}\partial_{xxx}u_k\cdot \Delta X_k^3 + \frac{1}{24}\partial_{xxxx}u(t_k, Y_{k})\cdot \Delta X_k^4.
    \end{split}
    \end{equation}

    We are interested in $\E_x(u_{k+1}-u_k)$. We start by assuming $x\in \Omega_b$. We now bound the expectation of each of the terms in the Taylor-expansion \eqref{eq:taylor4}. 
    We have 
    \begin{align*}
    \E_x|\partial_{tt} u(s_k, X_k) \cdot h^2| & \leq h^2\|\partial_{tt}u\|_{\infty}\\
    \E_x|\partial_{xxx}u(t_k, Y_{k})\cdot \Delta X_k^3| &\leq C(h^2+hx)\|u^{(3)}\|_\infty,
    %\E_x|\partial_{xxxx}u(t_k, Y_{k})\cdot \Delta X_k^4|&\leq \|u^{(4)}\|_{\infty}h^2,
    \end{align*}
    where the second inequality follows from Lemma \ref{lemma:3rd_mom_bound}.
    %and the third follows because $\Delta X_k^4 = \mathcal{O}(h^2)$, from Lemma \ref{lemma:1st_mom_order}. 

    We also have, using the fact that $u$ solves the backward equation \eqref{eq:backSBM} and using Lemma \ref{lemma:sticky_bc@x},
    \begin{equation}
        \begin{split}
           & \E_x\left[\partial_t u_k\cdot h + \partial_xu_k\cdot \Delta X_k + \frac{1}{2}\partial_{xx}u_k\cdot (\Delta X_k^2) \right]
            \\
            &\quad = -\frac{1}{2}\partial_{xx} u(t_k,x)\cdot h + (\kappa + x)\cdot\partial_{xx}u(t_k,x) \cdot \mathbb{E}_x(\Delta X_k) 
        + \frac{1}{2}\partial_{xx}u(t_k,x)\cdot \mathbb{E}_x(\Delta X_k^2)\\
        &\qquad \qquad - \E_x(\Delta X_k)\left(\kappa x\cdot u^{(3)}(t_k,y_x) + \frac{x^2}{2}\cdot u^{(3)}(t_k,z_x)\right) \\
        &\quad = - \E_x(\Delta X_k)\left(\kappa x\cdot u^{(3)}(t_k,y_x) + \frac{x^2}{2}\cdot u^{(3)}(t_k,z_x)\right).
        \end{split}
    \end{equation}
    Here $y_x,z_x\in [0,x]$, and the last step used Lemma \ref{lem:moments}. 

    Therefore, applying the triangle inequality to \eqref{eq:taylor4}, and then using Lemma \ref{lemma:1st_mom_order}, shows that there exist constants $M_1,M_2$, independent of $x,t,h$, such that 
   % \begin{equation}\label{eqn:local_err4}
    \begin{align*}
         \left|\mathbb{E}\left(u_{k+1}-u_k|X_k=x\right)\right| 
         & \leq \left|\mathbb{E}_x(\Delta X_k)\right|\cdot \left(\kappa x \|u^{(3)}\|_{\infty} + \frac{h}{2}\|u^{(3)}\|_{\infty}\right) \\
        &\qquad\qquad+ h^2\|\partial_{tt}u\|_{\infty}
        + C(h^2+hx)\|u^{(3)}\|_\infty\\
        %+ \|u^{(4)}\|_{\infty}h^2\\
        &\leq M_1 h^2 + M_2 hx. 
    \end{align*}
    %\end{equation}
    
    We now check $x \in \Omega_e$. 
    Recall that, outside the boundary layer, the simulation scheme makes a standard Euler Maruyama jump, with moments
    $$\mathbb{E}_x(\Delta X_k) = \mathbb{E}_x(\Delta X_k^3) = 0\,, \quad \mathbb{E}_x(\Delta X_k^2) = h, \quad \mathbb{E}_x(\Delta X_k^4) = \frac{9}{5}h^2.$$
    We exploit the Taylor expansion up to fourth order, %\redan{I use fourth order taylor expansion because I do need to worry about estimating $\E_x|\partial_{xxx}u(t_k, Y_{k})\cdot \Delta X_k^3|$ and I do not have lemma similar to \ref{lemma:3rd_mom_bound} where I could pull absolute value inside expectation. Without this lemma, the expectation becomes hard to approximate directly because both the derivative and $\Delta X_k$ are random.}
    \begin{equation}\label{eq:taylor4_ext}
    \begin{split}
        u_{k+1} - u_k
        &= \partial_t u_k \cdot h  +  \partial_xu_k\cdot \Delta X_k + \frac{1}{2}\partial_{xx}u_k\cdot \Delta X_k^2\\
        &\quad + \frac{1}{2}\partial_{tt} u(s_k, X_k) \cdot h^2 + \frac{1}{6}\partial_{xxx}u_k\cdot \Delta X_k^3 + \frac{1}{24}\partial_{xxxx}u(t_k, Y_{k})\cdot \Delta X_k^4.
    \end{split}
    \end{equation}
    Using the fact that $u$ solves the backward equation \eqref{eq:backSBM}, we know,
    \begin{align*}
       & \E_x\left[\partial_t u_k\cdot h + \partial_xu_k\cdot \Delta X_k + \frac{1}{2}\partial_{xx}u_k\cdot (\Delta X_k^2) \right]
        \\
        &\quad = -\frac{1}{2}\partial_{xx} u(t_k,x)\cdot h + \partial_xu_k \cdot 0 
    + \frac{1}{2}\partial_{xx}u(t_k,x)\cdot h\\
    &\quad = 0.
    \end{align*}
    Apply the Taylor expansion \eqref{eq:taylor4} and the fact that $\mathbb{E}_x(\Delta X_k^3) = 0$,
    \begin{align*}
         \left|\mathbb{E}\left(u_{k+1}-u_k|X_k=x\right)\right| 
         &= \left|\frac{1}{2}\partial_{tt} u(s_k, x)\cdot h^2 + 0 + \frac{3}{40}\partial_{xxxx}u(t_k, Y_{k})\cdot h^2\right|\\
        &\leq h^2 \left(\|\partial_{tt}u\|_{\infty} + \|\partial_{xxxx}u\|_{\infty}\right). 
    \end{align*}
\end{proof}

\medskip

It remains to prove Lemma \ref{lemma:key}. 
For this, we first record a standard discrete backward equation for additive functionals of a Markov chain, proven here for completeness \cite{Milstein.2021,Leimkuhler.2020er8}.

\begin{lemma}\label{lem:V}
Let $\{X_i\}_{i=0}^N$ be the Markov chain generated by our numerical scheme, with time grid
$t_i=ih$, $i=0,\dots,N$, and $t_N=T$. Let $P$ denote the transition operator of the scheme, i.e. whenever the conditional expectation is well defined,
\[
PU(t_i,x):=\mathbb{E}\bigl[U(t_{i+1},X_{i+1})\,\big|\,X_i=x\bigr],
\qquad i=0,\dots,N-1.
\]
Let
$$g:\{t_0,\dots,t_{N-1}\}\times [0,\infty)\to \mathbb{R}$$
be measurable, and assume that for every $i=0,\dots,N$ and every $x\in [0,\infty)$,
\[
\mathbb{E}\left(\sum_{k=i}^{N-1} |g(t_k,X_k)|\,\middle|\, X_i=x\right)<\infty.
\]
Define
\[
U(t_i,x):=\mathbb{E}\left(\sum_{k=i}^{N-1} g(t_k,X_k)\,\middle|\, X_i=x\right),
\qquad i=0,\dots,N,
\]
with the convention that the sum is empty when $i=N$. Then $U$ satisfies
\begin{equation}\label{eq:V}
\begin{split}
PU(t_i,x)-U(t_i,x) &= -g(t_i,x),
\qquad (t_i,x)\in \{t_0,\dots,t_{N-1}\}\times [0,\infty),\\
U(t_N,x) &= 0,
\qquad x\in [0,\infty).
\end{split}
\end{equation}
\end{lemma}

\begin{proof}
By assumption, the random variable $\sum_{k=i}^{N-1} g(t_k,X_k)$
is integrable under $\mathbb{P}(\cdot\mid X_i=x)$, so $U(t_i,x)$ is well defined. When $i=N$, the sum is empty, hence
\[
U(t_N,x)=0,
\qquad x\in [0,\infty).
\]
Now fix $i\in\{0,\dots,N-1\}$. Since $X_i=x$ on the conditioning event,
\begin{align*}
U(t_i,x)
&= g(t_i,x)+\mathbb{E}\left(\sum_{k=i+1}^{N-1} g(t_k,X_k)\,\middle|\, X_i=x\right).
\end{align*}
Applying the tower property,
\begin{align*}
\mathbb{E}\left(\sum_{k=i+1}^{N-1} g(t_k,X_k)\,\middle|\, X_i=x\right)
&= \mathbb{E}\left(
\mathbb{E}\left(\sum_{k=i+1}^{N-1} g(t_k,X_k)\,\middle|\, X_{i+1}\right)
\,\middle|\, X_i=x\right)\\
&= \mathbb{E}\bigl[U(t_{i+1},X_{i+1})\,\big|\,X_i=x\bigr]\\
&= PU(t_i,x).
\end{align*}
Therefore,
\[
U(t_i,x)=g(t_i,x)+PU(t_i,x).
\]
\end{proof}

Now, to prove Lemma \ref{lemma:key}, it suffices to show that

$$U(t_0,x)=\mathbb{E}\left(\sum_{k=0}^{N-1} X_k \mathbf{1}_{\{X_k\in[0,\sqrt{3h}]\}} \,\middle|\, X_0=x\right)$$
is uniformly bounded. From Lemma \ref{lem:V}, we know that $U$ solves \eqref{eq:V} with $g(t,x) = x \cdot \mathbf{1}_{x \in [0,\sqrt{3h}]}$.
Therefore, it is enough to construct a uniformly bounded function $V$ such that
\begin{equation*}
\begin{split}
    PV(t_i,x)-V(t_i,x) &\le -x\mathbf{1}_{\{x\in[0,\sqrt{3h}]\}},
\qquad i=0,\dots,N-1,\\
    V(T,x) &\ge 0.
\end{split}
\end{equation*}
Indeed, these two inequalities imply $U(t_i,x)\le V(t_i,x)$ for all grid times by backward induction (shown explicitly in the proof), and in particular
$$U(t_0,x)\le V(t_0,x).$$
Hence the uniform boundedness of $V$ yields the desired bound on $U$.

% Then, Lemma \ref{lemma:key} can be concluded if I construct a uniformly bounded $V(\cdot,\cdot)$ such that $PV(t,x) - V(t,x) \leq -x \cdot \mathbf{1}_{x \in [0,\sqrt{3h}]}$

\begin{proof}[Proof of Lemma \ref{lemma:key}]
%\todan{is there some way to give a bit of intuition into the overall strategy of the construction?}
    % According to Leimkuhler's paper (3.19-3.21)\cite{Leimkuhler.2020er8}, $U(t_0,x) = \mathbb{E}\left(\sum_{k=0}^{N-1} X_k \cdot \mathbf{1}_{X_k \in [0,\sqrt{3h}]} | X_{0} = x\right)$ is the solution of the following boundary value problem where $P$ represents the transition operator of our scheme $PU(t,x) = \mathbb{E}_x(U(t+h, X_{t+h}))$:
    % $$PU(t,x) - U(t,x) = -x \cdot \mathbf{1}_{x \in [0,\sqrt{3h}]},\quad (t,x)\in [t_0,T-h]\times [0,+\infty)$$
    % $$U(T,x) = 0,\quad x\in [0,+\infty)$$
    % it suffices to finish the proof by showing that $U(t_0,x)$ is upper bounded by a uniformly bounded function $V(t,x)$ such that
    % \begin{equation}
    % \label{eqn:V_goal}
    % \begin{split}
    %     PV(t,x) - V(t,x) &= -g(t,x)\\
    %     V(T,x) &= 0,\quad x\in [0,+\infty)\\
    %     g(t,x) &\geq x \cdot \mathbf{1}_{x \in [0,\sqrt{3h}]}
    % \end{split}
    % \end{equation}
    % because
    % $$U(t_0,x) = \mathbb{E}\left(\sum_{k=0}^{N-1} X_k \cdot \mathbf{1}_{X_k \in [0,\sqrt{3h}]} | X_{0} = x\right) \leq \mathbb{E}\left(\sum_{k=0}^{N-1} g(t_k,X_{k})  | X_{0} = x\right) = V(t_0,x)$$
    %\redan{I add an extra paragraph here.}
    We seek a bounded test function whose one-step drift is negative by an amount comparable to
    \(x\) whenever the chain lies in the boundary layer. Once such a function is found, a backward comparison argument will imply \(U\le V\), and hence the desired uniform bound. The function \(e^{w(x)}\) constructed below is maximal near the sticky boundary and decreases away from it, so each visit to \([0,\sqrt{3h}]\) forces a definite loss in one step. The exponential factor \(e^{K(T-t_i)}\) is then used later to absorb the order-\(h\) remainder terms in this drift estimate.
    Readers will see in \eqref{eqn:V_define} later that the function $V(t,x)$ we construct has a structure similar to (3.27) in \cite{Leimkuhler.2020er8}.
    
    We introduce the following function, where $M,s>16$:
    \begin{equation}
    \label{eqn:chosen_wx_continuous}
        w(x) =
        \begin{cases} 
            M , & x\in [0, \sqrt{3h}/2] \\
            M - s\left(x-\frac{\sqrt{3h}}{2}\right), & x\in (\sqrt{3h}/2,+\infty) .
        \end{cases}    
    \end{equation}

    A direct case-by-case computation, provided in the supplement \ref{sec:appendix_sec3}, shows that there exist constants $C_1(s,M)>0$ and $\delta>0$, independent of $x$ and $h$, such that for all $h<\delta$,
    \begin{equation}
    \label{eqn:wx_goal}
        Pe^{w(x)} \leq e^{w(x)}\left(1-f_h(x)+C_1h\right), \qquad \forall x\ge 0,\ \forall h<\delta,
    \end{equation}
    where $f_h$ is defined piecewise by \eqref{eqn:fh_in_0_to_sqrt(3h)/2}, \eqref{eqn:fh_in_sqrt(3h)/2_to_sqrt(3h)}, and \eqref{eqn:fh_in_sqrt(3h)_to_inf} in the supplement, and satisfies
    $$
    f_h(x)\geq x\,\mathbf{1}_{\{x\in[0,\sqrt{3h}]\}}.
    $$
    Now we introduce a new function
    \begin{equation}
    \label{eqn:V_define}
        V(t_i,x) := e^{K(T-t_i)}e^{w(x)}\,,\quad (t,x)\in \{t_0,\dots,t_N\}\times[0,\infty).
    \end{equation}
    % \todan{the above equation seems degenerate -- it is only equal to $e^w$ at $t=T$, otherwise it is 0. Is this intentional?}
    % %Therefore, $V(t,x)$ we define in \eqref{eqn:V_define} \todan{broken reference} satisfies \eqref{eqn:V_goal} and it is uniformly bounded.
    % \redan{Below might be a clearer explanation to replace what I have said from \eqref{eqn:wx_goal} where I am more explicit on the bounding constant for big-$\mathcal{O}$ notations.}\todan{I think I like the extra detail below.}

    Throughout the remainder of the proof, constants are allowed to depend on the fixed parameters $s,M$ (and on $K$ after $K$ is chosen), but are independent of $x$ and $h$.
    After possibly decreasing $\delta$, we may also assume that
    \[
    e^{-Kh}=1-Kh+r_K(h),
    \qquad |r_K(h)|\le C_2(K)h^2,
    \qquad \forall h<\delta.
    \]
    Hence, for every $h<\delta$,
    \begin{equation}\label{eq:PV_step1}
    \begin{split}
    PV(t_i,x)-V(t_i,x)
    &= e^{K(T-t_i-h)}Pe^{w(x)} - e^{K(T-t_i)}e^{w(x)}\\
    &= e^{K(T-t_i)}\Bigl(e^{-Kh}Pe^{w(x)}-e^{w(x)}\Bigr)\\
    &\le e^{K(T-t_i)}e^{w(x)}
    \Bigl[\bigl(1-Kh+C_2(K)h^2\bigr)\bigl(1-f_h(x)+C_1h\bigr)-1\Bigr].
    \end{split}
    \end{equation}
    Expanding the bracket gives
    \begin{equation}\label{eq:PV_step2}
    \begin{split}
    &\bigl(1-Kh+C_2(K)h^2\bigr)\bigl(1-f_h(x)+C_1h\bigr)-1\\
    &\qquad
    = -f_h(x) + (C_1-K)h + K h f_h(x) + C_2(K)h^2\bigl(1-f_h(x)+C_1h\bigr)-K C_1 h^2.
    \end{split}
    \end{equation}
    
    Next, note that $f_h(x)=0$ for $x>\sqrt{3h}$. On the other hand, when $x\in[0,\sqrt{3h}]$, the explicit formulas
    \eqref{eqn:fh_in_0_to_sqrt(3h)/2}--\eqref{eqn:fh_in_sqrt(3h)_to_inf} in the Supplement   together with
    $0\le \lambda(x)\le 1$ imply that
    \[
    0\le f_h(x)\le C_3\sqrt{h},
    \qquad \forall x\ge 0,\ \forall h<\delta,
    \]
    for some constant $C_3>0$ depending only on $s$. Therefore
    \[
    h f_h(x)\le C_3 h^{3/2},
    \qquad \forall x\ge 0,\ \forall h<\delta.
    \]
    After decreasing $\delta$ once more so that $\delta\le 1$, the $h^2$-terms in \eqref{eq:PV_step2}
    can be bounded by $C_4(K)h^{3/2}$ for some constant $C_4(K)>0$. Consequently, there exists a constant
    $C_5(K)>0$ such that $\forall x\ge 0,\ \forall h<\delta$,
    \begin{equation}\label{eq:PV_step3}
    \bigl(1-Kh+C_2(K)h^2\bigr)\bigl(1-f_h(x)+C_1h\bigr)-1
    \le -f_h(x)+(C_1-K)h + C_5(K)h^{3/2}.
    \end{equation}
    Substituting \eqref{eq:PV_step3} into \eqref{eq:PV_step1}, we obtain, $\forall x\ge 0,\ \forall h<\delta$
    \begin{equation}\label{eq:PV_final}
    PV(t_i,x)-V(t_i,x)
    \le
    -\,e^{K(T-t_i)}e^{w(x)}
    \Bigl(
    f_h(x) + (K-C_1)h - C_5(K)h^{3/2}
    \Bigr).
    \end{equation}
    Now choose $K:=C_1+1$. After decreasing $\delta$ again if necessary, we may assume that
    \[
    (K-C_1)h-C_5(K)h^{3/2}\ge 0,
    \qquad \forall h<\delta.
    \]
    Hence \eqref{eq:PV_final} yields
    \[
    PV(t_i,x)-V(t_i,x)\le -\,e^{K(T-t_i)}e^{w(x)}f_h(x),
    \qquad \forall x\ge 0,\ \forall h<\delta.
    \]
    Finally, after decreasing $\delta$ one last time so that
    \[
    M-\frac{s}{2}\sqrt{3h}\ge 0,
    \qquad \forall h<\delta,
    \]
    we have $w(x)\ge 0$ for every $x\in[0,\sqrt{3h}]$, and therefore
    \[
    e^{K(T-t_i)}e^{w(x)}\ge 1
    \qquad \text{on } [0,\sqrt{3h}].
    \]
    Since also $f_h(x)\ge x\,\mathbf{1}_{\{x\in[0,\sqrt{3h}]\}}$, we conclude that
    \begin{equation}
    \label{eq:PV_final_final}
        PV(t_i,x)-V(t_i,x)\le -\,x\,\mathbf{1}_{\{x\in[0,\sqrt{3h}]\}},
    \qquad \forall x\ge 0,\ \forall h<\delta.
    \end{equation}
    % \todan{+ uniformly bounded. is this obvious?}\redan{The uniform bound is $e^{KT+M}$, see next page}
    It remains to compare $V$ with the function
    \[
    U(t_i,x):=\mathbb{E}\left(\sum_{k=i}^{N-1} X_k\mathbf{1}_{\{X_k\in[0,\sqrt{3h}]\}}\,\middle|\,X_i=x\right),
    \qquad i=0,\dots,N.
    \]
    By Lemma \ref{lem:V}, $U$ satisfies
    \begin{equation}
    \label{eq:PU}
        PU(t_i,x)-U(t_i,x)=-x\,\mathbf{1}_{\{x\in[0,\sqrt{3h}]\}},
        \qquad i=0,\dots,N-1,
    \end{equation}
    with terminal condition
    \[
    U(T,x)=0.
    \]
    On the other hand, we have shown that for all sufficiently small $h$,
    \[
    PV(t_i,x)-V(t_i,x)\le -\,x\,\mathbf{1}_{\{x\in[0,\sqrt{3h}]\}},
    \qquad i=0,\dots,N-1.
    \]
    Moreover,
    \[
    V(T,x)=e^{w(x)}\ge 0 = U(T,x).
    \]
    
    We now prove by backward induction that
    \[
    U(t_i,x)\le V(t_i,x), \qquad \forall i=0,\dots,N,\ \forall x\ge 0.
    \]
    The claim is true for $i=N$ by the terminal inequality above. Assume it holds at time $t_{i+1}$, namely
    $U(t_{i+1},x)\le V(t_{i+1},x)$, for all $x\ge 0$.
    This implies
    \[
    PU(t_i,x) = \mathbb{E}\left(U(t_{i+1}, X_{i+1})\mid X_i=x\right)\le \mathbb{E}\left(V(t_{i+1}, X_{i+1})\mid X_i=x\right) = PV(t_i,x), \qquad \forall x\ge 0.
    \]
    Using \eqref{eq:PU}, the supersolution inequality \eqref{eq:PV_final_final}, and $f_h(x) \geq x \cdot \mathbf{1}_{x \in [0,\sqrt{3h}]}$ we obtain
    \begin{align*}
    U(t_i,x)
    &= PU(t_i,x)+x\,\mathbf{1}_{\{x\in[0,\sqrt{3h}]\}}\\
    &\le PV(t_i,x)+x\,\mathbf{1}_{\{x\in[0,\sqrt{3h}]\}}\\
    &\le V(t_i,x).
    \end{align*}
    This closes the induction.
    Finally, since $w(x)\le M$ for all $x\ge 0$, we have
    \[
    V(t_i,x)=e^{K(T-t_i)}e^{w(x)}\le e^{KT+M},
    \qquad \forall i=0,\dots,N,\ \forall x\ge 0.
    \]
    Hence
    \[
    U(t_0,x)\le V(t_0,x)\le e^{KT+M},
    \qquad \forall x\ge 0.
    \]
    %By the tower property and the uniform bound just proved,
    %\[
    %\mathbb{E}\left(\sum_{k=0}^{N-1} X_k\mathbf{1}_{\{X_k\in[0,\sqrt{3h}]\}}\right)
    %=
    %\mathbb{E}\left[
    %\mathbb{E}\left(\sum_{k=0}^{N-1} X_k\mathbf{1}_{\{X_k\in[0,\sqrt{3h}]\}}\,\middle|\,X_0\right)
    %\right]
    %\le e^{KT+M}.
    %\]
    Therefore there exists a constant $C = e^{KT+M}>0$, independent of $h, x_0$, such that
    \[
    \mathbb{E}\left(\sum_{k=0}^{N-1} X_k\mathbf{1}_{\{X_k\in[0,\sqrt{3h}]\}}\right)\le C.
    \]
    
\end{proof}

As a side remark, we in fact use a special case of 3.19-3.21 in Leimkuhler's paper \cite{Leimkuhler.2020er8} where $q(\cdot) = 1$.

\subsection{Sufficient conditions for backward equation regularity}
\label{sec:sufficient_condition_assumptions}

The main convergence theorems are stated in terms of regularity assumptions on the backward solution $u$, rather than directly on the terminal datum $\phi$. 
It is therefore natural to ask for a concrete condition on $\phi$ which guarantees that these regularity assumptions hold. 
On the half-line this requires a separate regularity theory for the sticky/Wentzell problem, which is not esaily available in the literature, so here we record a convenient sufficient condition on a compact interval $[0,L]$. 
For simplicity, we impose a reflecting boundary condition at $x=L$; the same argument extends to a general Wentzell boundary condition, as noted in the remark after proving Proposition \ref{prop:sufficient_condition_compact_reflecting}. 

The key point is that the compatibility conditions are not ad hoc. 
The boundary conditions imposed on $\phi$ ensure that $\phi$ belongs to the domain of the spatial generator $A$ of the SBM, while the additional conditions obtained by applying the same boundary conditions to $\frac12\phi''$ ensure that $A\phi$ also belongs to the domain of $A$, that is, $\phi\in D(A^2)$. 
Once this is established, standard semigroup regularity for the associated boundary-value problem yields uniform bounds on the second time derivative and on the spatial derivatives up to order four. The next proposition makes this precise. %\todan{I defined generator $A$ above, because it was used without saying what it is the generator of -- pls check}\redan{The definition is correct, but $A$ might not be the best notation, we can change if needed.}

\begin{proposition}
\label{prop:sufficient_condition_compact_reflecting}
Let $\phi\in C^4([0,L])$, and let $u(t,x)$ solve
\begin{subequations}\label{eq:back_compact_reflecting}
\begin{align}
\partial_t u + \frac12 \partial_{xx}u &= 0, && (t,x) \in [0,T)\times[0,L], \label{eq:back_compact_reflecting_pde}\\
\partial_x u(t,0) &= \kappa\,\partial_{xx}u(t,0), && t\in[0,T), \label{eq:back_compact_reflecting_sticky}\\
\partial_x u(t,L) &= 0, && t\in[0,T), \label{eq:back_compact_reflecting_reflecting}\\
u(T,x) &= \phi(x), && x\in[0,L]. \label{eq:back_compact_reflecting_terminal}
\end{align}
\end{subequations}
Assume that $\phi$ satisfies compatibility conditions 
$$
\phi'(0)=\kappa \phi''(0), \qquad \phi'(L)=0,
$$
and
$$
\phi^{(3)}(0)=\kappa \phi^{(4)}(0), \qquad \phi^{(3)}(L)=0.
$$
Then
$$
\sup_{(t,x)\in[0,T]\times[0,L]} |\partial_x^j u(t,x)|<\infty,
\qquad j=1,2,3,4,
$$
and
$$
\sup_{(t,x)\in[0,T]\times[0,L]} |\partial_{tt}u(t,x)|<\infty.
$$
In particular, Assumptions~\ref{assump:4th_spatial_derivative_bound} and \ref{assump:2nd_time_derivative_bound} hold on the compact strip $[0,T]\times[0,L]$.
\end{proposition}

\begin{proof}
Set $v(s,x):=u(T-s,x)$ for $(s,x)\in[0,T]\times[0,L]$. Then $v$ solves
\begin{align*}
\partial_s v &= \frac12 \partial_{xx}v, && (s,x) \in (0,T]\times[0,L],\\
\partial_x v(s,0) &= \kappa\,\partial_{xx}v(s,0), && s\in (0,T],\\
\partial_x v(s,L) &= 0, && s\in (0,T],\\
v(0,x) &= \phi(x), && x\in[0,L].
\end{align*}

Let $X:=C([0,L])$, and define $Af:=\frac12 f''$ on
$$
D(A):=\left\{f\in C^2([0,L]) : f'(0)=\kappa f''(0),\ f'(L)=0\right\}.
$$
By Theorem 1.1 in \cite{Engel2003Wentzell} for generalized Wentzell boundary conditions on $C([0,L])$, $A$ generates an analytic semigroup $\{T(s)\}_{s\ge0}$ on $X$. Hence $v(s)=T(s)\phi$.

The first pair of conditions on $\phi$ says exactly that $\phi\in D(A)$. Since $A\phi=\frac12 \phi''$, the second pair of conditions says exactly that $A\phi\in D(A)$. Therefore $\phi\in D(A^2)$.
By the standard differentiability theorem for $C_0$-semigroups (\cite{pazy2012semigroups} Section 2.4), since $\phi\in D(A^2)$, we have $v(s)\in D(A^2)$ for every $s\in[0,T]$, and
$$
v_s(s)=T(s)A\phi, \qquad v_{ss}(s)=T(s)A^2\phi.
$$
Since $\{T(s)\}_{s\ge0}$ is a strongly continuous semigroup on $X=C([0,L])$, it is uniformly bounded on the compact time interval $[0,T]$. Thus, there exists $M_T>0$ such that
$$
\|T(s)f\|_\infty \le M_T\|f\|_\infty,
\qquad f\in X,\quad 0\le s\le T.
$$
Hence
$$
\|v_{ss}(s,\cdot)\|_\infty
=
\|T(s)A^2\phi\|_\infty
\le
M_T\|A^2\phi\|_\infty,
\qquad 0\le s\le T.
$$
%\todan{what is this norm: $\|T(s)\|_{\mathcal L(X)}$ ? }
%\redan{This is the operator norm of $T(s)$ as a bounded linear operator on $X=C([0,L])$ with the supremum norm. To avoid confusion with the differential operator $\mathcal L$, I rewrote the estimate directly as $\|T(s)f\|_\infty\le M_T\|f\|_\infty$.}
Hence, $v_{ss}$ is uniformly bounded,
$$
\|v_{ss}(s,\cdot)\|_\infty \le M_T \|A^2\phi\|_\infty, \qquad 0\le s\le T.
$$
Next, since $Af=\frac12 f''$ and $A^2f=\frac14 f^{(4)}$ in the interior, we obtain $v_{xx}$ and $v_{xxxx}$ are uniformly bounded on $[0,T]\times[0,L]$ by,
$$
\|v_{xx}(s,\cdot)\|_\infty
=2\|Av(s,\cdot)\|_\infty
=2\|T(s)A\phi\|_\infty
\le 2M_T\|A\phi\|_\infty,
$$
$$
\|v_{xxxx}(s,\cdot)\|_\infty
=4\|A^2v(s,\cdot)\|_\infty
=4\|T(s)A^2\phi\|_\infty
\le 4M_T\|A^2\phi\|_\infty.
$$

To control the odd derivatives, we use the boundary conditions and integration. Since $v(s,\cdot)\in D(A)$, we have $\partial_x v(s,0)=\kappa\,\partial_{xx}v(s,0)$ and 
$$
|\partial_x v(s,0)|\le \kappa \|\partial_{xx}v(s,\cdot)\|_\infty.
$$
As a result, $\partial_x v$ is uniformly bounded,
\begin{equation*}
\begin{split}
    |\partial_x v(s,x)|&=|\partial_x v(s,0)|+\left|\int_0^x \partial_{xx}v(s,y)\,dy\right|\\
    &\leq (\kappa+L)\|\partial_{xx}v(s,\cdot)\|_\infty.
\end{split}
\end{equation*}

Similarly, since $v(s,\cdot)\in D(A^2)$, we know $Av(s,\cdot)\in D(A)$. Applying the sticky boundary condition to $Av=\frac12 \partial_{xx}v$ and the uniform bound on fourth order derivative, we obtain $|\partial_{xxx}v(s,0)|\le \kappa \|\partial_{xxxx}v(s,\cdot)\|_\infty$. Hence $\partial_{xxx}v$ is uniformly bounded by similar reasoning above
\begin{equation*}
\begin{split}
    |\partial_{xxx}v(s,x)| &= |\partial_{xxx}v(s,0)|+\left|\int_0^x \partial_{xxxx}v(s,y)\,dy\right|\\
    &\leq (\kappa+L)\|\partial_{xxxx}v(s,\cdot)\|_\infty.
\end{split}
\end{equation*}
We have now shown that
$$
\sup_{(s,x)\in[0,T]\times[0,L]} |\partial_x^j v(s,x)|<\infty,
\qquad j=1,2,3,4,
$$
and
$$
\sup_{(s,x)\in[0,T]\times[0,L]} |\partial_{ss}v(s,x)|<\infty.
$$
Since $u(t,x)=v(T-t,x)$, the same spatial bounds hold for $u$, and $\partial_{tt}u(t,x)=\partial_{ss}v(T-t,x)$. Therefore
$$
\sup_{(t,x)\in[0,T]\times[0,L]} |\partial_x^j u(t,x)|<\infty,
\qquad j=1,2,3,4,
$$
and
$$
\sup_{(t,x)\in[0,T]\times[0,L]} |\partial_{tt}u(t,x)|<\infty.
$$
This completes the proof.
\end{proof}

\begin{remark}
\label{remark:wentzell_extension_compact}
The same argument works if the reflecting boundary condition at $x=L$ in \eqref{eq:back_compact_reflecting_reflecting} is replaced by the general Wentzell boundary condition
$$
\frac12 \partial_{xx}u(t,L)+\beta_L\partial_x u(t,L)+\gamma_L u(t,L)=0.
$$
In that case one replaces the domain of $A$ by
$$
D(A):=\left\{f\in C^2([0,L]) : f'(0)=\kappa f''(0),\ \frac12 f''(L)+\beta_L f'(L)+\gamma_L f(L)=0\right\},
$$
and the compatibility conditions at $x=L$ become
$$
\frac12 \phi''(L)+\beta_L\phi'(L)+\gamma_L\phi(L)=0
$$
and
$$
\frac14 \phi^{(4)}(L)+\frac{\beta_L}{2}\phi^{(3)}(L)+\frac{\gamma_L}{2}\phi''(L)=0.
$$
The proof is unchanged.
\end{remark}

%%%%%%%%%%%%%%%%%%%%%%%%%%%%%%%%%%%%%%%%%%%
%%%%%    General Sticky Diffusion     %%%%%
%%%%%%%%%%%%%%%%%%%%%%%%%%%%%%%%%%%%%%%%%%%

\section{General Sticky Diffusion}\label{sec:generalsticky}

Now we consider a general one-dimensional sticky diffusion with drift $b(x)$ and diffusion $\sigma(x)$.
For ease of reference we recall the generator of the process is 
\begin{equation}\label{eq:Lgen}
    \mathcal{L}u = b \cdot \partial_x u + a\partial_{xx}u,
\end{equation}
with $a = \frac{\sigma^2}{2}$, and with a sticky boundary condition at $x=0$: 
\begin{equation}
\label{eqn:L_diffusion_BC}
    \left.a\partial_xu\right\vert_{x=0} = \left.\kappa\left(b\partial_xu + a\partial_{xx}u\right)\right\vert_{x=0}\, .
\end{equation}
The backward equation describing the process is
\begin{equation}
\label{eqn:diffusion_bckwd_eqn}
    \partial_t u + \mathcal{L}u = 0, \quad x \in (0,+\infty), \qquad u(T,x) = \phi(x), \quad x \in (0,+\infty),
\end{equation}
with the sticky boundary condition \eqref{eqn:L_diffusion_BC}.

%To justify boundary condition \eqref{eqn:L_diffusion_BC}, readers can refer to methodology in Section 2.6 of this paper \cite{Bou-Rabee.2020}.

\subsection{Overview of the simulation algorithm}
\label{sec:sticky_diffusion}

The scheme proceeds in the same way as for a SBM: in a boundary layer, we propose a reflected  EM jump with probability $\lambda(\cdot)$, and otherwise we jump directly to the sticky boundary.  A difference from a SBM is that the boundary layer is defined as a function of the variable diffusion $\sigma(x)$, and hence may not be a continuous interval; this introduces some technicalities into the proofs. Further, the procedure for reflecting the process must be modified, to ensure the jump probability $\lambda(\cdot)$ is indeed a valid probability. 
As reader will see later, subtleties arise in the scheme when the drift term is negative.

We generate a process $\{X_k\}_{k=1,\dots,N}$ with time step $h$. Given $X_k=x$, the boundary layer and exterior are defined as 
\begin{align}
    \text{boundary layer:}\quad  &\Omega_{b}=\{x: 0\leq \frac{x}{\sigma(x)} \leq \sqrt{3h}\},\label{eq:Omegab}\\
    \text{exterior:}\quad & \Omega_e = \{x: \frac{x}{\sigma(x)} > \sqrt{3h}\}.\nonumber
\end{align}

\begin{remark} The reader will observe that the boundary layer as defined above can be disconnected. We'll handle this by showing that, for small enough $h$, it is contained in an interval $[0,\sigma_2\sqrt{3h}]$ for some $\sigma_2>0$. We'll then do a local error analysis that is valid in this larger interval, hence it will apply to our boundary layer, which is a subset of this interval. Furthermore, the local error in the part of the exterior which is in the interval $[0,\sigma_2\sqrt{3h}]$ will be controlled by standard error analyses for reflected EM jumps. %\todan{I added this paragraph, please check}
\end{remark}

% \mhc{say: boundary layer can be disconnected. we'll deal with this by showing that for small enough $h$, it is contained in $[0,\sigma_2\sqrt{3h}$ for some $\sigma_2>0$. then we'll do a local error analysis that is valid in this larger interval, hence it will apply to our boundary layer, which is a subset. Furthermore, the exterior which is in the boundary layer will be controlled by standard error analyses for reflected EM jumps}

% \todan{I'm a bit puzzled about this formulation of a ``state-dependent boundary layer''. If we know $X_k=x$, then either it is in the boundary layer or it isn't. It seems to me that the boundary layer should be defined as 
% \[\Omega_{br} = \{x:\frac{x}{\sigma(x)} < \sqrt{3h}\}.\]
% Would it change much of the presentation if it were defined this way? 
% }
% \redan{Yes, I think your representation make it more clear, saying state-dependent is puzzling. 
% What I really want to emphasize is that, unlike the SBM case, the boundary layer is not given a priori by a fixed interval such as $[0,\sqrt{3h}]$.
% I do not think this will change the proof argument, will check}

Given a point $X_k$ in the exterior, $X_k \in \Omega_e$, we update $X_{k+1}$ using a standard Euler-Maruyama increment with reflection:%  to ensure positivity: % (this is necessary if the drift term is negative):
\begin{equation}
\label{eqn:diffusion_cont_scheme_outlayer}
    X_{k+1} = \left|X_k + \sigma(X_k)Z_k + b(X_k)h\right|.
\end{equation}
Here $\{Z_k\}$ is a sequence of \textit{iid} uniform random variables on $[-\sqrt{3h}, \sqrt{3h}]$.

When the particle is within the boundary layer, $X_k \in \Omega_b$, the scheme performs a type of reflected EM jump with probability $\lambda(X_k)$, %(which now depends on the sticky parameter $\kappa$ and the local coefficients), 
and otherwise it jumps to $0$ with probability $1-\lambda(X_k)$:
\begin{equation}
\label{eqn:diffusion_cont_scheme_inlayer}
    X_{k+1} =
    \begin{cases} 
    \Bigl| \left|X_k + \sigma(X_k)Z_k\right| + b(X_k)h \Bigr|, & \text{with probability } \lambda(X_k), \\
    0, & \text{with probability } 1-\lambda(X_k).
    \end{cases}
\end{equation}

\begin{remark}
The reflected EM jump in \eqref{eqn:diffusion_cont_scheme_inlayer} involves a nested absolute value. A more standard reflection would take a single absolute value of the entire expression. We attempted this more standard reflection first, but it led to an expression for $\lambda(x)$ that we were unable to prove is a valid probability. 
%While taking a single absolute value of the entire expression yields a new $\lambda(x)$, it is hard to check its validity and does not simplify any error analysis we are about to introduce. Both schemes are identical and exception occurs with probability $\mathcal{O}(\sqrt{h})$. 
\end{remark}

The jump probability function  is
\begin{equation}
\label{eqn:diffusion_lambda_cont_scheme}
    \lambda(x) = \frac{(\sigma^2(x)-2\kappa b(x))x^2 + 2\kappa \sigma^2(x) x + \sigma^4(x)h}{\left(\frac{\kappa\sigma(x)}{\sqrt{3h}} + \sigma^2(x) - 2\kappa b(x)\right)x^2 + \sigma^3(x) \kappa \sqrt{3h} + \sigma^4(x)h}.
\end{equation}
This expression is derived shortly via the method of undetermined coefficients. We prove in Section \ref{sec:lamprobgen} that it is a valid probability for small enough $h$.

An summary of the simulation scheme for a general sticky diffusion is given in Algorithm \ref{alg:diffusion_cont_scheme}.

%We note that \eqref{eqn:diffusion_lambda_cont_scheme} can be written in an equivalent, albeit more complex, form by dividing the numerator and denominator by an auxiliary function $M(x) = \frac{\kappa \sigma^2(x)}{\sigma^2(x)-2\kappa b(x)}$. 
%However, the form in \eqref{eqn:diffusion_lambda_cont_scheme} is preferred for implementation and analysis.

%Those assumptions makes it legal for us to do Taylor expansion and bound the coefficients of remainder terms with high-order derivatives.

\begin{algorithm}
\caption{Approximation of Sticky Diffusion $X_{T}$ starting from $x \in [0, \infty)$.}
\label{alg:diffusion_cont_scheme}
\begin{algorithmic}[1]
\State \textbf{Set} $X_0=x$.
\For{$k = 0$ \textbf{to} $N-1$}
    \State Generate $Z_{k} \sim \mathcal{U}([-\sqrt{3h}, \sqrt{3h}])$.
    \State Calculate boundary threshold $L_k = \sigma(X_k)\sqrt{3h}$.
    \If{$X_k > L_k$}
        \State Update $X_{k+1}$ using \eqref{eqn:diffusion_cont_scheme_outlayer}.
    \Else \quad ($X_k \in [0, L_k]$)
        \State Calculate $\lambda(X_k)$ using \eqref{eqn:diffusion_lambda_cont_scheme}.
        \State Generate $u \sim \mathcal{U}([0,1])$.
        \If{$u \leq \lambda(X_k)$}
            \State Update $X_{k+1} = \bigl| |X_k + \sigma(X_k)Z_k| + b(X_k)h \bigr|$.
        \Else
            \State Set $X_{k+1} = 0$.
        \EndIf
    \EndIf
\EndFor
\end{algorithmic}
\end{algorithm}

\subsection{How to choose $\lambda(x)$}

As for a SBM, we want to choose the jump probability $\lambda$ so that the transition operator of our discrete scheme is a good approximation to the generator of sticky diffusion in the boundary layer. 
By Taylor expansion of $P^h f - f$ and the generator $\mathcal L$ defined in \eqref{eq:Lgen}, we want to find $\lambda$ such that
\begin{equation}\label{eq:genapprox_diffusion}
    \left|(\partial_xf)\mathbb{E}_x(\Delta X) + \frac{1}{2}(\partial_{xx}f) \mathbb{E}_x(\Delta X^2) -hb(x)(\partial_xf) - ha(x)(\partial_{xx}f)\right| \leq \epsilon,
\end{equation}
for some sufficiently small $\epsilon$; as before we'll have $\epsilon \leq C(h^2+hx)$.
To replace $\partial_xf$ by $\partial_{xx}f$,
we Taylor-expand both sides of the sticky boundary condition \eqref{eqn:L_diffusion_BC}:
\begin{align*}
    \sigma^2(0) \partial_xf(t,0) &= \sigma^2(x)\partial_x f(t,x) - x\sigma^2(x)\partial_{xx}f(t,x) + \mathcal{O}(h+x),\\
     2\kappa b(0)\partial_x f(t,0) + \kappa \sigma^2(0)\partial_{xx}f(t,0)
    &= 2\kappa b(x)\partial_x f(t,x) +\\& \qquad \left(- 2x\kappa b(x) + \kappa \sigma^2(x)\right)\partial_{xx}f(t,x) + \mathcal{O}(h+x).
\end{align*}
Substituting these expansions back into the boundary condition \eqref{eqn:L_diffusion_BC}, and dropping the error terms, gives
\begin{equation}
\label{eq:taylorsticky_diffusion}
\begin{split}
    \partial_x f(t,x) &= \left(M(x) + x \right)\partial_{xx}f(t,x) + \mathcal{O}(h+x),
\end{split}
\end{equation}
where
$$M(x) = \frac{\kappa \sigma^2(x)}{\sigma^2(x)-2\kappa b(x)}.$$
%$\mathcal{O}(h+x)$ terms above represent remainder terms in Taylor expansion that can be upper bounded by $C(h+x)$ where $C>0$ is a constant depending on derivative bounds of $b(x)$ and $\sigma(x)$.
Approximation \eqref{eq:taylorsticky_diffusion} assumes $b(0) \neq \frac{\sigma^2(0)}{2\kappa}$ for $h$ sufficiently small  (if we assume continuity of $b(x)$ and $\sigma(x)$).
We will revisit the special case where this doesn't hold when we do the local error analysis. 
Substituting \eqref{eq:taylorsticky_diffusion} into \eqref{eq:genapprox_diffusion}, dropping the remainder terms gives
\begin{equation}\label{eq:genapprox2_diffusion}
    (M(x)+x)\left(\mathbb{E}_x(\Delta X)-hb(x)\right) + \half \mathbb{E}_x(\Delta X^2) - ha(x) = \epsilon.
\end{equation}
%Error term $\epsilon$ in \eqref{eq:genapprox2_diffusion} contains remainder from \eqref{eq:taylorsticky_diffusion} that can be controlled by a multiple of $hx+h^2$ if we can show $|\mathbb{E}_x(\Delta X)| \leq Ch$ (already proved for SBM), 
%$$\left(\mathbb{E}_x(\Delta X)-hb(x)\right) \cdot \mathcal{O}(h+x) = \mathcal{O}(hx+h^2).$$
%In other words, we expect $\lambda$ solved by method undetermined coefficients and computing  $\mathbb{E}_x(\Delta X)$, $\mathbb{E}_x(\Delta X^2)$ gives $\epsilon \leq C(hx+h^2)$.
%Similar to our remark for SBM, readers may want to remove term $x\partial_{xx}f$ in \eqref{eq:taylorsticky_diffusion}, but removal will not simplify expression of $\lambda$ or any calculation in proofs we are about to introduce. 
We may now calculate $\mathbb{E}_x(\Delta X)$, $\mathbb{E}_x(\Delta X^2)$, as functions of $\lambda$, and then solve \eqref{eq:genapprox2_diffusion} for $\lambda$ (after setting $\epsilon=0$). %This moment computation is more complicated than for an SBM, because the expressions for the moments depend on the sign of  $x - \sigma(x)\sqrt{3h}$. 
%Hence, we decided to concentrate on the case $x - \sigma(x)\sqrt{3h} < 0$ and to solve for $\lambda$ in this case only. 
%\mhc{difference is absorbed in big-O term} \todan{i added last sentence -- is it correct?}
%\todan{I'm still puzzled -- I thought we only needed to compute the moments, for $x$ in the boundary layer. Any $x$ such that $x - \sigma(x)\sqrt{3h} > 0$ is not in the boundary layer. So why do we care about this case? }\redan{I just delete my previous content, I misread the sentence you add, and it is correct.}
%The complementary case $x-\sigma(x)\sqrt{3h}>0$ corresponds to the exterior update, where the scheme does not use $\lambda$; it is handled later in the local error analysis.
For $x-\sigma(x)\sqrt{3h}<0$, readers can verify that
\begin{align}
\mathbb{E}_x(\Delta X_{k}) &= \frac{\lambda x^2}{2\sqrt{3h} \cdot \sigma(x)} + \frac{\sigma(x) \lambda \sqrt{3h}}{2} + \lambda b(x)h - x + \mathcal{O}(hx + h^2)\,, \nonumber\\
\mathbb{E}_x(\Delta X_k^2) &= -\frac{\lambda}{\sigma(x)\sqrt{3h}}\cdot x^3 + (1+\lambda)x^2 - \lambda\sigma(x)\sqrt{3h}\cdot x + \lambda\sigma^2(x)\cdot h + \mathcal{O}(hx + h^2)\,.  \label{eq:DelX_diffusion}
\end{align}
Substituting into \eqref{eq:genapprox2_diffusion}, and given $b(0) \neq \frac{\sigma^2(0)}{2\kappa}$, 
%recall $M(x) = \frac{\kappa \sigma^2(x)}{\sigma^2(x)-2\kappa b(x)}$, 
we have
$$\lambda\cdot\left( \left( \frac{M}{\sigma\sqrt{3h}} + 1 \right)x^2 + M(x)\sigma\sqrt{3h} + \sigma^2 h + 2bMh \right)  = \sigma^2 h + x^2 + 2Mx + 2bMh.$$
Solving for $\lambda$ gives \eqref{eqn:diffusion_lambda_cont_scheme}.

\subsection{$\lambda$ is a probability}\label{sec:lamprobgen}

To show that $\lambda$ defined in \eqref{eqn:diffusion_lambda_cont_scheme} is a valid probability for small enough $h$, we must impose assumptions on the drift and diffusion coefficients. The first is a local regularity assumption, which is used for Taylor expansions and coefficient bounds once the boundary layer has been localized near $0$:

\begin{assumption}[Local coefficient regularity near the sticky boundary]
\label{assump:local_coeff_boundary_regularity}
There exists $r_0>0$ such that $b,\sigma\in C^2([0,r_0])$, $\sigma(0)>0$, and
$$
C_{\mathrm{loc}}
:=
\max_{0\le j\le 2}
\left(
\sup_{x\in[0,r_0]} |\partial_x^j b(x)|
+
\sup_{x\in[0,r_0]} |\partial_x^j \sigma(x)|
\right)
<\infty.
$$
After decreasing $r_0$ if necessary, we also assume that
$$
\sigma(x)\ge \sigma_* >0,\qquad x\in[0,r_0].
$$
\end{assumption}

The second is a linear-growth assumption, which has two roles: it first ensures that the state-dependent boundary layer cannot extend far from the sticky boundary when $h$ is small, and it later provides the standard moment control needed for the Euler--Maruyama part outside the boundary layer.

\begin{assumption}[Linear growth of the coefficients]
\label{assump:linear_growth_coefficients}
There exists $L>0$ such that
$$
|b(x)|+|\sigma(x)|\le L(1+x),\qquad x\ge 0.
$$
\end{assumption}

%\redan{The assumption is subtle, I do not assume compact domain, only sticky boundary at the origin. Assumption~\ref{assump:local_coeff_boundary_regularity} is used to analyze local error near the sticky boundary, Assumption \ref{assump:linear_growth_coefficients} is used to control the local error outside the boundary layer (see Lemma~\ref{lemma:fourth_moment_bound_diffusion_scheme} and the local error outside the boundary layer is of order $h^2(1+x^4)$, not $h^2$). However, to prove local error convergence, we now assume uniform boundedness of the solution of PDE. However, if we use other PDE literature which usually assumes compact domain, Assmuption \ref{assump:linear_growth_coefficients} is unnecessary by uniform boundedness on drift and diffusion.}

%\begin{assumption}
%\label{assump:2nd_drift_derivative_bound}
%    Drift $b(x)$ has spatial derivative up to order 2 where where $\left\|\partial_{xx} b\right\|_{\infty} = \sup_{(t,x)\in [0,+\infty)^2}\|\partial_{xx}b(x)\| < \infty$.
%\end{assumption}
%\begin{assumption}
%\label{assump:2nd_diffusion_derivative_bound}
%    Diffusion $\sigma(x) > 0$ has spatial derivative up to order 2 where where $\left\|\partial_{xx} \sigma\right\|_{\infty} = \sup_{(t,x)\in [0,+\infty)^2}\|\partial_{xx}\sigma(x)\| < \infty$.
%\end{assumption}

The following lemma shows that, once $h$ is small, every point in the boundary layer lies in a compact interval on which $b(x)\in [b_1,b_2]$ and $\sigma(x)\in [\sigma_1,\sigma_2]$.
In particular, $\sigma_1$ and $\sigma_2$ are simply uniform lower and upper bounds for the diffusion coefficient on the region when $x \in \Omega_{br}$.

\begin{lemma}
\label{lemma:bound_on_drift_diffusion}
    Under Assumptions~\ref{assump:local_coeff_boundary_regularity} and \ref{assump:linear_growth_coefficients}, there exists $\delta>0$ and constants $b_1,b_2,\sigma_1,\sigma_2$, independent of $x$ and $h$, with $\sigma_1>0$, such that, for every $h<\delta$,
    $$
    b(x)\in [b_1,b_2],\qquad \sigma(x)\in[\sigma_1,\sigma_2],
    $$
    whenever
    $$
    x-\sigma(x)\sqrt{3h}<0
    \qquad\text{or}\qquad
    x-\sigma(x)\sqrt{3h}+b(x)h<0.
    $$
\end{lemma}
\begin{proof}
    See Supplement.
\end{proof}
As a consequence, whenever $h<\delta$, for any $x\in \Omega_b$, we may apply the fixed boundary-layer bounds
$$
b(x)\in[b_1,b_2],\qquad \sigma(x)\in[\sigma_1,\sigma_2].
$$
Moreover,
$$
x<\sigma(x)\sqrt{3h}\le \sigma_2\sqrt{3h},
$$
so every point in the boundary layer is in fact $O(\sqrt h)$ close to the sticky boundary, exactly as in the SBM case where the boundary layer is explicit.
%\mhc{REF to lemma, does this appear later?}

%\redan{The point is not only that the boundary layer shrinks with $h$, but that Lemma~\ref{lemma:bound_on_drift_diffusion} lets us replace the implicit boundary-layer condition by the explicit bound $x\le \sigma_2\sqrt{3h}$. This is the diffusion analogue of the SBM fact that points in the boundary layer are automatically $O(\sqrt h)$ close to the sticky boundary. We use this repeatedly afterward to apply local coefficient bounds and to estimate Taylor remainders by powers of $h$.}

\begin{lemma}
\label{lemma:bound_on_lambda_diffusion}
    Under Assumptions~\ref{assump:local_coeff_boundary_regularity} and \ref{assump:linear_growth_coefficients}, there exists $\delta>0$ such that, for every $h<\delta$ and every $x\in \Omega_b$
    the jump probability \eqref{eqn:diffusion_lambda_cont_scheme} satisfies
    $$
    \lambda(x)\in[0,1].
    $$
\end{lemma}
\begin{proof}
    See Supplement.    
\end{proof}

\subsection{Primary convergence results}
\label{sec:Primary convergence results}
Now we present our main convergence results. 
We wish to show weak convergence of the numerical scheme with order $h$ and $X_0=x_0$, i.e. that
\begin{equation}
    |\E \phi(X_N) - \E\phi(Y_T)| \leq Ch
\end{equation}
for some constant $C$, where $\phi$ is from a suitable class of functions.
Similar to our previous discussion in SBM, we let $u(t,x)= \E[\phi(Y_T) | Y_t=x]$ be a function solving the backward equation for the sticky diffusion, \eqref{eqn:diffusion_bckwd_eqn}.
Writing $t_k=kh$, $u_k=u(t_k,X_k)$, we have a similar decomposition of the error as in \eqref{eq:err_decompose}:
\begin{align}
|\E \phi(X_N) - \E\phi(Y_T)| %&= \left|\E\left(\sum_{k=0}^{N-1}u_{k+1}-u_k\right)\right|\\ 
%&= \left|\mathbb{E}\left(\sum_{k=0}^{N-1} \mathbb{E}\left(u_{k+1}-u_k|X_k\right)\right)\right|\nonumber\\
&\leq \mathbb{E}\left(\sum_{k=0}^{N-1} \left|\mathbb{E}\left(u_{k+1}-u_k|X_k\right)\right|\cdot \mathbf{1}_{X_k \in \Omega_{br}}\right) \nonumber\\&\quad  
        + \mathbb{E}\left(\sum_{k=0}^{N-1} \left|\mathbb{E}\left(u_{k+1}-u_k|X_k\right)\right|\cdot \mathbf{1}_{X_k \notin \Omega_{br}}\right).
        \label{eq:err_decompose_diffusion}
\end{align}
Here $\Omega_{br} = \Omega_{b} \cup \Omega_{r}$ with
\begin{align}
    %\Omega_{b}&:= \{x:\;x \leq \sigma(x)\sqrt{3h}\},  \nonumber\\
    \Omega_{r} &:= \{x > \sigma(x)\sqrt{3h};\; x-\sigma(x)\sqrt{3h}+b(x)h\leq 0\}.\label{eq:Omegabi}
\end{align}
We include $\Omega_{r}$ along with $\Omega_b$ because both regions have the same form of local error, and we refer to $\Omega_{br}$ henceforth as the boundary layer.  
%\mhc{both boundary layer terms are treated together because they have the same form of local error}

To present main results in the error analysis, we need the same assumptions we have for SBM and Assumptions~\ref{assump:local_coeff_boundary_regularity}, \ref{assump:linear_growth_coefficients}.
%The error outside the boundary layer $\Omega_e$ is trivial to analyze because the scheme does standard EM jumps.
Suppose that $h$ is always chosen such that $N := T/h \in \mathbb{Z}$.
%\begin{theorem}
%\label{thm:local_err_cont_scheme_diffusion_general}
%    (Local Error of $\{X_k\}_{k\in\mathbb{N}}$ in algorithm \ref{alg:diffusion_cont_scheme})\\
%    Under Assumptions \ref{assump:4th_spatial_derivative_bound},  \ref{assump:2nd_time_derivative_bound}, \ref{assump:2nd_drift_derivative_bound}, \ref{assump:2nd_diffusion_derivative_bound}, chose $\epsilon = \mathcal{O}(hx + h^2)$ in \eqref{eq:genapprox2_diffusion}, there exists a constant $M$ depending on $\left\| u^{(3)} \right\|_{\infty}$ $\left\| u^{(4)} \right\|_{\infty}$, $\left\|\partial_{tt} u\right\|_{\infty}$, $\|\partial_{xx}b\|_{\infty}$, and $\|\partial_{xx}\sigma\|_{\infty}$ such that 
%    \begin{equation}
%    \label{eqn:local_err_cont_scheme}
%        \left|\mathbb{E}\left(u_{k+1}-u_k|X_k=x\right)\right| \leq \begin{cases}
%            M( h^2 + hx), & \forall x\in \Omega_{br}\\
%            Mh^2 & \forall x \in \Omega_e.
%        \end{cases}
%    \end{equation}
%\end{theorem}
\begin{theorem}
\label{thm:local_err_cont_scheme_diffusion_general}
    (Local Error of $\{X_k\}_{k\in\mathbb{N}}$ in Algorithm \ref{alg:diffusion_cont_scheme})\\
    Under Assumptions \ref{assump:4th_spatial_derivative_bound}, \ref{assump:2nd_time_derivative_bound}, \ref{assump:local_coeff_boundary_regularity}, and \ref{assump:linear_growth_coefficients}, choose $\epsilon=\mathcal{O}(hx+h^2)$ in \eqref{eq:genapprox2_diffusion}. Then there exists a constant $M>0$, independent of $x$ and $h$, depending on the relevant derivative bounds of $u$, the local boundary constants for $b,\sigma$, the linear-growth constant of $b,\sigma$, and $\kappa$, such that for all sufficiently small $h$,
    \begin{equation}
    \label{eqn:local_err_cont_scheme}
        \left|\mathbb{E}\left(u_{k+1}-u_k\mid X_k=x\right)\right| \leq
        \begin{cases}
            M(h^2+hx), & x\in \Omega_{br},\\
            Mh^2(1+x^4), & x\in \Omega_e.
        \end{cases}
    \end{equation}
\end{theorem}

Hence, as for a SBM, the local error in the boundary layer is $O(h^{3/2})$, while that in the exterior is $O(h^2)$. Also as for a SBM,
we will show the number of steps in the boundary layer is $O(h^{-1/2})$, leading to a global error of $O(h)$. 
We have $Mh^2(1+x^4)$ because we analyze the local error outside the boundary layer by Taylor expansion up to fourth order, which contains products of drift and diffusion of order 4 and we apply the linear growth condition. 
Analogous to Corollary \ref{corollary:local_err} and Lemma~\ref{lemma:key}, we have
\begin{corollary} Under the assumptions of Theorem \ref{thm:local_err_cont_scheme_diffusion_general}, 
\label{corollary:local_err_diffusion}
    \[\mathbb{E}\left(\sum_{k=0}^{N-1} \left|\mathbb{E}\left(u_{k+1}-u_k|X_k\right)\right|\cdot \mathbf{1}_{X_k \in \Omega_{br}}\right) \leq M\left(h +  h \cdot \mathbb{E}\left(\sum_{k=0}^{N-1} X_k \cdot \mathbf{1}_{X_k \in \Omega_{br}} \right)\right).\]
\end{corollary}
\begin{lemma}
\label{lemma:key_diffusion_general}
    For sufficiently small $h$, there exists a positive constant $C$, independent of $h$ and initial condition $x_0$, such that
   \[\mathbb{E}\left(\sum_{k=0}^{N-1} X_k \cdot \mathbf{1}_{X_k \in \Omega_{br}} \right) < C.\]
\end{lemma}
The proof will recycle most of the construction we used to prove Lemma~\ref{lemma:key}. If  $\sigma(0) = 1$, we can show the jump probability $\lambda$ and the transition operator $P^h$ for SBM and the sticky diffusion are sufficiently close.

The following lemma helps us analyze the error coming from the exterior where the scheme does a standard EM jump where we will transform the drift and diffusion into $(1+X_k)^4$ from the linear growth assumption.
\begin{lemma}
\label{lemma:fourth_moment_bound_diffusion_scheme}
Under assumptions~\ref{assump:local_coeff_boundary_regularity}, \ref{assump:linear_growth_coefficients}, and suppose that $X_0=x_0$ deterministically. Then, for $T=Nh$, there exists a constant $C_T>0$, independent of $h$ and $x_0$, such that
$$
\sup_{0\le k\le N}\E(1+X_k^4)\le C_T(1+x_0^4).
$$
\end{lemma}
%\begin{lemma}
%\label{lemma:fourth_moment_bound_diffusion_scheme}
%Under assumptions~\ref{assump:local_coeff_boundary_regularity}, \ref{assump:linear_growth_coefficients} and assuming $\mathbb{E}(X_0^4)<\infty$, then for $T=Nh$, there exists a constant $C_T>0$, independent of $h$, such that
%$$
%\sup_{0\le k\le N}\mathbb{E}(1+X_k^4)\le C_T.
%$$
%\end{lemma}

Applying these Lemmas to \eqref{eq:err_decompose_diffusion} gives the global error bound for the sticky diffusion, our main convergence theorem. 

\begin{theorem}
\label{thm:global_err_cont_scheme_diffusion}
    (Global Error of $\{X_k\}_{k\in\mathbb{N}}$ in Algorithm \ref{alg:diffusion_cont_scheme})\\
    Let $\phi:[0,\infty)\to\mathbb{R}$ be such that the associated backward solution
    $$
    u(t,x):=\E\!\left[\phi(Y_T)\mid Y_t=x\right]
    $$
    satisfies Assumptions~\ref{assump:4th_spatial_derivative_bound},
    \ref{assump:2nd_time_derivative_bound},
    \ref{assump:local_coeff_boundary_regularity}, and
    \ref{assump:linear_growth_coefficients}. Assume moreover that $X_0=Y_0=x_0$ deterministically. 
    Then there exist constants $\delta>0$ and $C>0$,
    independent of $h$ and $x_0$, such that for all $h<\delta$,
    \begin{equation}
    \label{eqn:global_err_cont_scheme}
        \left|\E\bigl(\phi(X_N)-\phi(Y_T)\bigr)\right|
        \le C(1+x_0^4)h.
    \end{equation}
\end{theorem}
\begin{proof}
    This follows from the error decomposition \eqref{eq:err_decompose_diffusion}. By Theorem~\ref{thm:local_err_cont_scheme_diffusion_general}, there exists a constant $M>0$, independent of $x$ and $h$, such that
    $$
    \left|\E\left(u_{k+1}-u_k\mid X_k=x\right)\right|
    \le
    \begin{cases}
    M(h^2+hx), & x\in\Omega_{br},\\
    Mh^2(1+x^4), & x\in\Omega_e.
    \end{cases}
    $$
    Applying this bound to \eqref{eq:err_decompose_diffusion}, we obtain
    \begin{equation*}
    \begin{split}
    \left|\E\phi(X_N)-\E\phi(Y_T)\right|
    &\le
    M\E\left(\sum_{k=0}^{N-1}(h^2+hX_k)\mathbf{1}_{\{X_k\in\Omega_{br}\}}\right)\\
    &\quad+
    Mh^2\sum_{k=0}^{N-1}\E\left((1+X_k^4)\mathbf{1}_{\{X_k\in\Omega_e\}}\right).
    \end{split}
    \end{equation*}
    For the boundary-layer contribution, Lemma~\ref{lemma:key_diffusion_general} gives
    $$
    \E\left(\sum_{k=0}^{N-1}X_k\mathbf{1}_{\{X_k\in\Omega_{br}\}}\right)\le C_b,
    $$
    where $C_b$ is independent of $h$ and $x_0$. Since $N=T/h$, we also have
    $$
    \E\left(\sum_{k=0}^{N-1}h^2\mathbf{1}_{\{X_k\in\Omega_{br}\}}\right)
    \le Nh^2
    =Th.
    $$
    Therefore,
    $$
    M\E\left(\sum_{k=0}^{N-1}(h^2+hX_k)\mathbf{1}_{\{X_k\in\Omega_{br}\}}\right)
    \le (MT+MC_b)h.
    $$
    For the exterior contribution, Lemma~\ref{lemma:fourth_moment_bound_diffusion_scheme} yields
    $$
    \sup_{0\le k\le N}\E(1+X_k^4)\le C_T(1+x_0^4),
    $$
    where $C_T$ is independent of $h$ and $x_0$. Hence
    \begin{equation*}
    \begin{split}
    Mh^2\sum_{k=0}^{N-1}\E\left((1+X_k^4)\mathbf{1}_{\{X_k\in\Omega_e\}}\right)
    &\le
    Mh^2\sum_{k=0}^{N-1}\E(1+X_k^4)\\
    &\le
    Mh^2N\, C_T(1+x_0^4)\\
    &=
    MTC_T(1+x_0^4)h.
    \end{split}
    \end{equation*}

    Combining the two estimates gives
    $$
    \left|\E\phi(X_N)-\E\phi(Y_T)\right|
    \le
    (MT+MC_b)h + MTC_T(1+x_0^4)h.
    $$
    Since $1\le 1+x_0^4$, we may absorb the first term into the second and obtain
    $$
    \left|\E\phi(X_N)-\E\phi(Y_T)\right|
    \le
    C(1+x_0^4)h
    $$
    for some constant $C>0$ independent of $h$ and $x_0$.
\end{proof}

%We test the convergence results by using our algorithm to simulate a sticky diffusion with a reflecting boundary at $x=L$ and estimate $\mathbb{E}(X_T|X_0=0)$.

\medskip

\begin{figure}
    \centering
    \includegraphics[width=0.85\linewidth]{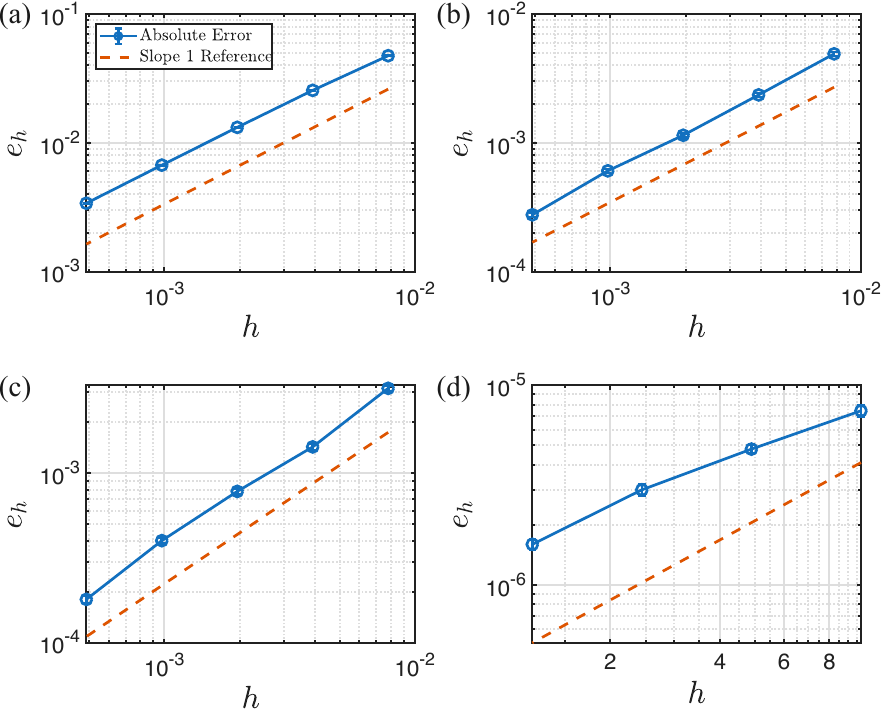}
    \caption{Weak error 
    $e_h = |\mathbb{E}[X_T\mid x_0=0]-u(T,0)|$ of Algorithm
    \ref{alg:diffusion_cont_scheme} for different choices of drift, diffusion, and sticky parameter: (a) $b(x)=2+\sin(x)$, $\sigma(x)=\frac{1}{2}+x$, $\kappa=4$; (b) $b(x)=2+\sin(x)$, $\sigma(x)=2+x$,  $\kappa=1$; (c) $b(x)=-1+\sin(x)$, $\sigma(x)=2+x$, $\kappa=1$; (d) $b(x)=-2+\sin(x)$, $\sigma(x)=\frac{1}{2}+x$,  $\kappa=4$. 
    All experiments use $L=4$, $T=1$, and $x_0=0$.
    Error bars (commensurate with marker sizes) are $\pm 1$ empirical standard deviation.
    The respective least-squares fitted slopes are (a) $0.954162$ (b) $1.026675$, (c) $1.010317$, (d) $0.907653$ (fitted only through the last 2 data points in this case). }
    \label{fig:SDsim}
\end{figure}

\subsection{Numerical experiments}
Figure \ref{fig:SDsim} empirically validates our convergence results using different choices for $b,\sigma$ as given in the caption.  As for a SBM, we impose a reflecting boundary condition at $x=L$, in order to compare our stochastic estimates with a ``ground truth'' constructed by solving the backward equation numerically using finite difference with a very small time step. Also as before we choose $\phi(x)=x$, which does not satisfy the boundary condition compatibility conditions that are typically required in PDE proofs of smoothness. Figure \ref{fig:SDsim} (a,b,c) used timesteps $h\in\{1/128,1/256,1/512,1/1024,1/2048\}$, and a  number of realizations which started at $5\times 10^7$ and was quadrupled each time $h$ was halved. Figure \ref{fig:SDsim} (d) used $h\in\{1/1024,1/2048,1/4096,1/8192\}$, with initial number of realizations equal to $10^9$.

All cases approach $O(h)$ convergence, though it is notable that case (d) with strongly negative drift and weaker diffusion approaches it more slowly, requiring must smaller timesteps to approach a slope of 1. The error however in this case is several orders of magnitude smaller than for the others.

%Besides, we need to address the following two questions to ensure that our scheme for sticky diffusion is well-defined and handles the subtlety in the boundary layer caused by negative drift:
%\begin{enumerate}
%    \item Is $\lambda(\cdot)$ we define in \eqref{eqn:diffusion_lambda_cont_scheme} stays in $[0,1]$ so that the scheme itself is well-defined as we have checked in Lemma \ref{lemma:bound_on_lambda_SBM}? 
%    \item Do we have the same $3/2$ order of local error in the scheme within the boundary layer as we have shown for sticky brownian motion scheme?
%    In particular, there are two situation that make local error analysis subtle.
%    The first situation is similar to what we have encountered in SBM scheme where the particle has probability $1-\lambda(X_k)$ to get absorbed to the sticky boundary.
%    The second situation may be caused by negative drift $b(X_k)$ where our scheme does a reflected EM jump when $x - \sigma(x)\sqrt{3h} + b(x)h< 0$. 
%    It is non-trivial to show this situation gives desired order of local accuracy.
%\end{enumerate}
%The next part will focus on resolving those questions.

\subsection{Key steps in the error analysis}

We now show the key steps in performing a local error analysis, Theorem \ref{thm:local_err_cont_scheme_diffusion_general}, and in proving Lemma \ref{lemma:key_diffusion_general}. 
Compared to an SBM, an additional challenge arises when we have negative drift $b(X_k)$, where our scheme does a reflected EM jump when $x - \sigma(x)\sqrt{3h} + b(x)h< 0$. 
It is non-trivial to show this situation gives desired order of local accuracy.
Compared to the SBM case, the challenge is that the near-boundary region is no longer exactly $[0,\sqrt{3h}]$. The variable diffusion changes the boundary threshold, and negative drift can create an additional thin reflection region. Thus, for the number-of-steps lemma, we first need to show that this enlarged region is still contained in an $O(\sqrt h)$ neighborhood of the boundary, and then compare the resulting transition operator with that for an SBM.

This section collects the intermediate statements behind the convergence results in Section~\ref{sec:Primary convergence results}. All the proofs use standard techniques and are found in the Supplement.

\subsubsection{Local error: Theorem \ref{thm:local_err_cont_scheme_diffusion_general}}

%This subsection collects the intermediate statements behind the convergence results in Section~\ref{sec:Primary convergence results}.
The role of this section is the same as that of Section~\ref{sec:Convergence proofs_SBM} for sticky Brownian motion: we first isolate the local error estimate in the boundary layer, then record the occupation bound needed in the global error argument, and finally return to the weak error decomposition \eqref{eq:err_decompose_diffusion}.
Unlike in Section~\ref{sec:Convergence proofs_SBM}, the error in $\Omega_e$ is less straightforward because we do not want to impose uniform boundedness assumptions on the drift and diffusion. Instead, we assume linear growth and local regularity in the boundary layer.\footnote{If we can make assumptions like compact domain, making drift and diffusion uniformly bounded, the $\Omega_e$ local error becomes  straightforward to show.} The Euler-Maruyama jump and its local weak error is of order $h^2$. 

We start by concentrating on $\Omega_{br}$. Recall this was divided into two separate pieces $\Omega_{b},\Omega_{r}$ defined in \eqref{eq:Omegab},\eqref{eq:Omegabi}.
There are three situations to treat separately:
\begin{enumerate}
    \item $x\in \Omega_{b}$ and $b(0)\neq \frac{\sigma^2(0)}{2\kappa}$;
    \item $x\in \Omega_{r}$, which is the additional reflection strip created by negative drift;
    \item $x\in \Omega_{b}$ and $b(0)= \frac{\sigma^2(0)}{2\kappa}$ (notice $\Omega_{br} = \Omega_{b}$ and $\Omega_{r} = \emptyset$ because$ b(x)>0$ for small $h$).
\end{enumerate}
All lemmas and theorems below impose Assumptions \ref{assump:local_coeff_boundary_regularity}-\ref{assump:linear_growth_coefficients} and the proofs can be found in the Supplements.

%The following lemma shows that, once $h$ is small, every point in the boundary layer lies in a compact interval on which $b(x)\in [b_1,b_2]$ and $\sigma(x)\in [\sigma_1,\sigma_2]$.
%In particular, $\sigma_1$ and $\sigma_2$ are simply uniform lower and upper bounds for the diffusion coefficient on the region when $x \in \Omega_{br}$.
%\begin{lemma}
%\label{lemma:bound_on_drift_diffusion}
%    There exists $\delta>0;\; b_1, b_2, \sigma_1, \sigma_2 \in \mathbb{R}$ such that $\forall h < \delta$, $b(x) \in [b_1, b_2]$ and $\sigma(x) \in [\sigma_1, \sigma_2]$ whenever $x - \sigma(x)\sqrt{3h} < 0$  or  $x - \sigma(x)\sqrt{3h} + b(x)h< 0$.
    %%In particular, denote $\sigma_0 = \sigma(0)$, $\sigma_2$ can be expressed as $\sigma_2 := \sigma_0 + Kh$ for some constant $K>0$.
%\end{lemma}
%\begin{proof}
%    See Supplement.
%\end{proof}
%The following lemma shows that the scheme we propose is well-defined.
%%It may seem strange that we choose $x\in [0,\sigma_2\sqrt{3h}]$ in following lemmas, but $B$ we define above is in fact its subset (\textit{i.e. $\sigma_0\sqrt{3h}+Kh \leq \sigma_2 \sqrt{3h}$}).
%\begin{lemma}
%\label{lemma:bound_on_lambda_diffusion}
%    There exists $\delta>0$ such that $\forall h < \delta$, $\forall x$ such that $x - \sigma(x)\sqrt{3h} < 0,$
%    $$\lambda(x) \in [0,1], \;$$
%\end{lemma}
%\begin{proof}
%    See Supplement.    
%\end{proof}

%Notice that $\delta > 0$ we choose is independent of $x$, for all remaining lemmas, we always make $\delta$ we choose independent of $x$.
Similar to the proofs for SBM, we proceed to show that $\lambda(x)$ is upper bounded by a multiple of $\frac{x}{\sqrt{h}}+\sqrt{h}$. This not only tells us the jump probability to the sticky boundary is higher when the particle position $x$ is close to the sticky boundary, but it also enables us to conclude that $\lambda(x)\cdot h^{3/2} = \mathcal{O}(hx+h^2)$ in the local error analysis.

\begin{lemma}
\label{lemma:UB_lambda_diffusion}
    There exists constant $C, \delta > 0$, depending on $\sigma$ and $b$, such that $\forall h < \delta$ and $\forall x\in \Omega_{b}$, and for $\lambda$ defined in \eqref{eqn:diffusion_lambda_cont_scheme}, 
    $$0 \leq \lambda(x) \leq S_1\cdot \left(\frac{x}{\sqrt{h}} + \sqrt{h} \right).$$
\end{lemma}
%\begin{proof}
%    See Supplement.
%\end{proof}

Now we would like to prove a couple of lemmas to prepare for showing that our diffusion scheme has local error of order $3/2$.
We begin with the generic case $x\in \Omega_{b}$ and $b(0)\neq \frac{\sigma^2(0)}{2\kappa}$.
The first lemma is the analogue of Lemma \ref{lemma:sticky_bc@x} from the SBM proof.
It rewrites the sticky boundary condition in a form that lets us replace $\partial_xu$ by $\partial_{xx}u$ up to a remainder whose size is exactly what the local error proof can absorb.

\begin{lemma}
\label{lemma:sticky_bc@x_diffusion}
    Let $u$ be the solution of \eqref{eqn:diffusion_bckwd_eqn} satisfying corresponding boundary condition \eqref{eqn:L_diffusion_BC} and assume that $b(0) \neq \frac{\sigma^2(0)}{2\kappa}$.
    Then, there exists $\delta>0$ such that for all $h<\delta$ and $x\in \Omega_{b}$,
    \begin{equation}
    \label{eqn:diffusion_BC_consequence}
    \begin{split}
        \partial_x u(t,x) &= \left(x+M(x) \right)\partial_{xx}u(t,x) + Re(x)
    \end{split}
    \end{equation}
    where
    $$
    |Re(x)|\leq C(h+x)=\mathcal{O}(h+x),
    $$
    and $C$ depends on bounds of $\partial_xu,\partial_{xx}u,u^{(3)}$ and on the local bounds of $b,\sigma$ and their derivatives up to order $2$ on the boundary neighborhood $[0,r_0]$, but is independent of $x$ and $h$.
\end{lemma}

The general idea of the proof is to Taylor expand both sides of the boundary condition \eqref{eqn:L_diffusion_BC} and bound remainder terms by regularity assumptions.

%\begin{proof}
%    See Supplement. 
%    The general idea is to Taylor expand both sides of the boundary condition \eqref{eqn:L_diffusion_BC} and bound remainder terms by regularity assumptions.
%\end{proof}

The following two lemmas record the first and second conditional moments of the increment in $\Omega_{b}$.
They are the diffusion analogues of the moment identities for SBM.
The first and second moments both have an extra remainder term when the drift is negative because of the nested absolute value; we apply Lemma \ref{lemma:UB_lambda_diffusion} to show that this term is  of order $hx+h^2$.

\begin{lemma}
\label{lemma:1st_mom_diffusion}
There exists $\delta>0$ such that $\forall h < \delta$ and $\forall x\in \Omega_{b}$, we have, 
\begin{equation}
    \label{eqn:first_mom_negative_drifted_diffusion_approx}
    \begin{split}
        \mathbb{E}_x(\Delta X_{k}) &= \frac{\lambda(x)x^2}{2\sqrt{3h} \cdot \sigma(x)} + \frac{\sigma(x) \lambda(x)\sqrt{3h}}{2} + \lambda(x)b(x)h - x + \mathcal{O}(hx + h^2).\\
    \end{split}
    \end{equation}
\end{lemma}
%\begin{proof}
%    See Supplement.
%\end{proof}

\begin{lemma}
\label{lemma:2nd_mom_order_negative_drifted_diffusion}
There exists $\delta>0$ and $C > 0$ such that $\forall h < \delta$ and $\forall x\in \Omega_{b}$, we have, 
\begin{equation*}
    \mathbb{E}_x(\Delta X_k^2) = -\frac{\lambda(x)}{\sigma(x)\sqrt{3h}}\cdot x^3 + (1+\lambda(x))x^2 - \lambda(x)\sigma(x)\sqrt{3h}\cdot x + \lambda(x)\sigma^2(x)\cdot h + R_2(x),
\end{equation*}
    where
    $$|R_2(x)| \leq C\cdot (hx + h^2) = \mathcal{O}(hx + h^2).$$
\end{lemma}

%\begin{proof}
%    See Supplement.
%\end{proof}

The next lemma is the exact analogue of the moment identity used in Lemma~\ref{lem:moments}.
Its purpose is to isolate the leading-order cancellation built into the choice of $\lambda(x)$.
We apply the computation done in Lemma \ref{lemma:1st_mom_diffusion}-\ref{lemma:2nd_mom_order_negative_drifted_diffusion}. They produce the remainder term in the following lemma.

\begin{lemma}
\label{lemma:1st_2nd_mom_combination_generator_diffusion}
Assume that $b(0) \neq \frac{\sigma^2(0)}{2\kappa}$, there exist $\delta>0$ such that $\forall h < \delta$ and $\forall x\in \Omega_{b}$
    \begin{equation}
    \label{eqn:1st_2nd_mom_goal_cont_scheme_drifted}
        2\left(\mathbb{E}_x(\Delta X_k) - hb(x)\right) \cdot (x + M(x)) 
        + \mathbb{E}_x(\Delta X_k^2) = \sigma^2(x)h + R_3(x) 
    \end{equation}
    where $|R_3(x)| \leq K_3 \cdot (hx + h^2) = \mathcal{O}(hx + h^2)$ for some $K_3$ independent of $h$ but  depending on the boundary-layer constants and the nondegeneracy constant from $b(0)\neq\sigma^2(0)/(2\kappa)$.
    %\todan{what does this latter statement mean?}
    %\redan{I mean the positive lower bound on $|\sigma^2(x)-2\kappa b(x)|$ in a sufficiently small neighborhood of $0$. This is what controls the denominator in $M(x)=\kappa\sigma^2(x)/(\sigma^2(x)-2\kappa b(x))$, and it is used here and again in the next first-moment lemma.}
\end{lemma}

%\begin{proof}
%    See Supplement.
%\end{proof}

To complete the local error proof we still need bounds on the first and third conditional moments themselves.
The first-moment estimate is not immediate, because jumps to $0$ are of size $\mathcal{O}(\sqrt{h})$.
Moreover, we will estimate first moment depending on $x\in \Omega_{b}$ and $x\in \Omega_{r}$ separately.
Exactly as in the SBM argument, the previous cancellation lemma is what lowers the first moment to order $h$.
\begin{lemma}
\label{lemma:1st_mom_order_negative_drifted_diffusion}
    There exists $\delta>0$ such that $\forall h < \delta$ and $\forall x\in \Omega_{b}$, 
     \begin{equation}
     \label{eqn:first_mom_negative_drifted_bound_diffusion}
        \left|\mathbb{E}_x(\Delta X_k)\right| \leq C(h + hx+h^2),
     \end{equation}
     where $C>0$ is independent of $x$ and $h$ and depends only on $\kappa$, the local boundary bounds for $b,\sigma$, and the nondegeneracy constant coming from $b(0)\neq \sigma^2(0)/(2\kappa)$.
\end{lemma}
%\begin{proof}
%    See Supplement.
%\end{proof}

\begin{lemma}
\label{lemma:3rd_mom_order_negative_drifted_diffusion}
    There exists $\delta>0$ such that $\forall h < \delta$ and $\forall x\in \Omega_{br}$,
     \begin{equation}
     \label{eqn:first_mom_drifted_bound}
        \left|\mathbb{E}_x(\Delta X_k^3)\right| \leq K_4(hx + h^2),
     \end{equation}
     where $K_4$ is a fixed constant dependent on drift, diffusion, and $\kappa$.
\end{lemma}
%\begin{proof}
%    See Supplement.
%\end{proof}

We now prove the local error estimate in $\Omega_{b}$ for the generic case $b(0)\neq \frac{\sigma^2(0)}{2\kappa}$.
The proof follows the same Taylor-expansion template as in Section 3.5, but the cancellation now uses Lemma \ref{lemma:sticky_bc@x_diffusion} and Lemma \ref{lemma:1st_2nd_mom_combination_generator_diffusion}.
\begin{theorem}
\label{thm:local_err_cont_scheme_diffusion}
    (Local Error of $\{X_k\}_{k\in\mathbb{N}}$ in algorithm \ref{alg:diffusion_cont_scheme}, situation 1)\\
    Given $b(0) \neq \frac{\sigma^2(0)}{2\kappa}$, there exists $\delta > 0$ and $C\geq 0$ depending on the relevant bounds of derivatives of $u$ and on the local boundary constants for $b,\sigma$ such that $\forall h < \delta$ and $\forall x\in \Omega_{b}$,
    \begin{equation}
    \label{eqn:local_err_cont_scheme_drifted}
        \left|\mathbb{E}\left(u_{k+1}-u_k|X_k=x\right)\right| \leq C(h^2 + hx) = \mathcal{O}(hx + h^2) 
    \end{equation}
\end{theorem}

%\begin{proof}
%    See Supplement.
%\end{proof}

Now, as mentioned before, if $b(x) < 0$, then for some $x\in \Omega_{r} \subseteq [0,\sigma_2\sqrt{3h}]$, we may have $x - \sigma(x)\sqrt{3h} + b(x)h < 0$ and $x - \sigma(x)\sqrt{3h} > 0$.
According to our scheme (Algorithm \ref{alg:diffusion_cont_scheme}), we do a reflected EM jump, and there is zero probability to jump to the sticky boundary.
It is non-trivial to show this will give the desired order of local error because the first conditional moment of $\Delta X_k$ might not be a good approximation of $b(x)h$.
The next lemma aims to resolve this doubt.

\begin{lemma}
\label{lemma:first_mom_est_reflection_layer_diffusion}
   There exists $\delta > 0$ and a constant $C$, which depends only on local boundary constants for $b,\sigma$ coming from Lemma \ref{lemma:bound_on_drift_diffusion} such that $\forall h < \delta$ and $\forall x\in \Omega_{r}$, 
   %\redan{In this lemma, only the diffusion lower bound $\sigma_1$ and the drift bound $B_M:=|b_1|\vee |b_2|$ are actually used.}
    $$|E_x(\Delta X_k) - b(x)h| \leq C( hx + h^2) = \mathcal{O}(hx + h^2).$$ 
\end{lemma}
%\begin{proof}
%    See Supplement.
%\end{proof}

It is not hard to further compute that the second conditional moment of $\Delta X_k$ is approximately $\sigma^2(x)h$ with a remainder term of order $hx + h^2$.
Now, we are ready to show that the local error for $x$ such that $x - \sigma(x)\sqrt{3h} + b(x)h < 0$ and $x - \sigma(x)\sqrt{3h}>0$, 
\begin{theorem}
\label{thm:local_err_cont_scheme_diffusion_ring}
    (Local Error of $\{X_k\}_{k\in\mathbb{N}}$ in algorithm \ref{alg:diffusion_cont_scheme}, situation 2)\\
    Given $b(0) \neq \frac{\sigma^2(0)}{2\kappa}$, there exists $\delta > 0$ and $C \geq 0$ depending on the relevant derivative bounds of $u$ and on the local boundary constants for $b$ and $\sigma$ such that $\forall h < \delta$ and $\forall x\in \Omega_{r}$,
    \begin{equation}
    \label{eqn:local_err_cont_scheme_drifted_ring}
        \left|\mathbb{E}\left(u_{k+1}-u_k|X_k=x\right)\right| \leq C(h^2 + hx) = \mathcal{O}(hx + h^2).
    \end{equation}
\end{theorem}
%\begin{proof}
%    See Supplement.
%\end{proof}

So far, we have shown the order of local error for our scheme when $x - \sigma(x)\sqrt{3h} < 0$ (where we need to consider absorption probability $\lambda(x)$) and $\sigma(x)\sqrt{3h} \leq x < \sigma(x)\sqrt{3h} - b(x)h$ (where the negative drift can cause a reflected EM jump).
Notice for the first situation, we assume that $\sigma(0) \neq \frac{\sigma^2(0)}{2\kappa}$ so that $M(x)$ can be a bounded function within the boundary layer.
We need to compute the local error of our scheme when $b(0) = \frac{\sigma^2(0)}{2\kappa}$ separately. 
We will show, similar to situation 2, that the first conditional moment of $\Delta X_k$ is a good approximation of $b(x)h$ within the boundary layer. 
\begin{lemma}
\label{lemma:1st_2nd_mom_combination_generator_diffusion_special}
    When $b(0) = \frac{\sigma^2(0)}{2\kappa}$, there exist $\delta>0$ and a constant $C$ which depends only on the boundary-layer bounds, such that $\forall h< \delta$ and $\forall x\in \Omega_{b}$,
    $$\left| \mathbb{E}_x\left(\Delta X_k\right) - hb(x) \right| \leq C(hx + h^2) = \mathcal{O}(hx + h^2).$$
\end{lemma}
%\begin{proof}
%    See Supplement.
%\end{proof}

We now proceed to show the order of local error for case 3, where $x-\sigma(x)\sqrt{3h} < 0$ and $b(0) = \frac{\sigma^2(0)}{2\kappa}$. This case must have a positive drift in $\Omega_{br}$ for $h$ small enough by its regularity assumption so there is no extra subtlety caused by negative drift.
\begin{theorem}
\label{thm:local_err_cont_scheme_diffusion_special}
    (Local Error of $\{X_k\}_{k\in\mathbb{N}}$ in algorithm \ref{alg:diffusion_cont_scheme}, situation 3)\\
    If $b(0)= \frac{\sigma^2(0)}{2\kappa}$, there exist $\delta > 0$ and $C\geq 0$, independent of $x$ and $h$, depending only on the relevant bounds of derivatives of $u$ and on the local boundary constants for $b,\sigma$ such that $\forall h<\delta$ and for all $x$ such that $x \in \Omega_{br}$,
    \begin{equation}
    \label{eqn:local_err_cont_scheme_drifted_speical}
        \left|\mathbb{E}\left(u_{k+1}-u_k|X_k=x\right)\right| \leq C(h^2 + hx) = \mathcal{O}(hx + h^2).
    \end{equation}
\end{theorem}
%\begin{proof}
%    See Supplement.
%\end{proof}

%\subsection{Global Error Analysis}

\subsubsection{Bounding the number of steps: Lemma \ref{lemma:key_diffusion_general}}

To enable comparison with the SBM construction, we first normalize the boundary diffusion coefficient. Let $\sigma_0:=\sigma(0)>0$ and define the rescaled process
$$
\hat X_t:=\frac{X_t}{\sigma_0}.
$$
Then $\hat X$ is again a sticky diffusion of the same form, with coefficients
$$
\hat b(y):=\frac{b(\sigma_0 y)}{\sigma_0},
\qquad
\hat \sigma(y):=\frac{\sigma(\sigma_0 y)}{\sigma_0},
$$
and sticky parameter $\hat\kappa=\kappa/\sigma_0$. In particular, $\hat\sigma(0)=1$. Therefore it is enough to prove the desired estimate in the normalized case $\sigma(0)=1$, and we adopt this normalization below.

Our hope is to recycle our calculation for the global error analysis done in the proof of Theorem \ref{thm:global_err_cont_scheme} to show similar theorems for the scheme with drift.

The first thing we notice is that when $\sigma(x) = 1$, then the jump probability $\lambda$ for the general diffusion, \eqref{eqn:diffusion_lambda_cont_scheme}, is
\[
 \lambda(x) = \frac{(1-2\kappa b(x))x^2 + 2\kappa x + h}{\left(\frac{\kappa}{\sqrt{3h}}+1-2\kappa b(x)\right)x^2 + h + \kappa\sqrt{3h}},
\]
which differs from the jump probability for the SBM, \eqref{eqn:lambda_cont_scheme}, only in terms involving $b(x)$.  Since these terms multiply $x^2$, which is small in the boundary layer, and since if $\sigma(0) = 1$ then $\sigma(x)\approx 1$ in the boundary layer, we expect that these two jump probabilities will be not too different. 
The next lemma shows that the difference between the two is $\mathcal{O}(\sqrt{h})$.

%for SBM, \eqref{eqn:lambda_cont_scheme}, and for the general diffusion, \eqref{eqn:diffusion_lambda_cont_scheme}, are 
%$\lambda(\cdot)$ we propose in algorithm \ref{alg:diffusion_cont_scheme} is not significantly different from the one we propose in algorithm \ref{alg:cont_scheme}  when the diffusion at boundary are them same.
%In fact, I will show that it is a $\mathcal{O}(\sqrt{h})$ approximation when the drift at the origin is $\sigma(0) = \sigma_0 = 1$.
%In other words, for general $\sigma(x)$, the jump probability in the boundary layer is an approximation of the jump probability with constant diffusion term.
%Reader will see it is also used in proving the global error of our scheme.

\begin{lemma}
\label{lemma:lambda_approx}
    Denote $\tilde{\lambda}(\cdot)$ as the jump probability for SBM, \eqref{eqn:lambda_cont_scheme}, %we propose in algorithm \ref{alg:cont_scheme} 
    and $\lambda(\cdot)$ as the jump probability for the general diffusion, \eqref{alg:diffusion_cont_scheme}. %one we propose in algorithm \ref{alg:diffusion_cont_scheme}, if for the diffusion $\sigma(0)= 1$ for a sticky diffusion, then
    If $\sigma(0) =1$, then 
    $$\left|\lambda(x) - \tilde{\lambda}(x)\right| \leq K_{b,\sigma} \sqrt{h},\quad \forall x\in [0,\sqrt{3h}],$$ 
    where $K_{b,\sigma} > 0$ is a constant depending on the drift term $b(\cdot)$, diffusion $\sigma(\cdot)$, and sticky parameter $\kappa$.
\end{lemma}
%\begin{proof}
%    See Supplement.
%\end{proof}

These results may be used to show Lemma \ref{lemma:key_diffusion_general}, which is analogous to Lemma \ref{lemma:key}, and is the final step in proving the order of global error. The proof of Lemma \ref{lemma:key_diffusion_general} is in the Supplement.

%\begin{proof}[Proof of Lemma \ref{lemma:key_diffusion_general}]
%    See Supplement.
%\end{proof}

%%%%%%%%%%%%%%%%%%%%%%%%%%
%%%%%   Conclusion   %%%%%
%%%%%%%%%%%%%%%%%%%%%%%%%%

\section{Conclusion}

This paper presented and analyzed a numerical scheme for simulating one-dimensional sticky diffusions. 
The scheme combines a reflected Euler-Maruyama increment, with a carefully designed jump probability to the sticky point, to handle the challenge of moving both in continuous space, and to a lower-dimensional boundary. 
%Unlike previously studied schemes which make discrete jumps in space, the scheme we propose operates in continuous space. 
Our scheme achieves a weak local error of order 3/2 within the boundary layer and a weak global error of order 1. 

We provide rigorous theoretical analysis of our schemes, including detailed proofs of local and global error bounds. 
Numerical experiments corroborate our theoretical findings, showing empirical evidence of first-order global convergence for both sticky Brownian motion and sticky diffusion processes under various conditions.

A natural extension of our work would be to develop schemes for multidimensional sticky diffusions. A novel features of such diffusions, compared to the one-dimensional case, is that they can have nontrivial dynamics along their sticky boundaries, which can be different from the dynamics in the interior. We are optimistic that the method we have outlined can be adapted to such diffusions, because it provides a clear guideline for choosing how to switch between different types of  dynamics (i.e. on sets of different dimensions), which are otherwise standard to simulate. A challenge however will be ensuring that the switching probabilities are indeed probabilities, as we have found this nontrivial to ensure even for a one-dimensional diffusion. Further extensions could include adapting these methods to processes on manifolds \cite{Ciccotti.2007}, and considering multiple intersecting sticky boundaries which form a stratification, as outlined in \cite{Holmes-Cerfon.2013,Holmes-Cerfon.2020}. 

Even for a one-dimensional diffusion there is further theoretical development possible. One could investigate higher-order schemes, which might achieve even better convergence rates, or could ask how to solve general parabolic-type PDEs with sticky boundary conditions \cite{leimkuhler2023reflect_SDE,Sharma.2025}. One could also ask for \emph{worse} schemes, which achieve only $O(h^{1/2})$ convergence -- perhaps there is a wider scope of jump probabilities $\lambda$ which allow for worse convergence, but are easier to either derive, or to implement. Such analysis could perhaps be applied to Monte-Carlo schemes which perform local moves to sample from a stationary distribution supported on manifolds of different dimensions, and show their convergence to a particular sticky diffusion \cite{Holmes-Cerfon.2020}. Finally, it would be useful to perform a theoretical analysis of the long-time behavior of the numerical schemes and their ability to capture stationary distributions accurately.

% Future work could explore several directions. One important avenue is to investigate if there exists a better scheme that can achieve faster convergence when the drift is negative, as our current scheme requires extremely small time steps. 
% Other potential areas of research include:
% \begin{itemize}
%     \item Extension of the schemes to multi-dimensional sticky processes and more complex geometries.
%     \item Investigation of higher-order schemes that might achieve even better convergence rates.
%     \item Application of these methods to specific real-world problems involving sticky processes.
%     \item Theoretical analysis of the long-time behavior of the numerical schemes and their ability to capture stationary distributions accurately.
% \end{itemize}

In conclusion, this work contributes to the growing toolkit for numerical simulation of stochastic processes with complex boundary condition,  offering both practical algorithms and theoretical guarantees for simulating certain sticky processes. 
As the study of such processes continues to grow in importance across various disciplines, we believe these methods will prove valuable tools for researchers and practitioners alike.

%%%%%%%%%%%%%%%%%%%%%%%%%%%%%%%%%%%%%%%%%%%%%%
%% Single Appendix:                         %%
%%%%%%%%%%%%%%%%%%%%%%%%%%%%%%%%%%%%%%%%%%%%%%
%\begin{appendix}
%\section*{???}%% if no title is needed, leave empty \section*{}.
%\end{appendix}
%%%%%%%%%%%%%%%%%%%%%%%%%%%%%%%%%%%%%%%%%%%%%%
%% Multiple Appendixes:                     %%
%%%%%%%%%%%%%%%%%%%%%%%%%%%%%%%%%%%%%%%%%%%%%%
%\begin{appendix}
%\section{???}
%
%\section{???}
%
%\end{appendix}

%%%%%%%%%%%%%%%%%%%%%%%%%%%%%%%%%%%%%%%%%%%%%%
%% Support information, if any,             %%
%% should be provided in the                %%
%% Acknowledgements section.                %%
%%%%%%%%%%%%%%%%%%%%%%%%%%%%%%%%%%%%%%%%%%%%%%
%\begin{acks}[Acknowledgments]
% The authors would like to thank ... 
%\end{acks}
%%%%%%%%%%%%%%%%%%%%%%%%%%%%%%%%%%%%%%%%%%%%%%
%% Funding information, if any,             %%
%% should be provided in the                %%
%% funding section.                         %%
%%%%%%%%%%%%%%%%%%%%%%%%%%%%%%%%%%%%%%%%%%%%%%
\begin{funding}
M.H.C. acknowledges support from the Natural Sciences and Engineering Research Council of Canada (NSERC), RGPIN-2023-04449 / Cette recherche a \'{e}t\'{e} financ\'{e}e par le Conseil de recherches en sciences naturelles et en g\'{e}nie du Canada (CRSNG).
\end{funding}

%%%%%%%%%%%%%%%%%%%%%%%%%%%%%%%%%%%%%%%%%%%%%%
%% Supplementary Material, including data   %%
%% sets and code, should be provided in     %%
%% {supplement} environment with title      %%
%% and short description. It cannot be      %%
%% available exclusively as external link.  %%
%% All Supplementary Material must be       %%
%% available to the reader on Project       %%
%% Euclid with the published article.       %%
%%%%%%%%%%%%%%%%%%%%%%%%%%%%%%%%%%%%%%%%%%%%%%
\begin{supplement}
\stitle{Supplemental Information for ``A modified Euler-Maruyama method to simulate a one-dimensional sticky diffusion''}
\sdescription{Contains detailed proofs of intermediate convergence statements stated in the main text.}
\end{supplement}

%%%%%%%%%%%%%%%%%%%%%%%%%%%%%%%%%%%%%%%%%%%%%%%%%%%%%%%%%%%%%
%%                  The Bibliography                       %%
%%                                                         %%
%%  imsart-???.bst  will be used to                        %%
%%  create a .BBL file for submission.                     %%
%%                                                         %%
%%  Note that the displayed Bibliography will not          %%
%%  necessarily be rendered by Latex exactly as specified  %%
%%  in the online Instructions for Authors.                %%
%%                                                         %%
%%  MR numbers will be added by VTeX.                      %%
%%                                                         %%
%%  Use \cite{...} to cite references in text.             %%
%%                                                         %%
%%%%%%%%%%%%%%%%%%%%%%%%%%%%%%%%%%%%%%%%%%%%%%%%%%%%%%%%%%%%%

%% if your bibliography is in bibtex format, uncomment commands:
\bibliographystyle{imsart-number} % Style BST file (imsart-number.bst or imsart-nameyear.bst)
\bibliography{Sticky_BM,reference,SBMPaper}       % Bibliography file (usually '*.bib')

%% or include bibliography directly:
% \begin{thebibliography}{}
% \bibitem{b1}
% \end{thebibliography}

%%%%%%%%%%%%%%%%%%%%%%%%%%%%%%%%%%%%%%%%%%%%%%
%% Single Appendix:                         %%
%%%%%%%%%%%%%%%%%%%%%%%%%%%%%%%%%%%%%%%%%%%%%%
\begin{appendix}
\section{Supplement}%% if no title is needed, leave empty \section*{}.

\subsection{Additional proofs from Section 3}\label{sec:appendix_sec3}
\phantom{x}

\noindent\textbf{Verification of \eqref{eqn:wx_goal} in the proof of Lemma~\ref{lemma:key}}

As a reminder, we wish to show that there exist constants $C_1(s,M)>0$ and $\delta>0$, independent of $x$ and $h$, such that for all $h<\delta$,
    \begin{equation*}
        Pe^{w(x)} \leq e^{w(x)}\left(1-f_h(x)+C_1h\right),
    \end{equation*}
    where $f_h$ satisfies
    $$
    f_h(x)\geq x\,\mathbf{1}_{\{x\in[0,\sqrt{3h}]\}}.
    $$
We would like to calculate $Pe^{w(x)}-e^{w(x)}$ for $x$ in three cases: $[0,\frac{\sqrt{3h}}{2}]$, $[\frac{\sqrt{3h}}{2}, \sqrt{3h}]$ and $[\sqrt{3h},+\infty)$.
\begin{enumerate}
        \item 
        For $x\in [0, \frac{\sqrt{3h}}{2}]$, we compute
        \begin{equation*}
        \begin{split}
            Pe^{w(x)} &= \frac{\lambda(x)}{2}e^M 
            + \frac{\lambda(x)}{\sqrt{3h}}\cdot\int_{\frac{\sqrt{3h}}{2}}^{\sqrt{3h}-x} e^{M - s\left(z-\frac{\sqrt{3h}}{2}\right)}\; dz
            %\\&\quad 
            + \frac{\lambda(x)}{2\sqrt{3h}} \cdot \int_{\sqrt{3h}-x}^{\sqrt{3h}+x} e^{M- s\left(z-\frac{\sqrt{3h}}{2}\right)}\; dz
            \\&\quad + e^M(1-\lambda(x))\\
            &= \lambda(x) \left(\frac{1}{2}e^M 
            + e^M \cdot \frac{1}{2\sqrt{3h}s} \cdot \left( 2 - e^{-s\left(\frac{\sqrt{3h}}{2}-x\right)} - e^{-s\left(\frac{\sqrt{3h}}{2}+x\right)} \right)\right)
            + e^M(1-\lambda(x)).
        \end{split}
        \end{equation*}
        In the first line, the first term corresponds to reflected jumps that land in the flat part of $w$, the two integral terms correspond to reflected jumps that land in the sloping part of $w$ (split according to whether the reflection map has one or two preimages), and the last term corresponds to the jump directly to $0$.
        
        Now we simplify this expression by Taylor expansion of exponentials with a remainder of order $\mathcal{O}(h)$:
        \begin{equation*}
        \begin{split}
            \quad&Pe^{w(x)} - e^{w(x)}\\
            &= e^M \cdot \frac{\lambda(x)}{2\sqrt{3h}s} \cdot \left( 2 - e^{-s\left(\frac{\sqrt{3h}}{2}-x\right)} 
            - e^{-s\left(\frac{\sqrt{3h}}{2}+x\right)} \right) 
            - e^M \cdot \frac{\lambda(x)}{2}\\
            &= e^M\cdot \frac{\lambda(x)}{2\sqrt{3h}s} \cdot \left(\sqrt{3h}s - \frac{1}{2}s^2\left(\frac{3h}{2}+2x^2\right) + \mathcal{O}(h^{3/2})\right) - e^M \cdot \frac{\lambda(x)}{2}\\
            &= e^M\cdot \frac{\lambda(x)}{2\sqrt{3h}s} \cdot \left(- \frac{1}{2}s^2\left(\frac{3h}{2}+2x^2\right) + \mathcal{O}(h^{3/2})\right)\\
            &= e^M\cdot \left( -\frac{s}{4}\cdot \frac{\lambda(x)}{2}\sqrt{3h} - s \cdot \frac{\lambda(x)x^2}{2\sqrt{3h}} + \mathcal{O}(h)\right).
        \end{split}
        \end{equation*}

        In other words, we have shown that, for  $x \in [0,\sqrt{3h}/2]$,
        $$Pe^{w(x)} \leq e^{w(x)}\left(1-f_h(x) + \mathcal{O}(h)\right)$$
        where we define
        \begin{equation}
        \label{eqn:fh_in_0_to_sqrt(3h)/2}
            f_h(x):= \frac{s}{4}\cdot\frac{\lambda(x)}{2}\sqrt{3h} + s \cdot \frac{\lambda(x)\cdot x^2}{2\sqrt{3h}}
            \qquad \text{for } x\in [0, \frac{\sqrt{3h}}{2}].%, \quad \forall x\in [0, \frac{\sqrt{3h}}{2}].
        \end{equation}

%\todan{is there another way to explain the below, a bit more step-by-step? goal is to show that $f_h(x)-x>0$}
        
        %Notice that when $s>16$, 
        %\redan{$s>4$ shall be enough, I take $s>16$ because it is needed for case 2 \eqref{eqn:fh_in_sqrt(3h)/2_to_sqrt(3h)}, $\mathbb{E}_x(\Delta X)$ comes from \eqref{eq:DelX}, I make this more explicit below.}
        
        Recall the expression for $\mathbb{E}_x(\Delta X)$, \eqref{eq:DelX} and substitute the expression for $\lambda(x)$,  \eqref{eqn:lambda_cont_scheme},
        \begin{equation*}
        \begin{split}
            \mathbb{E}_x(\Delta X_k) &=-x + \frac{\lambda(x)\sqrt{3h}}{2} + \frac{\lambda(x)  x^2}{2\sqrt{3h}}\\ 
            &= \frac{1}{2\sqrt{3h}}\cdot\frac{\left(x-\sqrt{3h}\right)^2(x^2+h)}{\left(\frac{\kappa}{\sqrt{3h}}+1\right)x^2+h+\kappa\sqrt{3h}}\\
            &\geq 0.
        \end{split}
        \end{equation*}
        Hence,
        $$f_h(x) - x > \mathbb{E}_x(\Delta X) > 0, \quad \forall x\in [0, \frac{\sqrt{3h}}{2}].$$
        
        \item For $x\in [\frac{\sqrt{3h}}{2}, \sqrt{3h}]$, 
        % \todan{I deleted comments like ``\textit{i.e.} $\sqrt{3h}-x \in [0, \frac{\sqrt{3h}}{2}]$'' in these steps as I wasn't sure what this added to the calculation, but feel free to put back in if you think it helps understand the calculation.}
        %\textit{i.e.} $\sqrt{3h}-x \in [0, \frac{\sqrt{3h}}{2}]$, the scheme does a reflected uniform jump of length $\sqrt{3h}$ centered at $x$ with probability $\lambda(x)$ or jump to $0$ with probability $1-\lambda(x)$.
        the scheme does the same jump, but the expression for $e^{w(x)}$ is different, which makes the calculation much more complex.
        It can be checked that
        \begin{equation*}
        \begin{split}
            Pe^{w(x)} &= \lambda(x) \left(e^M \cdot \frac{\sqrt{3h}-x}{\sqrt{3h}}
            + e^M \cdot \frac{\frac{\sqrt{3h}}{2}-(\sqrt{3h}-x)}{2\sqrt{3h}}\right)\\
            &\quad + \frac{\lambda(x)}{2\sqrt{3h}}\cdot\int_{\frac{\sqrt{3h}}{2}}^{\sqrt{3h}+x}e^{M-s\left(z-\frac{\sqrt{3h}}{2}\right)}\; dz + (1-\lambda(x))e^M\\
            &= \lambda(x) e^M \cdot \frac{3\sqrt{3h}-2sx + 2 - 2e^{-s\left(\frac{\sqrt{3h}}{2}+x\right)}}{4\sqrt{3h}s}
            + (1-\lambda(x))e^M.
        \end{split}
        \end{equation*}
        Since $e^{w(x)} = e^{M-s\left(x-\frac{\sqrt{3h}}{2}\right)}$,
        \begin{equation*}
        \begin{split}
            Pe^{w(x)} - e^{w(x)} &= A + B
        \end{split}
        \end{equation*}
        where 
        \begin{align*}
            A&=  \lambda(x) e^M \cdot \frac{3\sqrt{3h}-2sx + 2 - 2e^{-s\left(\frac{\sqrt{3h}}{2}+x\right)}}{4\sqrt{3h}s} - \lambda(x) \cdot e^{M-s\left(x-\frac{\sqrt{3h}}{2}\right)}\\
        B &= (1-\lambda(x))\cdot e^M - (1-\lambda(x))\cdot e^{M-s\left(x-\frac{\sqrt{3h}}{2}\right)}.
        \end{align*}
        Simplifying $A,B$ by expanding the exponential, we have 
        \begin{equation}
        \label{eqn:A_simplify}
        \begin{split}
            A &= \frac{\lambda(x)}{4\sqrt{3h}s} \cdot e^{M-s\left(x-\frac{\sqrt{3h}}{2}\right)} 
            \cdot \left( (3\sqrt{3h}s-2sx+2)e^{s\left(x-\frac{\sqrt{3h}}{2} \right)} - 2e^{-s\sqrt{3h}} - 4\sqrt{3h}s\right)\\
            &= e^{w(x)} \cdot \frac{\lambda(x)}{4\sqrt{3h}s}
            \cdot \left( 3\sqrt{3h}s^2 x - \frac{9}{4}s^2\cdot 3h - s^2x^2 +  \mathcal{O}(h^{3/2})\right)\\
            &= \frac{\lambda(x)}{4}e^{w(x)} \left( 3sx - \frac{9}{4}\sqrt{3h}s - \frac{sx^2}{\sqrt{3h}} +  \mathcal{O}(h) \right).
        \end{split}
        \end{equation}
        Meanwhile,
        \begin{equation}
        \label{eqn:B_simplify}
        \begin{split}
            B &= e^{M-s\left(x-\frac{\sqrt{3h}}{2}\right)} \cdot (1-\lambda(x)) 
            \left( e^{s\left(x-\frac{\sqrt{3h}}{2}\right)} - 1 \right)\\
            &= e^{M-s\left(x-\frac{\sqrt{3h}}{2}\right)} \cdot (1-\lambda(x)) 
            \left( s\left(x-\frac{\sqrt{3h}}{2}\right) + \mathcal{O}(h) \right).
        \end{split}
        \end{equation}
        Combining \eqref{eqn:A_simplify}, \eqref{eqn:B_simplify} together, we have
        $$Pe^{w(x)} - e^{w(x)} = A + B = -e^{w(x)} \left( f_h(x) +  \mathcal{O}(h)\right)$$
        where we define
        \begin{equation}
        \label{eqn:fh_in_sqrt(3h)/2_to_sqrt(3h)}
            f_h(x) := \frac{s\lambda(x)}{16}\sqrt{3h} + \frac{s\lambda(x)x^2}{4\sqrt{3h}} + \frac{s\lambda(x)x}{4} + \frac{s\sqrt{3h}}{2} - sx\quad \text{for $x\in [\frac{\sqrt{3h}}{2}, \sqrt{3h}]$}.
        \end{equation}
        It can be checked by derivative that $f_h(x)$ is decreasing for $x\in [\frac{\sqrt{3h}}{2}, \sqrt{3h}]$, and $f_h(\sqrt{3h}) = \frac{s}{16}\sqrt{3h} > \sqrt{3h}$, so we must have that
        $$f_h(x) - x > 0, \quad \forall x\in [\frac{\sqrt{3h}}{2}, \sqrt{3h}].$$
        
        \item
        For $x\in (\sqrt{3h},+\infty)$: in this region, the scheme does a standard EM jump.
        Using the fact that $\mathbb{E}_x(\Delta X) = 0$, and then using the fact that $w$ is a piecewise linear function with a constant piece near the boundary, so it smaller than its linear part: $w(x+\Delta X) \leq M - s\left((x+\Delta X)-\frac{\sqrt{3h}}{2}\right)$, 
        we have 
        %\todan{what is $\Delta X_h$ in first line? and for the second line, is this a taylor expansion?}
        %\redan{It is a small typo $\Delta X_h$ should be $\Delta X$. For the second line, it is not Taylor, it is by definition of \eqref{eqn:chosen_wx_continuous} because it is a (piecewise) linear function with a constant piece near the boundary, so it smaller than its linear part: $w(x+\Delta X) \leq M - s\left((x+\Delta X)-\frac{\sqrt{3h}}{2}\right)$}
        \begin{equation*}
        \begin{split}
            Pe^{w(x)} &= \mathbb{E}_x\left(e^{w(x+\Delta X)}\right)\\
            &\leq \mathbb{E}_x\left(e^{w(x)- s\cdot \Delta X}\right)\\
            &= e^{w(x)}\cdot (1 - s\cdot\mathbb{E}_x(\Delta X) + \mathcal{O}(h))\\
            &= e^{w(x)}\cdot (1 + 0\cdot (-s) + \mathcal{O}(h))\\
            &= e^{w(x)}\left(1 - f_h(x) + \mathcal{O}(h) \right)\\
        \end{split}
        \end{equation*}
        where
        \begin{equation}
        \label{eqn:fh_in_sqrt(3h)_to_inf}
            f_h(x) := 0 ,\qquad \text{for $x\in (\sqrt{3h},+\infty).$}
        \end{equation}
    \end{enumerate}
    So far, we have shown that for $x\in [0,+\infty)$, 
    \begin{equation}
    \label{eqn:wx_goal_supp}
        Pe^{w(x)} \leq e^{w(x)} \left(1 - f_h(x) + \mathcal{O}(h)\right)
    \end{equation}
    where $f_h(x) \geq x \cdot \mathbf{1}_{x \in [0,\sqrt{3h}]}$ is defined according to \eqref{eqn:fh_in_0_to_sqrt(3h)/2},\eqref{eqn:fh_in_sqrt(3h)/2_to_sqrt(3h)},\eqref{eqn:fh_in_sqrt(3h)_to_inf}, and the bounding constant of the big-$\mathcal{O}$ term depends on $s$.
    In each of the three cases above, the remainder term $\mathcal O(h)$ is uniform in $x$ on the corresponding interval. Therefore, after taking the maximum of the three remainder constants and the minimum of the corresponding smallness thresholds, there exist constants $C_1(s,M)>0$ and $\delta>0$, independent of $x$ and $h$, such that for all $h<\delta$,
    $$
    Pe^{w(x)} \le e^{w(x)}\bigl(1-f_h(x)+C_1h\bigr),
    \qquad x\ge 0.
    $$
    Together with the bounds established above,
    $$
    f_h(x)\ge x\,\mathbf{1}_{\{x\in[0,\sqrt{3h}]\}},
    $$
    which proves \eqref{eqn:wx_goal}.

\subsection{Additional proofs from Section 4}
\phantom{x}

\noindent\textbf{Lemma~\ref{lemma:bound_on_drift_diffusion}.}
    There exists $\delta>0$ and constants $b_1,b_2,\sigma_1,\sigma_2$ such that, for every $h<\delta$,
    $$
    b(x)\in [b_1,b_2],\qquad \sigma(x)\in[\sigma_1,\sigma_2],
    $$
    whenever
    $$
    x-\sigma(x)\sqrt{3h}<0
    \qquad\text{or}\qquad
    x-\sigma(x)\sqrt{3h}+b(x)h<0.
    $$
\begin{proof}
    We first show that the boundary layer is contained in a fixed compact neighborhood of $0$ for all sufficiently small $h$. By Assumption~\ref{assump:linear_growth_coefficients}, there exists $L>0$ such that
    $$
    |b(x)|+|\sigma(x)|\leq L(1+x),\qquad x\geq 0.
    $$
    If $x-\sigma(x)\sqrt{3h}<0$, then
    $$
    x<\sigma(x)\sqrt{3h}\leq |\sigma(x)|\sqrt{3h}\leq L(1+x)\sqrt{3h}.
    $$
    If $x-\sigma(x)\sqrt{3h}+b(x)h<0$, then
    $$
    x<\sigma(x)\sqrt{3h}-b(x)h
    \leq |\sigma(x)|\sqrt{3h}+|b(x)|h
    \leq L(1+x)(\sqrt{3h}+h).
    $$
    Hence, after choosing $\delta_0>0$ small enough so that
    $$
    L(\sqrt{3\delta_0}+\delta_0)\leq \frac{1}{2},
    $$
    both cases imply
    $$
    x\leq 2L(\sqrt{3h}+h),\qquad h<\delta_0.
    $$
    In particular, by decreasing $\delta_0$ if necessary, every such $x$ lies in the interval $[0,r_0]$ from Assumption~\ref{assump:local_coeff_boundary_regularity}. Moreover,
    $$
    |\partial_x b(y)|+|\partial_x\sigma(y)|\leq C_{\mathrm{loc}},
    \qquad y\in[0,r_0].
    $$
    Now we sharpen this localization using the local boundedness of $\partial_x b$ and $\partial_x\sigma$ on $[0,r_0]$. Denote
    $$
    \sigma_0:=\sigma(0),\qquad b_0:=b(0),
    $$
    If $x-\sigma(x)\sqrt{3h}+b(x)h<0$, then by the mean value theorem,
    \begin{equation}
    \label{eqn:diffusion_boundary_layer_expansion}
    \begin{split}
        x
        &< \sigma(x)\sqrt{3h}-b(x)h\\
        &= \left(\sigma_0+x\partial_x\sigma(z_x)\right)\sqrt{3h}
        -\left(b_0+x\partial_x b(y_x)\right)h
    \end{split}
    \end{equation}
    for some $y_x,z_x\in[0,x]$. Therefore
    $$
    x\left(1-\partial_x\sigma(z_x)\sqrt{3h}+\partial_x b(y_x)h\right)
    < \sigma_0\sqrt{3h}-b_0h.
    $$
    After decreasing $\delta_0$ if necessary, we have
    $$
    1-\partial_x\sigma(z_x)\sqrt{3h}+\partial_x b(y_x)h\geq \frac{1}{2},
    \qquad h<\delta_0,
    $$
    and hence
    $$
    x<\sigma_0\sqrt{3h}+C_1h,
    $$
    where $C_1$ depends only on $b_0,\sigma_0$ and the local bound $C_{\mathrm{loc}}$.
    
    \noindent Similarly, if $x-\sigma(x)\sqrt{3h}<0$, then
    \begin{equation}
    \label{eqn:diffusion_boundary_layer_expansion_2}
        x
        < \sigma(x)\sqrt{3h} = \left(\sigma_0+x\partial_x\sigma(z_x)\right)\sqrt{3h},
    \end{equation}
    so
    $$
    x\left(1-\partial_x\sigma(z_x)\sqrt{3h}\right)<\sigma_0\sqrt{3h}.
    $$
    After decreasing $\delta_0$ if necessary,
    $$
    1-\partial_x\sigma(z_x)\sqrt{3h}\geq \frac{1}{2},
    $$
    and therefore
    $$
    x<\sigma_0\sqrt{3h}+C_2h,
    $$
    where $C_2$ depends only on $\sigma_0$ and $C_{\mathrm{loc}}$.
    Let $C:=C_1\vee C_2$. We have shown that for every $h<\delta_0$,
    $$
    \left\{x\geq 0:\ x-\sigma(x)\sqrt{3h}<0
    \ \text{or}\ 
    x-\sigma(x)\sqrt{3h}+b(x)h<0\right\}
    \subseteq [0,\sigma_0\sqrt{3h}+Ch].
    $$
    Now choose $\delta\in(0,\delta_0]$ small enough so that
    $$
    B_\delta:=[0,\sigma_0\sqrt{3\delta}+C\delta]\subseteq[0,r_0].
    $$
    Since $b$ and $\sigma$ are continuous on $B_\delta$, define
    $$
    b_1:=\min_{x\in B_\delta} b(x),\qquad
    b_2:=\max_{x\in B_\delta} b(x),
    $$
    and
    $$
    \sigma_1:=\min_{x\in B_\delta}\sigma(x),\qquad
    \sigma_2:=\max_{x\in B_\delta}\sigma(x).
    $$
    By Assumption~\ref{assump:local_coeff_boundary_regularity}, after decreasing $\delta$ if necessary, $\sigma_1>0$. Therefore, for every $h<\delta$ and every $x$ satisfying either boundary-layer inequality,
    $$
    b(x)\in[b_1,b_2],\qquad \sigma(x)\in[\sigma_1,\sigma_2],
    $$
    as claimed.
\end{proof}
As a side remark, $\sigma_i$ and $b_i$ are different from $\sigma_0 = \sigma(0)$ and $b_0 = b(0)$ up to a term of order $\sqrt{h}$ because we have shown the boundary layer of our scheme is included in set $B_\delta$. 
Lemma \ref{lemma:bound_on_drift_diffusion} is useful in the sense that it holds when we want to do error analysis for our scheme whenever we do not make a standard EM jump.

\medskip

\noindent\textbf{Lemma~\ref{lemma:bound_on_lambda_diffusion}.}
    There exists $\delta>0$ such that $\forall h < \delta$, $\forall x$ such that $x - \sigma(x)\sqrt{3h} < 0,$
    $$\lambda(x) \in [0,1]. \;$$
\begin{proof}
    By Lemma \ref{lemma:bound_on_drift_diffusion}, there exist $\delta_1>0$ and constants
    $b_1,b_2,\sigma_1,\sigma_2$, depending only on the local boundary constants and the linear-growth constant of $b,\sigma$, such that for every $h<\delta_1$ and every $x$ with
    $x-\sigma(x)\sqrt{3h}<0$, we have
    $$
    b(x)\in [b_1,b_2], \qquad \sigma(x)\in [\sigma_1,\sigma_2].
    $$
    Hence
    \[
    x<\sigma(x)\sqrt{3h}\le \sigma_2\sqrt{3h}.
    \]
    Therefore it suffices to prove the following algebraic statement: for all
    $(x,b,\sigma)\in[0,\sigma_2\sqrt{3h}]\times[b_1,b_2]\times[\sigma_1,\sigma_2]$,
    the quantity $\Lambda(x,b,\sigma)$ defined below \eqref{eqn:Lambda} satisfies
    $$
    0\le \Lambda(x,b,\sigma)\le 1.
    $$
    If $b_2 \leq 0$, choose $\delta_2 = 1$; if $b_2 > 0$, $\delta_2$ is chosen in such a way that the following inequalities hold for all $h < \delta_2$,
    \begin{equation*}
    \begin{split}
        \frac{\kappa}{\sqrt{3h}} &> \frac{2\kappa b_2 \sigma_2}{\sigma_1^2},\\
        \sqrt{3h} &< \frac{\sigma_2}{2b_2},
    \end{split}
    \end{equation*}
    which implies that for all $b\in [b_1, b_2]$ and $\sigma \in [\sigma_1, \sigma_2]$
    \begin{equation}
    \label{eqn:delta_lambda}
    \begin{split}
        \frac{\kappa}{\sqrt{3h}} &> \frac{(2\kappa b - \sigma^2)\sigma_2}{\sigma^2},\\
        \frac{\kappa \sigma}{\sqrt{3h}} &> -(\sigma^2 - 2\kappa b).\\
    \end{split}
    \end{equation}
    Consider, $\forall h < \min\{\delta_1, \delta_2\}$, 
    \begin{equation}
    \label{eqn:Lambda}
         \Lambda(x, b, \sigma) = 
         \frac{(\sigma^2-2\kappa b)x^2 + 2\kappa \sigma^2 x + \sigma^4h}{\left(\frac{\kappa\sigma}{\sqrt{3h}} + \sigma^2 - 2\kappa b\right)x^2 + \sigma^3 \kappa \sqrt{3h} + \sigma^4 h},
    \end{equation}
    and
    \begin{equation}
    \label{eqn:one_minus_lambda}
        1-\Lambda(x, b, \sigma) =
        \frac{{\frac{\kappa \sigma}{\sqrt{3h}}\left(x-\sigma\sqrt{3h}\right)^2}}{\left(\frac{\kappa\sigma}{\sqrt{3h}} + \sigma^2 - 2\kappa b\right)x^2 + \sigma^3 \kappa \sqrt{3h} + \sigma^4 h}.
    \end{equation}
    It suffices to conclude the lemma if we can show $\forall (x,b, \sigma) \in [0,\sigma_2\sqrt{3h}]\times [b_1, b_2] \times [\sigma_1, \sigma_2]$,
    \begin{equation}
    \label{eqn:sign_of_lambda}
        \Lambda(x, b, \sigma) \geq 0, \quad \text{and} \quad  1-\Lambda(x, b, \sigma) \geq 0.
    \end{equation}
    We consider three cases: (i) $b < \frac{\sigma^2}{2\kappa}$  (ii) $b > \frac{\sigma^2}{2\kappa}$ ;\; (iii) $b = \frac{\sigma^2}{2\kappa}$.
    \begin{enumerate}
        \item For $b < \frac{\sigma^2}{2\kappa}$, we notice that both numerator (denote as $N_1(x,b)$) and denominator (denote as $D(x,b)$) of $\Lambda(x,b)$ are strictly positive.
        Besides, the numerator of $1-\Lambda(x,b)$ is a square and its denominator is identical to $D(x,b)$, so it is strictly positive.
        \eqref{eqn:sign_of_lambda} is verified.
        \item For $b > \frac{\sigma^2}{2\kappa}$, we have $2\kappa b-\sigma^2>0$, and in particular $b_2>0$.
        From the first inequality in \eqref{eqn:delta_lambda}, for every $x\in [0,\sigma_2\sqrt{3h}]$,
        \[
        (2\kappa b-\sigma^2)x
        \le
        (2\kappa b-\sigma^2)\sigma_2\sqrt{3h}
        <
        \kappa\sigma^2.
        \]
        Hence
        \[
        2\kappa\sigma^2-(2\kappa b-\sigma^2)x>\kappa\sigma^2>0.
        \]
        Using this, the numerator of $\Lambda$ satisfies
        \begin{align*}
        N_1(x,b,\sigma)
        &= (\sigma^2-2\kappa b)x^2 + 2\kappa \sigma^2 x + \sigma^4h \\
        &= \sigma^4h + x\bigl(2\kappa \sigma^2-(2\kappa b-\sigma^2)x\bigr) \\
        &> \sigma^4h + \kappa\sigma^2 x \\
        &> 0.
        \end{align*}
        For the denominator, by the second inequality in \eqref{eqn:delta_lambda},
        \[
        \frac{\kappa\sigma}{\sqrt{3h}}+\sigma^2-2\kappa b>0.
        \]
        Therefore
        \begin{align*}
        D(x,b,\sigma)
        &=
        \left(\frac{\kappa\sigma}{\sqrt{3h}}+\sigma^2-2\kappa b\right)x^2
        +\sigma^3\kappa\sqrt{3h}
        +\sigma^4h \\
        &\ge \sigma^3\kappa\sqrt{3h}+\sigma^4h \\
        &>0.
        \end{align*}
        Thus $\Lambda(x,b,\sigma)\ge 0$ on $[0,\sigma_2\sqrt{3h}]\times [b_1,b_2]\times [\sigma_1,\sigma_2]$.
        We also have $1-\Lambda(x,b,\sigma)\ge 0$. 
        Hence \eqref{eqn:sign_of_lambda} holds in this case.
        \item For $b = \frac{\sigma^2}{2\kappa}$, by simple algebra,
        \[
        \Lambda(x,b,\sigma)
        =
        \frac{2\sigma^2x + 2b\sigma^2h}
        {\frac{\sigma x^2}{\sqrt{3h}}+ \sigma^3\sqrt{3h} + 2b\sigma^2 h}
        \ge 0,
        \]

        \[
        1-\Lambda(x,b,\sigma)
        =
        \frac{\frac{\sigma}{\sqrt{3h}}\left(x-\sigma\sqrt{3h}\right)^2}
        {\frac{\sigma x^2}{\sqrt{3h}}+ \sigma^3\sqrt{3h} + 2b\sigma^2 h}
        \ge 0.
        \]
        Hence \eqref{eqn:sign_of_lambda} is verified in this case.
    \end{enumerate}
\end{proof}

\noindent\textbf{Lemma~\ref{lemma:UB_lambda_diffusion}.}
    There exists constant $C, \delta > 0$ dependent on $\sigma(x)$ and $b(x)$ such that $\forall h < \delta$ and $\forall x\in [0, \sigma_2\sqrt{3h}]$ such that $x-\sigma(x)\sqrt{3h} < 0$, 
    $$0 \leq \lambda(x) \leq C\cdot \left(\frac{x}{\sqrt{h}} + \sqrt{h} \right).$$
\begin{proof}
    By Lemma~\ref{lemma:bound_on_drift_diffusion}, after decreasing $\delta$ if necessary, there exist constants
    $b_1,b_2,\sigma_1,\sigma_2$, independent of $x$ and $h$, such that whenever
    $h<\delta$ and $x-\sigma(x)\sqrt{3h}<0$,
    $$b(x)\in[b_1,b_2],\qquad \sigma(x)\in[\sigma_1,\sigma_2],$$
    and hence
    $$ 0\leq x<\sigma(x)\sqrt{3h}\leq \sigma_2\sqrt{3h}.$$
    Let $B_M:=|b_1|\vee |b_2|$, by Lemma~\ref{lemma:bound_on_lambda_diffusion}, after decreasing $\delta$ if necessary, we already know that $\lambda(x)\in [0,1]$.
    It remains to prove the sharper upper bound. Write
    $$\lambda(x)=\frac{N_h(x)}{D_h(x)}, $$
    where
    $$ N_h(x):=(\sigma^2(x)-2\kappa b(x))x^2+2\kappa\sigma^2(x)x+\sigma^4(x)h,$$
    and
    $$
    D_h(x):=
    \left(\frac{\kappa\sigma(x)}{\sqrt{3h}}+\sigma^2(x)-2\kappa b(x)\right)x^2
    +\sigma^3(x)\kappa\sqrt{3h}+\sigma^4(x)h.
    $$
    Decrease $\delta$ once more so that, for all $h<\delta$,
    $$\frac{\kappa\sigma_1}{\sqrt{3h}}\geq 2\kappa B_M.$$
    Then, for all $x$ under consideration,
    $$
    \frac{\kappa\sigma(x)}{\sqrt{3h}}+\sigma^2(x)-2\kappa b(x)
    \geq
    \frac{\kappa\sigma_1}{\sqrt{3h}}-2\kappa B_M
    \geq 0.
    $$
    Therefore
    $$D_h(x)\geq \sigma^3(x)\kappa\sqrt{3h}\geq \sigma_1^3\kappa\sqrt{3h}. $$
    On the other hand,
    $$
    |N_h(x)| \leq
    (\sigma_2^2+2\kappa B_M)x^2 +2\kappa\sigma_2^2x +\sigma_2^4h.
    $$
    Since $x\leq\sigma_2\sqrt{3h}$, we have $x^2\leq 3\sigma_2^2h$. Hence
    \begin{equation}
    \label{eqn:diffusion_lambda_UB}
    \begin{split}
        \lambda(x)
        &\leq
        \frac{|N_h(x)|}{D_h(x)}\\
        &\leq
        \frac{(\sigma_2^2+2\kappa B_M)\sigma_2^2\cdot 3h
        +2\kappa\sigma_2^2x+\sigma_2^4h}
        {\sigma_1^3\kappa\sqrt{3h}}\\
        &=
        \frac{2\sigma_2^2}{\sqrt{3}\sigma_1^3}\cdot \frac{x}{\sqrt h}
        +
        \frac{4\sigma_2^4+6\kappa B_M\sigma_2^2}
        {\sqrt{3}\sigma_1^3\kappa}\cdot \sqrt h.
    \end{split}
    \end{equation}
    Choosing
    $$
    C:=
    \max\left\{
    \frac{2\sigma_2^2}{\sqrt{3}\sigma_1^3},
    \frac{4\sigma_2^4+6\kappa B_M\sigma_2^2}
    {\sqrt{3}\sigma_1^3\kappa}
    \right\}
    $$
    completes the proof.
\end{proof}

\noindent\textbf{Lemma~\ref{lemma:sticky_bc@x_diffusion}.}
    Let $u$ be the solution of \eqref{eqn:diffusion_bckwd_eqn} satisfying corresponding boundary condition \eqref{eqn:L_diffusion_BC} and assume that $b(0) \neq \frac{\sigma^2(0)}{2\kappa}$.
    Then, there exists $\delta$ small such that for all $h<\delta$ and $x\in \Omega_{b}$,
    \begin{equation*}
    \label{eqn:diffusion_BC_consequence}
    \begin{split}
        \partial_x u(t,x) &= \left(x+M(x) \right)\partial_{xx}u(t,x) + Re(x)
    \end{split}
    \end{equation*}
    where
    $$
    |Re(x)|\leq C(h+x)=\mathcal{O}(h+x),
    $$
    and $C$ depends on bounds of $\partial_xu,\partial_{xx}u,u^{(3)}$ and on the local bounds of $b,\sigma$ and their derivatives up to order $2$ on the boundary neighborhood $[0,r_0]$, but is independent of $x$ and $h$.
\begin{proof}
    Similar to proof of Lemma \ref{lemma:sticky_bc@x}, apply Taylor expansion with remainder to \eqref{eqn:L_diffusion_BC} centered at $x$. 
    By Lemma~\ref{lemma:bound_on_drift_diffusion}, after decreasing $\delta$ if necessary, every $x\in\Omega_{b}$ lies in $B_\delta\subset[0,r_0]$. Hence all intermediate points in the Taylor expansions below also lie in $[0,r_0]$, where the local coefficient bounds from Assumption~\ref{assump:local_coeff_boundary_regularity} apply. We have
    \begin{equation*}
    \begin{split}
        \sigma^2(0) \partial_xu(t,0) &= \sigma^2(x)\partial_x u(t,x) - x\sigma^2(x)\partial_{xx}u(t,x) -2x\sigma(x)\sigma'(x)\partial_xu(t,x) \\&\qquad+ \frac{x^2}{2}\partial_{xx}\left(\sigma^2(y_{1,x})\partial_x u(t,y_{1,x})\right)\\
        &= \sigma^2(x)\partial_x u(t,x) - x\sigma^2(x)\partial_{xx}u(t,x) + \mathcal{O}(h+x)
    \end{split}
    \end{equation*}
    and
    \begin{equation*}
    \begin{split}
        &\quad \quad 2\kappa b(0)\partial_x u(t,0) + \kappa \sigma^2(0)\partial_{xx}u(t,0)\\ &= 2\kappa \left(b\cdot\partial_xu - x\partial_x\left(b\cdot \partial_x u\right) + \frac{x^2}{2}\partial_{xx}\left(b\cdot \partial_x u\right)\right)\\
        &\quad + \kappa \left(\sigma^2\cdot \partial_{xx}u - x\cdot\partial_x\left(\sigma^2\cdot \partial_{xx}u\right)\right)\\
        &= 2\kappa b(x)\partial_x u(t,x) + \left(- 2x\kappa b(x) + \kappa \sigma^2(x)\right)\partial_{xx}u(t,x) + \mathcal{O}(h+x)
    \end{split}
    \end{equation*}
    where we omit function arguments in expansion above.
    We choose $\delta$ small such that $\sigma^2(x) - 2\kappa b(x) \neq 0$ for all $x\in [0,\sigma_2\sqrt{3h}]$ because we assume drift and diffusion terms to be continuous and $b(0) \neq \frac{\sigma^2(0)}{2\kappa}$.
    This allows us to conclude that $(\sigma^2(x) - 2\kappa b(x))^{-1}$ to be bounded for $x\in [0,\sigma_2\sqrt{3h}]$.
    This allows us to finish proving the lemma by equating two expressions above and rearranging terms.
\end{proof}

\noindent\textbf{Lemma~\ref{lemma:1st_mom_diffusion}.}
There exists $\delta>0$ such that $\forall h < \delta$ and $\forall x\in \Omega_{b}$, we have, 
\begin{equation*}
    \label{eqn:first_mom_negative_drifted_diffusion_approx}
    \begin{split}
        \mathbb{E}_x(\Delta X_{k}) &= \frac{\lambda(x)x^2}{2\sqrt{3h} \cdot \sigma(x)} + \frac{\sigma(x) \lambda(x)\sqrt{3h}}{2} + \lambda(x)b(x)h - x + \mathcal{O}(hx + h^2).\\
    \end{split}
\end{equation*}
where
    $$R_1(x) = \frac{\lambda(x)b(x)^2h^2}{\sigma(x)\sqrt{3h}} = \mathcal{O}(hx + h^2)$$
\begin{proof}
    We first calculate $\mathbb{E}_x(X_{k+1})$ where $b(x)<0$.
    According to \eqref{eqn:diffusion_cont_scheme_inlayer}, given $X_k=x$, $X_{k+1}$ can take four different forms
    \begin{enumerate}
        \item If $x + \sigma(x)Z_{k+1} \geq 0$ and $x + \sigma(x)Z_{k+1} + b(x)h \geq 0$, then 
        $$X_{k+1} = x + \sigma(x)Z_{k+1} + b(x)h \;\;\; \text{where}\;\;\; Z_{k+1} \geq -\frac{1}{\sigma(x)}\left(b(x)h+x\right)$$
        \item If $x + \sigma(x)Z_{k+1} \geq 0$ and $x + \sigma(x)Z_{k+1} + b(x)h \leq 0$, then 
        $$X_{k+1} = -x - \sigma(x)Z_{k+1} - b(x)h \;\;\; \text{where}\;\;\; -\frac{x}{\sigma(x)}\leq Z_{k+1} \leq -\frac{1}{\sigma(x)}\left(b(x)h+x\right)$$
        \item If $x + \sigma(x)Z_{k+1} \leq 0$ and $x + \sigma(x)Z_{k+1} + b(x)h \leq 0$, then 
        $$X_{k+1} = x + \sigma(x)Z_{k+1} - b(x)h \;\;\; \text{where}\;\;\; \frac{1}{\sigma(x)}\left(b(x)h-x\right)\leq Z_{k+1} \leq -\frac{x}{\sigma(x)}$$
        \item If $x + \sigma(x)Z_{k+1} \leq 0$ and $x + \sigma(x)Z_{k+1} + b(x)h \geq 0$, then 
        $$X_{k+1} = -x - \sigma(x)Z_{k+1} + b(x)h \;\;\; \text{where}\;\;\; Z_{k+1} \leq \frac{1}{\sigma(x)}\left(b(x)h-x\right)$$
    \end{enumerate}
    Therefore,
    \begin{equation}
    \label{eqn:1st_mom_Xk+1_negative_drift_diffusion}
    \begin{split}
        \mathbb{E}_x(X_{k+1}) &= \lambda(x)\mathbb{E}_x \left( \left| \left|x + \sigma(x)Z_{k+1}\right| + b(x)h \right| \right) + (1-\lambda(x))\cdot 0\\
        &= \int_{-\frac{1}{\sigma(x)}\left(b(x)h+x\right)}^{\sqrt{3h}} (x+\sigma(x)y+b(x)h)\cdot \frac{\lambda(x)}{2\sqrt{3h}}\;dy\\
        &\quad +  \int_{-\frac{x}{\sigma(x)}}^{-\frac{1}{\sigma(x)}\left(b(x)h+x\right)} (-x-\sigma(x)y-b(x)h)\cdot \frac{\lambda(x)}{2\sqrt{3h}}\;dy\\
        &\quad +  \int_{\frac{1}{\sigma(x)}\left(b(x)h-x\right)}^{-\frac{x}{\sigma(x)}} (x+\sigma(x)y-b(x)h)\cdot \frac{\lambda(x)}{2\sqrt{3h}}\;dy\\
        &\quad +  \int_{-\sqrt{3h}}^{\frac{1}{\sigma(x)}\left(b(x)h-x\right)} (-x-\sigma(x)y+b(x)h)\cdot \frac{\lambda(x)}{2\sqrt{3h}}\;dy \\
        &= \frac{\lambda(x)x^2}{2\sqrt{3h} \cdot \sigma(x)} + \frac{\sigma(x) \lambda(x)\sqrt{3h}}{2} + \lambda(x)b(x)h + \frac{\lambda(x)b(x)^2h^2}{\sigma(x)\sqrt{3h}}.
    \end{split}
    \end{equation}
    This gives that
    \begin{equation}
    \label{eqn:first_mom_negative_drifted_diffusion}
    \begin{split}
        \mathbb{E}_x(\Delta X_{k}) &= \frac{\lambda(x)x^2}{2\sqrt{3h} \cdot \sigma(x)} + \frac{\sigma(x) \lambda(x)\sqrt{3h}}{2} + \lambda(x)b(x)h - x + \frac{\lambda(x)b(x)^2h^2}{\sigma(x)\sqrt{3h}}.\\
    \end{split}
    \end{equation}
    When the drift term $b(x)$ is positive calculation is simpler (as there is only absolute value on $x+\sigma(x)Z_{k+1}$), we can verify that:
    \begin{equation}
    \label{eqn:first_mom_positive_drifted_diffusion}
        \mathbb{E}_x(\Delta X_{k}) = \frac{\lambda(x)x^2}{2\sqrt{3h} \cdot \sigma(x)} + \frac{\sigma(x) \lambda(x)\sqrt{3h}}{2} + \lambda(x)b(x)h - x.
    \end{equation}
    In other word, the proof is complete if the $b(x) \geq 0$ and $R_1 = 0$.
    Back to the case where $b(x) < 0$, the first conditional moment of $\Delta X_k$ is
    \begin{equation*}
        \mathbb{E}_x(\Delta X_k) = \eqref{eqn:first_mom_positive_drifted_diffusion} + \frac{\lambda(x)b(x)^2h^2}{\sigma(x)\sqrt{3h}},
    \end{equation*}
    and the remainder
    \begin{equation}
    \label{eqn:first_mom_diffusion_remainder}
        R_1(x) = \frac{\lambda(x)b(x)^2h^2}{\sigma(x)\sqrt{3h}} = \mathcal{O}(hx + h^2)
    \end{equation}
    because of Lemma \ref{lemma:UB_lambda_diffusion} and the bounds
    $b(x)\in[b_1,b_2]$, $\sigma(x)\geq\sigma_1>0$ from Lemma~\ref{lemma:bound_on_drift_diffusion}.
\end{proof}

\noindent\textbf{Lemma~\ref{lemma:2nd_mom_order_negative_drifted_diffusion}.}
There exists $\delta>0$ and $C > 0$ such that $\forall h < \delta$ and $\forall x\in \Omega_{b}$, we have,
\begin{equation*}
    \mathbb{E}_x(\Delta X_k^2) = -\frac{\lambda(x)}{\sigma(x)\sqrt{3h}}\cdot x^3 + (1+\lambda(x))x^2 - \lambda(x)\sigma(x)\sqrt{3h}\cdot x + \lambda(x)\sigma^2(x)\cdot h + R_2(x)
\end{equation*}
    where
    $$|R_2(x)| \leq C\cdot (hx + h^2) = \mathcal{O}(hx + h^2).$$
\begin{proof}
    Again, by Lemma~\ref{lemma:bound_on_drift_diffusion}, after decreasing $\delta$ if necessary, every $x\in\Omega_{b}$ lies in $B_\delta\subset[0,r_0]$.
    When drift is negative, recycle part of calculation in \eqref{eqn:1st_mom_Xk+1_negative_drift_diffusion} and separate out remainder terms of order greater that $3/2$, we have: (when drift is positive, $\frac{b(x)^2h^2}{\sigma(x)\sqrt{3h}}$ will be removed from $R_2(x)$, but we still reach same conclusion)
    \begin{equation*}
    \begin{split}
       & \mathbb{E}_x(\Delta X_k^2) \\
        &= \lambda(x)\mathbb{E}_x\left(\left(\left| \left|x + \sigma(x) Z_{k+1}\right| + b(x)h \right| - x\right)^2\right) + (1-\lambda(x))x^2\\
        &= \lambda(x)\mathbb{E}_x\left( \left( \left|x + \sigma(x) Z_{k+1}\right| + b(x)h \right)^2 - 2x\left| \left|x + \sigma(x)Z_{k+1}\right| + b(x)h \right| + x^2 \right) + (1-\lambda(x))x^2\\
        &= \lambda(x)\mathbb{E}_x\left( \left|x + \sigma(x) Z_{k+1}\right|^2 - 2x \left| \left|x + \sigma(x) Z_{k+1}\right| + b(x)h \right| + x^2 \right) + (1-\lambda(x))x^2 + R_2'(x)\\
        &= \lambda(x) \cdot \left( x^2 + \sigma^2(x) h - 2x \left( \frac{x^2}{2 \sigma(x) \sqrt{3h}} + \frac{\sigma(x)\sqrt{3h}}{2} + b(x)h + \frac{b(x)^2h^2}{\sigma(x)\sqrt{3h}}\right) + x^2 \right)\\
        & \quad + (1-\lambda(x))x^2 + R_2'(x)\\
        &= \lambda(x) \cdot \left( 2x^2 + \sigma^2(x)h - 2x \left( \frac{x^2}{2\sigma(x)\sqrt{3h}} + \frac{\sigma(x)\sqrt{3h}}{2} \right) \right)\\
        &\quad + (1-\lambda(x))x^2 - 2x\lambda(x)\left(b(x)h + \frac{b(x)^2h^2}{\sigma(x)\sqrt{3h}}\right) + R_2'(x)\\
        &= -\frac{\lambda(x)}{\sigma(x)\sqrt{3h}}\cdot x^3 + (1+\lambda(x))x^2 - \lambda(x)\sigma(x)\sqrt{3h}\cdot x + \lambda(x)\sigma^2(x)\cdot h  + R_2(x)
    \end{split}
    \end{equation*}
    where
    $$R_2'(x) = \lambda(x)\mathbb{E}_x\left(b(x)^2h^2 + 2b(x)h\left|x+\sigma(x)Z_{k+1}\right|\right)$$
    and
    \begin{equation*}
    \begin{split}
        R_2(x) &= -2x\lambda(x)\left(b(x)h + \frac{b(x)^2h^2}{\sigma(x)\sqrt{3h}}\right) + R_2'(x)\\
        &= 2x\lambda(x)\left(b(x)h + \frac{b(x)^2h^2}{\sigma(x)\sqrt{3h}}\right) + \lambda(x)\mathbb{E}_x\left(b(x)^2h^2 + 2b(x)h\left|x+\sigma(x)Z_{k+1}\right|\right).
    \end{split}
    \end{equation*}
    We can bound $R_2(\cdot)$ using the local boundary-layer bounds from Lemma~\ref{lemma:bound_on_drift_diffusion}. Recall that $B_M:=|b_1|\vee |b_2|$ and $\sigma(x)\geq\sigma_1>0$ on the boundary layer:
    $$|R_2(x)|\leq 2B_Mhx + \frac{2B_M^2h^2x}{\sigma_1\sqrt{3h}} + B_M^2h^2 + 2B_Mh\mathbb{E}_x\left(\left|x+\sigma(x)Z_{k+1}\right|\right)\lambda(x).$$
    Hence, it suffices to complete the proof by showing that $h\mathbb{E}_x\left(\left|x+\sigma(x)Z_{k+1}\right|\right)\lambda(x) = \mathcal{O}(hx+h^2)$.
    By Lemma \ref{lemma:UB_lambda_diffusion}, and using
    $x\leq \sigma_2\sqrt{3h}$ and $|Z_{k+1}|\leq \sqrt{3h}$, we have
    \begin{equation*}
    \begin{split}
        h\mathbb{E}_x\left(\left|x+\sigma(x)Z_{k+1}\right|\right)|\lambda(x)|
        &\leq h\left(x+\sigma_2\sqrt{3h}\right)|\lambda(x)|\\
        &\leq C h^{3/2}\left(\frac{x}{\sqrt h}+\sqrt h\right)\\
        &\leq C(hx+h^2).
    \end{split}
    \end{equation*}
\end{proof}

\noindent\textbf{Lemma~\ref{lemma:1st_2nd_mom_combination_generator_diffusion}.}
Assume that $b(0) \neq \frac{\sigma^2(0)}{2\kappa}$, there exist $\delta>0$ such that $\forall h < \delta$ and $\forall x\in \Omega_{b}$
    \begin{equation}
    \label{eqn:1st_2nd_mom_goal_cont_scheme_drifted}
        2\left(\mathbb{E}_x(\Delta X_k) - hb(x)\right) \cdot (x + M(x)) 
        + \mathbb{E}_x(\Delta X_k^2) = \sigma^2(x)h + R_3(x) 
    \end{equation}
    where $|R_3(x)| \leq C(hx+h^2)=\mathcal{O}(hx+h^2)$ for some constant $C$ independent of $x$ and $h$. The constant depends only on the boundary-layer bounds $b_1,b_2,\sigma_1,\sigma_2$, $\kappa$, and the lower bound of $|\sigma^2(x)-2\kappa b(x)|$ in a sufficiently small neighborhood of $0$.

\begin{proof}
Let
$$\mathbb E_x(\Delta X_k)=A_h(x)+R_1(x),$$
where $A_h(x)$ denotes the principal term from
\eqref{eqn:first_mom_positive_drifted_diffusion}. By Lemma~\ref{lemma:UB_lambda_diffusion},
\begin{equation}\label{eq:R1_bound_rewrite}
|R_1(x)| \le C_1(hx+h^2)
\end{equation}
for some constant $C_1$ independent of $h$.

Likewise, by Lemma~\ref{lemma:2nd_mom_order_negative_drifted_diffusion},
\[
\mathbb E_x(\Delta X_k^2)=B_h(x)+R_2(x),
\]
where
\[
B_h(x)
=
-\frac{\lambda(x)}{\sigma(x)\sqrt{3h}}x^3
+(1+\lambda(x))x^2
-\lambda(x)\sigma(x)\sqrt{3h}\,x
+\lambda(x)\sigma^2(x)h,
\]
and
\begin{equation}\label{eq:R2_bound_rewrite}
|R_2(x)| \le C_2(hx+h^2)
\end{equation}
for some constant \(C_2\) independent of \(h\). We now substitute these expansions into the quantity of interest:
\[
2\bigl(\mathbb E_x(\Delta X_k)-hb(x)\bigr)(x+M(x))
+\mathbb E_x(\Delta X_k^2),
\]
where
\[
M(x)=\frac{\kappa \sigma^2(x)}{\sigma^2(x)-2\kappa b(x)}.
\]
This gives
\begin{equation}\label{eq:split_main_remainder}
\begin{aligned}
&2\bigl(\mathbb E_x(\Delta X_k)-hb(x)\bigr)(x+M(x))
+\mathbb E_x(\Delta X_k^2) \\
&\qquad=
\underbrace{2\bigl(A_h(x)-hb(x)\bigr)(x+M(x))+B_h(x)}_{=: \mathcal A_h(x)}
+2R_1(x)(x+M(x))+R_2(x).
\end{aligned}
\end{equation}

A direct rearrangement of \(\mathcal A_h(x)\) yields
\begin{equation}\label{eq:Ah_rearranged}
\begin{aligned}
\mathcal A_h(x)
&=
\lambda(x)\left(
\left(\frac{M(x)}{\sigma(x)\sqrt{3h}}+1\right)x^2
+M(x)\sigma(x)\sqrt{3h}
+\sigma^2(x)h
+2b(x)M(x)h
\right) \\
&\quad
-x^2-2M(x)x-2b(x)M(x)h
+2x(\lambda(x)-1)b(x)h.
\end{aligned}
\end{equation}
By the definition \eqref{eqn:diffusion_lambda_cont_scheme} of \(\lambda(x)\),
the bracketed terms are chosen precisely so that
\begin{multline*}
\lambda(x)\left(
\left(\frac{M(x)}{\sigma(x)\sqrt{3h}}+1\right)x^2
+M(x)\sigma(x)\sqrt{3h}
+\sigma^2(x)h
+2b(x)M(x)h
\right)\\
-x^2-2M(x)x-2b(x)M(x)h
=
\sigma^2(x)h.
\end{multline*}
Hence
\begin{equation}\label{eq:Ah_identity}
\mathcal A_h(x)=\sigma^2(x)h+2x(\lambda(x)-1)b(x)h.
\end{equation}

Substituting \eqref{eq:Ah_identity} into \eqref{eq:split_main_remainder}, we obtain
\[
2\bigl(\mathbb E_x(\Delta X_k)-hb(x)\bigr)(x+M(x))
+\mathbb E_x(\Delta X_k^2)
=
\sigma^2(x)h+R_3(x),
\]
where
\begin{equation}\label{eq:R3_definition_rewrite}
R_3(x):=
2x(\lambda(x)-1)b(x)h
+2R_1(x)(x+M(x))
+R_2(x).
\end{equation}

It remains to bound $R_3(x)$. Since $b(0)\neq \sigma^2(0)/(2\kappa)$, continuity implies that,
for sufficiently small \(\delta>0\), the quantity
$$\sigma^2(x)-2\kappa b(x)$$
is bounded away from $0$ on $x\in[0,\sigma_2\sqrt{3h}]$ whenever $h<\delta$. Therefore $M(x)$ is uniformly bounded on this region.
Using also the boundary-layer bound $b(x)\in[b_1,b_2]$, the fact that
$|\lambda(x)-1|\le 1$, the remainder bounds \eqref{eq:R1_bound_rewrite}--\eqref{eq:R2_bound_rewrite},
and $x\in[0,\sigma_2\sqrt{3h}]$, we conclude that
$$|R_3(x)| \le C(hx+h^2)$$
for some constant $C$ independent of $h$. This proves the claim.
\end{proof}

\noindent\textbf{Lemma~\ref{lemma:1st_mom_order_negative_drifted_diffusion}.}
    Assume that $b(0) \neq \frac{\sigma^2(0)}{2\kappa}$, there exists $\delta>0$ such that $\forall h < \delta$ and $\forall x\in \Omega_{b}$, 
     \begin{equation}
     \label{eqn:first_mom_negative_drifted_bound_diffusion}
        \left|\mathbb{E}_x(\Delta X_k)\right| \leq C(h + hx+h^2),
     \end{equation}
     where $C>0$ is independent of $x$ and $h$ and depends only on $\kappa$, the local boundary bounds for $b,\sigma$, and the nondegeneracy constant coming from $b(0)\neq \sigma^2(0)/(2\kappa)$.

\begin{proof}
    By Lemma~\ref{lemma:bound_on_drift_diffusion}, after decreasing $\delta$ if necessary, all $x\in[0,\sigma_2\sqrt{3h}]$ lie in the fixed boundary neighborhood where the local coefficient bounds apply. Since
    $$
    b(0)\neq \frac{\sigma^2(0)}{2\kappa},
    $$
    continuity implies that, after decreasing $\delta$ if necessary, there exists $c_0>0$ such that
    $$
    |\sigma^2(x)-2\kappa b(x)|\ge c_0,
    \qquad x\in[0,\sigma_2\sqrt{3h}].
    $$
    Hence $M(x)$ is uniformly bounded on this region. Moreover, since
    $$
    M(0)=\frac{\kappa\sigma^2(0)}{\sigma^2(0)-2\kappa b(0)}\neq 0,
    $$
    continuity also implies that there exists $m_0>0$ such that
    $$
    |M(x)|\ge m_0,\qquad x\in[0,\sigma_2\sqrt{3h}]
    $$
    for all sufficiently small $h$. Decreasing $\delta$ once more so that
    $$
    \sigma_2\sqrt{3h}\le \frac{m_0}{2}, \qquad h<\delta,
    $$
    we obtain
    $$
    |x+M(x)|\ge |M(x)|-x\ge \frac{m_0}{2},
    \qquad x\in[0,\sigma_2\sqrt{3h}].
    $$
    We derive from remainder estimation in Lemma~\ref{lemma:1st_2nd_mom_combination_generator_diffusion} and the crude bound $\mathbb{E}_x(\Delta X_k^2)\le Ch$  that
    \begin{equation*}
    \begin{split}
        |\mathbb{E}_x(\Delta X_k)|
        &= \left|\frac{\sigma^2(x)h + R_3(x) - \mathbb{E}_x(\Delta X_k^2)}
        {2\left(x+M(x)\right)} + hb(x)\right|\\
        &\leq \frac{\sigma_2^2 h + |R_3(x)| + (\sigma_2 \sqrt{3h} + B_Mh)^2}{m_0}
        + B_M h\\
        &\leq C(h+hx+h^2).
    \end{split}
    \end{equation*}
\end{proof}

\noindent\textbf{Lemma~\ref{lemma:3rd_mom_order_negative_drifted_diffusion}.}
    There exists $\delta>0$ such that $\forall h < \delta$ and $\forall x\in \Omega_{br}$,
     \begin{equation}
     \label{eqn:first_mom_drifted_bound}
        \left|\mathbb{E}_x(\Delta X_k^3)\right| \leq K_4(hx + h^2),
     \end{equation}
     where $K_4$ is a fixed constant dependent on drift, diffusion, and $\kappa$.
\begin{proof}
    By Lemma~\ref{lemma:bound_on_drift_diffusion}, after decreasing $\delta$ if necessary, all $x\in[0,\sigma_2\sqrt{3h}]$ lie in the fixed boundary neighborhood where the local coefficient bounds apply.
    We begin with the case where $x-\sigma(x)\sqrt{3h} \leq 0$.
    
    Similar to proof of Lemma~\ref{lemma:3rd_mom_bound}, fix $x\in[0,\sigma_2\sqrt{3h}]$, an increment $\Delta X_k$ is supported on $[-x,\sigma_2\sqrt{3h} + B_Mh]$.
    We decompose $\Delta X_k$ into its positive and negative parts, so that $|\Delta X_k|^3=(\Delta X_k^+)^3+(\Delta X_k^-)^3$ and therefore
    \[
    \mathbb{E}_x(|\Delta X_k|^3)
    =\mathbb{E}_x[(\Delta X_k^+)^3]+\mathbb{E}_x[(\Delta X_k^-)^3].
    \]
    
    \medskip\noindent
    \textit{Positive part.}
    Whenever $\Delta X_k>0$ the increment comes from the continuous branch $Z_k$ in Algorithm \ref{alg:diffusion_cont_scheme}.
    This means the largest positive increment is $\sigma(x)\sqrt{3h} + b(x)h$, hence
    \begin{equation}\label{eq:3nd_mom_pospart_diffusion}
    \mathbb{E}_x[(\Delta X_k^+)^3]\le \lambda(x)\,(\sigma_2\sqrt{3h} + B_Mh)^{3}.
    \end{equation}
    Recall by Lemma~\ref{lemma:UB_lambda_diffusion}, for $x\in[0,\sqrt{3h}]$, we have the following bound on $\lambda(x)$, 
    $$|\lambda(x)| \leq \frac{2 \sigma^2_2}{\sqrt{3}\sigma_1^3}\cdot \frac{x}{\sqrt{h}} + \frac{(4\sigma_2^4+6\kappa B_M\sigma_2^2)}{\sqrt{3}\sigma_1^3\kappa}\cdot \sqrt{h} \leq S_1\left(\frac{x}{\sqrt{h}} + \sqrt{h} \right).$$
    Combining with \eqref{eq:3nd_mom_pospart_diffusion} gives
    \begin{equation*}
    \begin{split}
        \mathbb{E}_x[(\Delta X_k^+)^3] &\leq S_1\left(\frac{x}{\sqrt{h}} + \sqrt{h} \right)(\sigma_2\sqrt{3h} + B_Mh)^{3}\\
        &\leq S_1\left( \sqrt{3}\sigma_2x + B_M\sqrt{h}x + \sqrt{3}\sigma_2 h + B_M h^{3/2} \right)(\sigma_2\sqrt{3h} + B_Mh)^{2}.
    \end{split}
    \end{equation*}
    In other word, exist constant $K_4^+ > 0$ that
    \begin{equation}\label{eq:3nd_mom_pos-final_diffusion}
    \mathbb{E}_x[(\Delta X_k^+)^3]\le K^+_4 (hx + h^2).
    \end{equation}
    
    \medskip\noindent
    \textit{Negative part.}
    Because $\Delta X_k\ge -x$, we have $0\le \Delta X_k^-\le x$, and therefore
    \[
    (\Delta X_k^-)^3 \le x^3.
    \]
    Since $x\in [0,\sigma_2\sqrt{3h}]$,
    \begin{equation}\label{eq:3nd_mom_neg-final_diffusion}
    \mathbb{E}_x[(\Delta X_k^-)^3] \le x^3 \le 3\sigma_2^2hx.
    \end{equation}
    
    \medskip\noindent
    Combining \eqref{eq:3nd_mom_pos-final_diffusion} and \eqref{eq:3nd_mom_neg-final_diffusion} gives
    \[
    \mathbb{E}_x(|\Delta X_k|^3)
    \;\le\;
    K_4^+ h^2 + (K_4^+ + 3\sigma_2^2)hx.
    \]
    Since $|\mathbb{E}_x(\Delta X_k^3)|\le \mathbb{E}_x(|\Delta X_k|^3)$, the same bound holds for
    $|\mathbb{E}_x(\Delta X_k^3)|$ with $K_4=K_4^+ + 3\sigma_2^2$.  
    We have shown the lemma for $x\in\Omega_{b}$. It remains to consider $x\in\Omega_{r}$. In this region the scheme performs the reflected Euler--Maruyama update
    $$
    X_{k+1}=\left|x+\sigma(x)Z_{k+1}+b(x)h\right|.
    $$
    Let $\Delta X_k^{\mathrm{EM}}:=\sigma(x)Z_{k+1}+b(x)h$ denote the unreflected Euler--Maruyama increment. By symmetry of $Z_{k+1}$ and the boundary-layer bounds on $b$ and $\sigma$,
    $$
    \left|\mathbb{E}_x\left[\left(\Delta X_k^{\mathrm{EM}}\right)^3\right]\right|
    =
    \left|3b(x)\sigma^2(x)h^2+b^3(x)h^3\right|
    \leq Ch^2.
    $$
    The reflected and unreflected increments differ only on
    $$
    A:=\left\{x+\sigma(x)Z_{k+1}+b(x)h<0\right\}.
    $$
    Since $x\in\Omega_{r}$, Lemma~\ref{lemma:bound_on_drift_diffusion} gives
    $x=\mathcal{O}(\sqrt h)$, and the interval of $Z_{k+1}$ values producing reflection has length $\mathcal{O}(h)$. Because $Z_{k+1}$ is uniform on an interval of length $2\sqrt{3h}$, we have
    $$
    \mathbb{P}_x(A)=\mathcal{O}(\sqrt h).
    $$
    On $A$, both $|\Delta X_k|$ and $|\Delta X_k^{\mathrm{EM}}|$ are $\mathcal{O}(\sqrt h)$, so
    $$
    \left|\mathbb{E}_x\left[\Delta X_k^3-\left(\Delta X_k^{\mathrm{EM}}\right)^3\right]\right|
    \leq C h^{3/2}\mathbb{P}_x(A)
    =\mathcal{O}(h^2).
    $$
    Hence
    $$
    |\mathbb{E}_x(\Delta X_k^3)|\leq K_4'h^2\leq K_4'(hx+h^2),
    $$
    and the proof is complete.

\end{proof}

\noindent\textbf{Theorem~\ref{thm:local_err_cont_scheme_diffusion}.}
    Given $b(0) \neq \frac{\sigma^2(0)}{2\kappa}$, there exists $\delta > 0$ and $C\geq 0$ depending on the relevant bounds of derivatives of $u$ and on the local boundary constants for $b,\sigma$ such that $\forall h < \delta$ and $\forall x\in \Omega_{b}$,
    \begin{equation}
    \label{eqn:local_err_cont_scheme_drifted}
        \left|\mathbb{E}\left(u_{k+1}-u_k|X_k=x\right)\right| \leq C(h^2 + hx) = \mathcal{O}(hx + h^2) 
    \end{equation}
\begin{proof}
    We choose $\delta$ small enough so that all lemmas we have proved in this section hold for $h < \delta$.
    In particular, by Lemma~\ref{lemma:bound_on_drift_diffusion}, for all $h<\delta$ and all $x\in\Omega_{b}$,
    $$
    b(x)\in[b_1,b_2],\qquad \sigma(x)\in[\sigma_1,\sigma_2].
    $$
    These constants are local boundary-layer constants and are independent of $x$ and $h$.
    Let $u$ be the solution to the backward equation \eqref{eqn:diffusion_bckwd_eqn}.
    We perform a Taylor expansion of $u$ to 3rd order centered at $(t_k, X_k)$. We write $u_k=u(t_k,X_k), \partial_xu_k = \partial_xu(t_k,X_k)$ (and similarly for $\partial_{xx}u_k, \partial_{xxx}u_k$), and we write $Y_k$ for a point such that $Y_k\in[X_k,X_{k+1}]$ or $Y_k\in[X_{k+1},X_{k}]$, and $s_k$ will be a point in interval $s_k\in[kh, (k+1)h]$). Expanding gives
    \begin{equation}
    \label{eq:taylor4_diffusion}
    \begin{split}
        u_{k+1} - u_k
        &= \partial_t u_k \cdot h  +  \partial_xu_k\cdot \Delta X_k + \frac{1}{2}\partial_{xx}u_k\cdot \Delta X_k^2\\
        &\quad + \frac{1}{2}\partial_{tt} u(s_k, X_k) \cdot h^2 + \frac{1}{6}\partial_{xxx}u(t_k, Y_{k})\cdot \Delta X_k^3.
    \end{split}
    \end{equation}
    We are interested in $\E_x(u_{k+1}-u_k)$ and know $x\in \Omega_{b}$. We now bound the expectation of each of the terms in the Taylor-expansion \eqref{eq:taylor4_diffusion}. 
    We have 
    \begin{align*}
    \E_x|\partial_{tt} u(s_k, X_k) \cdot h^2| & \leq h^2\|\partial_{tt}u\|_{\infty}\\
    \E_x|\partial_{xxx}u(t_k, Y_{k})\cdot \Delta X_k^3| &\leq C(h^2+hx)\|u^{(3)}\|_\infty,
    %\E_x|\partial_{xxxx}u(t_k, Y_{k})\cdot \Delta X_k^4|&\leq \|u^{(4)}\|_{\infty}h^2,
    \end{align*}
    where the second inequality follows from Lemma \ref{lemma:3rd_mom_order_negative_drifted_diffusion}.
    We also have, using the backward equation \eqref{eqn:diffusion_bckwd_eqn} and Lemma \ref{lemma:sticky_bc@x_diffusion},
    \begin{equation}
        \begin{split}
           & \E_x\left[\partial_t u_k\cdot h + \partial_xu_k\cdot \Delta X_k + \frac{1}{2}\partial_{xx}u_k\cdot \Delta X_k^2 \right]
            \\
            &\quad =
            \partial_xu(t_k,x)\left(\mathbb{E}_x(\Delta X_k)-b(x)h\right)
            +\frac{1}{2}\partial_{xx}u(t_k,x)
            \left(\mathbb{E}_x(\Delta X_k^2)-\sigma^2(x)h\right)
            \\
            &\quad =
            \left[
            \left(x+M(x)\right)\left(\mathbb{E}_x(\Delta X_k)-b(x)h\right)
            +\frac{1}{2}\left(\mathbb{E}_x(\Delta X_k^2)-\sigma^2(x)h\right)
            \right]\partial_{xx}u(t_k,x)
            \\
            &\qquad
            +\left(\mathbb{E}_x(\Delta X_k)-b(x)h\right)Re(x)
            \\
            &\quad =
            \frac{1}{2}\partial_{xx}u(t_k,x)R_3(x)
            +\left(\mathbb{E}_x(\Delta X_k)-b(x)h\right)Re(x).
        \end{split}
    \end{equation}
    The last step uses Lemma~\ref{lemma:1st_2nd_mom_combination_generator_diffusion}.
    Therefore, applying the triangle inequality to \eqref{eq:taylor4_diffusion}, and then using Lemma \ref{lemma:1st_mom_order_negative_drifted_diffusion} and remainder estimates in Lemma~\ref{lemma:sticky_bc@x_diffusion}, \ref{lemma:1st_2nd_mom_combination_generator_diffusion} shows that there exist constants $M_1,M_2$, independent of $x,t,h$, such that 
    % \begin{equation}\label{eqn:local_err4}
    \begin{align*}
         \left|\mathbb{E}\left(u_{k+1}-u_k|X_k=x\right)\right| 
         & \leq \frac{1}{2}\|\partial_{xx}u\|_{\infty} \cdot |R_3(x)| + \left(\left|\mathbb{E}_x(\Delta X_k)\right| + B_Mh\right)\cdot |Re(x)| \\
        &\qquad+ h^2\|\partial_{tt}u\|_{\infty}
        + C(h^2+hx)\|u^{(3)}\|_\infty\\
        %+ \|u^{(4)}\|_{\infty}h^2\\
        &\leq M_1 h^2 + M_2 hx 
    \end{align*}
    %\end{equation}
    where $B_M:=|b_1|\vee |b_2|.$
\end{proof}

\noindent\textbf{Lemma~\ref{lemma:first_mom_est_reflection_layer_diffusion}.}
   There exists $\delta > 0$ and constant $C$ which depends only on the boundary-layer bounds such that $\forall h < \delta$ and $\forall x\in \Omega_{r}$,
    $$|E_x(\Delta X_k) - b(x)h| \leq C( hx + h^2) = \mathcal{O}(hx + h^2)$$
\begin{proof}
Recall that
\[
\Omega_{r}:=\left\{x\ge \sigma(x)\sqrt{3h}:\ x-\sigma(x)\sqrt{3h}+b(x)h\le 0\right\}.
\]
For $x\in \Omega_{r}$ the scheme performs a reflected Euler--Maruyama step,
\[
X_{k+1}=\left|x+\sigma(x)Z_{k+1}+b(x)h\right|,
\qquad Z_{k+1}\sim \mathrm{Unif}[-\sqrt{3h},\sqrt{3h}].
\]
Hence
\begin{multline*}
\mathbb{E}_x(X_{k+1})
=\\
\frac{1}{2\sqrt{3h}}
\int_{-\frac{x+b(x)h}{\sigma(x)}}^{\sqrt{3h}}
\left(x+\sigma(x)z+b(x)h\right)\,dz
+
\frac{1}{2\sqrt{3h}}
\int_{-\sqrt{3h}}^{-\frac{x+b(x)h}{\sigma(x)}}
\left(-x-\sigma(x)z-b(x)h\right)\,dz.
\end{multline*}
A direct calculation gives
\[
\mathbb{E}_x(X_{k+1})
=
\frac{\sigma(x)\sqrt{3h}}{2}
+\frac{(x+b(x)h)^2}{2\sigma(x)\sqrt{3h}}.
\]
Therefore
\begin{align*}
\mathbb{E}_x(\Delta X_k)-b(x)h
&=
\frac{\sigma(x)\sqrt{3h}}{2}
+\frac{(x+b(x)h)^2}{2\sigma(x)\sqrt{3h}}
-x-b(x)h \\
&=
\frac{\left(x+b(x)h-\sigma(x)\sqrt{3h}\right)^2}{2\sigma(x)\sqrt{3h}}.
\end{align*}
In particular,
\[
\left|\mathbb{E}_x(\Delta X_k)-b(x)h\right|
=
\frac{\left(x+b(x)h-\sigma(x)\sqrt{3h}\right)^2}{2\sigma(x)\sqrt{3h}}.
\]

\noindent Let $B_M:=|b_1|\vee |b_2|$. By Lemma~\ref{lemma:bound_on_drift_diffusion}, after decreasing $\delta$ if needed, we have
\[
\sigma(x)\ge \sigma_1>0,
\qquad
|b(x)|\le B_M
\]
for all $x\in \Omega_{r}$ and all $h<\delta$. Moreover, since $x\in \Omega_{r}$,
$$
x\ge \sigma(x)\sqrt{3h},
\qquad
x+b(x)h\le \sigma(x)\sqrt{3h}.
$$
Hence
\[
0\le \sigma(x)\sqrt{3h}-x-b(x)h \le -b(x)h \le B_M h,
\]
and therefore
\[
\left|\mathbb{E}_x(\Delta X_k)-b(x)h\right|
\le
\frac{B_M^2h^2}{2\sigma(x)\sqrt{3h}}
\le
\frac{B_M^2}{2\sigma_1\sqrt{3}}\,h^{3/2}.
\]
Finally, $x\in \Omega_{r}$ also implies
\[
x\ge \sigma(x)\sqrt{3h}\ge \sigma_1\sqrt{3h},
\]
so
\[
h^{3/2}\le \frac{1}{\sigma_1\sqrt{3}}\,hx.
\]
Combining the last two bounds yields
\[
\left|\mathbb{E}_x(\Delta X_k)-b(x)h\right|
\le
\frac{B_M^2}{6\sigma_1^2}\,hx
\le C(hx+h^2)
\]
for some constant $C$ independent of $h$ and $x$. This completes the proof.
\end{proof}

\noindent\textbf{Theorem~\ref{thm:local_err_cont_scheme_diffusion_ring}.}
    (Local Error of $\{X_k\}_{k\in\mathbb{N}}$ in algorithm \ref{alg:diffusion_cont_scheme}, situation 2)\\
    Given $b(0) \neq \frac{\sigma^2(0)}{2\kappa}$, there exists $\delta > 0$ and $C \geq 0$ depending on the relevant derivative bounds of $u$ and on the local boundary constants for $b$ and $\sigma$ such that $\forall h < \delta$ and $\forall x\in \Omega_{r}$,
    \begin{equation}
    \label{eqn:local_err_cont_scheme_drifted_ring}
        \left|\mathbb{E}\left(u_{k+1}-u_k|X_k=x\right)\right| \leq C(h^2 + hx) = \mathcal{O}(hx + h^2) 
    \end{equation}
\begin{proof}
    The proof is a simplified version of proof of Theorem \ref{thm:local_err_cont_scheme_diffusion}.
    By Lemma~\ref{lemma:bound_on_drift_diffusion}, after decreasing $\delta$ if necessary, for all $h<\delta$ and all $x\in\Omega_{r}$,
    $$
    b(x)\in[b_1,b_2],\qquad \sigma(x)\in[\sigma_1,\sigma_2].
    $$
    Let $B_M:=|b_1|\vee |b_2|$.
    Recall from regularity assumption and Lemma~\ref{lemma:3rd_mom_order_negative_drifted_diffusion},
    \begin{align*}
    \E_x|\partial_{tt} u(s_k, X_k) \cdot h^2| & \leq h^2\|\partial_{tt}u\|_{\infty}\\
    \E_x|\partial_{xxx}u(t_k, Y_{k})\cdot \Delta X_k^3| &\leq C(h^2+hx)\|u^{(3)}\|_\infty,
    \end{align*}
    We also have, using the fact that $u$ solves the backward equation \eqref{eqn:diffusion_bckwd_eqn},
    \begin{equation*}
        \begin{split}
           & \E_x\left[\partial_t u_k\cdot h + \partial_xu_k\cdot \Delta X_k + \frac{1}{2}\partial_{xx}u_k\cdot (\Delta X_k^2) \right]
            \\
            &\quad = \partial_xu(t_k,x)\cdot \left(\mathbb{E}_x(\Delta X_k) - b(x)h\right) + \frac{1}{2}\partial_{xx}u(t_k,x) \cdot \left(-\sigma^2(x)\cdot h +  \mathbb{E}_x(\Delta X_k^2)\right).
        \end{split}
    \end{equation*}
    Apply Lemma~\ref{lemma:first_mom_est_reflection_layer_diffusion} and apply the lower bound of $\mathbb{E}_x(X_{k+1})$ to the second conditional moment calculation, we have,
    $$|E_x(\Delta X_k) - b(x)h| \leq C_2( hx + h^2)$$
    $$|E_x(\Delta X_k^2) - \sigma^2(x)h| \leq C_3( hx + h^2)$$
    Therefore, applying the triangle inequality to the Taylor expansion \eqref{eq:taylor4_diffusion}, and then using the first- and second-moment estimates above and Lemma~\ref{lemma:3rd_mom_order_negative_drifted_diffusion}, there exist constants $M_1,M_2$, independent of $x,t,h$, such that 
    % \begin{equation}\label{eqn:local_err4}
    \begin{align*}
         \left|\mathbb{E}\left(u_{k+1}-u_k|X_k=x\right)\right| 
         & \leq \|\partial_{x}u\|_{\infty} \cdot C_2( hx + h^2) + \frac{1}{2}\|\partial_{xx}u\|_{\infty} \cdot C_3( hx + h^2)\\
        &\qquad+ h^2\|\partial_{tt}u\|_{\infty}
        + C_1(h^2+hx)\|u^{(3)}\|_\infty\\
        %+ \|u^{(4)}\|_{\infty}h^2\\
        &\leq M_1 h^2 + M_2 hx. 
    \end{align*}
    %\end{equation}
\end{proof}

\noindent\textbf{Lemma~\ref{lemma:1st_2nd_mom_combination_generator_diffusion_special}.}
    When $b(0) = \frac{\sigma^2(0)}{2\kappa}$, there exist $\delta>0$ and and constant $C$ which depends only on the boundary-layer bounds such that $\forall h< \delta$ and $\forall x\in \Omega_{b}$,
    $$\left| \mathbb{E}_x\left(\Delta X_k\right) - hb(x) \right| \leq C(hx + h^2) = \mathcal{O}(hx + h^2).$$
\begin{proof}
    As before,we know there exists $\delta > 0$ such that for all $h < \delta$,
    $$b(x) \in [b_1, b_2]\,,\qquad \sigma(x) \in [\sigma_1, \sigma_2]\,,\qquad \lambda(x)\in[0,1], \qquad \forall x \in \Omega_{b}.$$
    It suffices to complete the proof if we show the inequality holds for all $x\in [0, \sigma_2\sqrt{3h}]$.
    Since $\sigma(0)>0$, the assumption $b(0)=\sigma^2(0)/(2\kappa)$ implies $b(0)>0$. By continuity, after further decreasing $\delta$ if necessary, $b(x)\ge 0$ for any $x\in \Omega_{b}$ and $\Omega_{r} = \emptyset$.
    We can check that when $b(x) = \frac{\sigma^2(x)}{2\kappa}$, 
    $$\mathbb{E}_x\left(\Delta X_k\right) = hb(x),$$ 
    we shall expect that $=$ is replaced by $\approx$ within the boundary layer.
    We use Taylor expansion to make the argument more accurate, $\mathbb{E}_x\left(\Delta X_k\right)$ is a good approximation to $hb(x)$.
    We know that, denote $\sigma_0 := \sigma(0)$ and $b_0 := b(0)$,
    $$\sigma^2(x) = \left(\sigma_0 + x\cdot\partial_x\sigma(z_x)\right)^2\,,$$
    $$b(x) = b_0 + x\cdot\partial_x b(y_x)\,.$$
    Therefore,
    \begin{equation}
    \label{eqn:sigma_square_miunus_2kappa_b_remainder}
    \begin{split}
        \sigma^2(x) - 2\kappa b(x) &= \sigma_0^2 - 2\kappa b_0 + \hat{R}(x) = \hat{R}(x)
    \end{split}
    \end{equation}
    where 
    $$|\hat{R}(x)| = \left|\left(2\sigma_0\cdot\partial_x\sigma(z_x)-2\kappa\cdot\partial_xb(y_x)\right)x + \left(\partial_x\sigma(z_x)\right)^2 \cdot x^2 \right| \leq \hat{K}(x+x^2).$$
    where $\hat K$ depends only on the local first-derivative bounds of $b$ and $\sigma$ on the boundary neighborhood, and is independent of $x$ and $h$.
    We substitute \eqref{eqn:sigma_square_miunus_2kappa_b_remainder} into \eqref{eqn:diffusion_lambda_cont_scheme},
    \begin{equation*}
    \begin{split}
        \lambda(x) &= \frac{\hat{R}(x)x^2 + 2\kappa \sigma^2(x)x + \sigma^2(x)h\cdot\left(2\kappa b(x) + \hat{R}(x)\right)}{\frac{\kappa \sigma(x)}{\sqrt{3h}}\cdot x^2 + \hat{R}(x)\cdot x^2 + \sigma^3(x)\kappa\sqrt{3h} + \sigma^2(x)h\cdot\left(2\kappa b(x) + \hat{R}(x)\right)}\\
        \\
        &= \frac{\kappa \sigma^2(x)\left(2x + 2b(x)h + R(x)\right)}{\kappa \sigma^2(x)\left(\frac{x^2}{\sigma(x)\sqrt{3h}} + \sigma(x)\sqrt{3h} + 2b(x)h + R(x)\right)}\\
        &= \frac{2x + 2b(x)h + R(x)}{\frac{x^2}{\sigma(x)\sqrt{3h}} + \sigma(x)\sqrt{3h} + 2b(x)h + R(x)}
    \end{split}
    \end{equation*}
    where
    $$R(x) = \frac{\sigma^2(x)h +x^2}{\kappa \sigma^2(x)}\cdot \hat{R}(x)\,.$$
    Recall $x \le \sigma_2\sqrt{3h}$,
    $$\frac{\sigma^2(x)h +x^2}{\kappa \sigma^2(x)} \leq \frac{\sigma_2^2h + 3\sigma_2^2h}{k\sigma_1^2} = Ch.$$
    Hence,
    $$|R(x)| \leq Ch |\hat R(x)| \leq Ch(x + x^2) \leq C(hx+h^2).$$
    Since the drift is positive, substitute $\lambda(x)$ back into \eqref{eqn:first_mom_positive_drifted_diffusion}, we have
    \begin{equation*}
    \begin{split}
        \mathbb{E}_x(\Delta X_k) &= \lambda(x) \cdot \left( \frac{x^2}{2\sigma(x)\sqrt{3h}} + \frac{1}{2}\sigma(x)\sqrt{3h} + b(x)h + \frac{1}{2}R(x) - \frac{1}{2}R(x) \right) - x\\
        &= x + b(x)h + \frac{1}{2}\left(1-\lambda(x)\right)R(x) - x\\
        &= b(x)h + \frac{1}{2}\left(1-\lambda(x)\right)R(x)
    \end{split}
    \end{equation*}
    Hence,
    $$\left| \mathbb{E}_x(\Delta X_k) - b(x)h\right| = \left|\frac{1}{2}\left(1-\lambda(x)\right)R(x)\right| \leq C(hx + h^2)$$
    because we know $\lambda(x) \in [0,1]$, $x\in [0,\sigma_2\sqrt{3h}]$, $|\hat{R}(x)| \leq \hat{K}(x + x^2)$, and the coefficient $\frac{\sigma^2(x)h +x^2}{\kappa \sigma^2(x)}$ we scale $\hat{R}(\cdot)$ to $R(\cdot)$ is of order $h$ using the boundary-layer bounds $\sigma_1\le \sigma(x)\le \sigma_2$ from Lemma~\ref{lemma:bound_on_drift_diffusion}.
\end{proof}

\noindent\textbf{Theorem~\ref{thm:local_err_cont_scheme_diffusion_special}.}
    (Local Error of $\{X_k\}_{k\in\mathbb{N}}$ in algorithm \ref{alg:diffusion_cont_scheme}, situation 3)\\
    If $b(0)= \frac{\sigma^2(0)}{2\kappa}$, there exist $\delta > 0$ and $C\geq 0$, independent of $x$ and $h$, depending only on the relevant bounds of derivatives of $u$ and on the local boundary constants for $b,\sigma$ such that $\forall h<\delta$ and for all $x \in \Omega_{br}$,
    \begin{equation}
    \label{eqn:local_err_cont_scheme_drifted_speical}
        \left|\mathbb{E}\left(u_{k+1}-u_k|X_k=x\right)\right| \leq C(h^2 + hx) = \mathcal{O}(hx + h^2).
    \end{equation}
\begin{proof}
    The proof follows the same Taylor-expansion argument as in Theorem~\ref{thm:local_err_cont_scheme_diffusion}. 
    Fix $x$ such that $x-\sigma(x)\sqrt{3h}<0$. 
    By Lemma~\ref{lemma:bound_on_drift_diffusion}, after decreasing $\delta$ if necessary, for all $h<\delta$ and all $x$ satisfying $x-\sigma(x)\sqrt{3h}<0$,
    $$
    b(x)\in[b_1,b_2],\qquad \sigma(x)\in[\sigma_1,\sigma_2].
    $$
    These constants are independent of $x$ and $h$.
    By Taylor expansion of $u$ centered at $(t_k,x)$, with $Y_k$ between $x$ and $X_{k+1}$ and $s_k\in[t_k,t_{k+1}]$, we have
    \begin{equation*}
    \begin{split}
    u_{k+1}-u_k
    &= \partial_tu_k\,h + \partial_xu_k\,\Delta X_k + \frac12 \partial_{xx}u_k\,\Delta X_k^2 \\
    &\quad + \frac12 \partial_{tt}u(s_k,x)\,h^2 + \frac16 \partial_{xxx}u(t_k,Y_k)\,\Delta X_k^3 .
    \end{split}
    \end{equation*}
    Taking conditional expectation given $X_k=x$ and using the backward equation \eqref{eqn:diffusion_bckwd_eqn}, we obtain
    \begin{equation}
    \label{eq:taylor4_diffusion_expectation_s3}
    \begin{split}
    &\left|\mathbb{E}\left(u_{k+1}-u_k\mid X_k=x\right)\right| \\
    &\le \left|\partial_xu(t_k,x)\right|\cdot \left|\mathbb{E}_x(\Delta X_k)-b(x)h\right|
     + \frac12 \left|\partial_{xx}u(t_k,x)\right|\cdot \left|\mathbb{E}_x(\Delta X_k^2)-\sigma^2(x)h\right| \\
    &\quad + \frac{h^2}{2}\|\partial_{tt}u\|_\infty
     + \frac16 \|u^{(3)}\|_\infty\,\left|\mathbb{E}_x(\Delta X_k^3)\right|.
    \end{split}
    \end{equation}
    
    Because $b(0)=\frac{\sigma^2(0)}{2\kappa}$, the sticky boundary condition implies
    \[
    \partial_{xx}u(t,0)=0.
    \]
    Hence, by Taylor expansion in space,
    \[
    \partial_{xx}u(t_k,x)=x\,u^{(3)}(t_k,y_x)\,, \qquad \text{for some $y_x\in[0,x]$.}
    \]
    Therefore,
    $$|\partial_{xx}u(t_k,x)|\le \|u^{(3)}\|_\infty\,x.$$
    Next, by Lemma~\ref{lemma:1st_2nd_mom_combination_generator_diffusion_special},
    \[
    \left|\mathbb{E}_x(\Delta X_k)-b(x)h\right|\le C_1(hx+h^2).
    \]
    Moreover, since $x-\sigma(x)\sqrt{3h}<0$, the increment is supported in an interval of size $O(\sqrt h)$, so
    \[
    \mathbb{E}_x(\Delta X_k^2)\le C_2 h
    \]
    for some constant $C_2>0$ independent of $x$ and $h$ because on the boundary layer the update either jumps to $0$ or takes a reflected Euler--Maruyama step.
    Finally, by the third-moment bound already established for the boundary layer,
    \[
    \left|\mathbb{E}_x(\Delta X_k^3)\right|\le C_3(hx+h^2).
    \]
    Substituting these estimates into the inequality~\eqref{eq:taylor4_diffusion_expectation_s3} yields
    \begin{equation*}
    \begin{split}
    \left|\mathbb{E}\left(u_{k+1}-u_k\mid X_k=x\right)\right|
    &\le \|\partial_xu\|_\infty \cdot C_1(hx+h^2)
     + \frac12 \|u^{(3)}\|_\infty x \cdot \bigl(\sigma^2(x)h+C_2h\bigr) \\
    &\quad + \frac{h^2}{2}\|\partial_{tt}u\|_\infty
     + \frac16 \|u^{(3)}\|_\infty C_3(hx+h^2) \\
    &\le M_1 h^2 + M_2 hx,
    \end{split}
    \end{equation*}
    for suitable constants $M_1,M_2\ge 0$ independent of $x$ and $h$, after possibly decreasing $\delta$ so that all preceding lemmas hold simultaneously.
\end{proof}

\begin{proof}[Proof of Theorem \ref{thm:local_err_cont_scheme_diffusion_general}]
    The estimate on $\Omega_{br}$ follows from the three boundary-layer results already established. More precisely, when
    $b(0)\neq \frac{\sigma^2(0)}{2\kappa}$, Theorems
    \ref{thm:local_err_cont_scheme_diffusion} and
    \ref{thm:local_err_cont_scheme_diffusion_ring}
    cover $\Omega_{b}$ and $\Omega_{r}$, respectively. When
    $b(0)=\frac{\sigma^2(0)}{2\kappa}$, we have $b(0)>0$, so by continuity,
    after possibly decreasing $\delta$, one has $b(x)>0$ throughout the boundary layer; hence
    $\Omega_{r}=\varnothing$, and Theorem
    \ref{thm:local_err_cont_scheme_diffusion_special}
    yields the estimate on all of $\Omega_{br}=\Omega_{b}$.
    Therefore, after decreasing $\delta>0$ finitely many times and enlarging the constant if needed, we may assume that
    $$\left|\mathbb{E}\left(u_{k+1}-u_k\mid X_k=x\right)\right|
    \le M_b(h^2+hx),
    \qquad \forall x\in \Omega_{br},\ \forall h<\delta.$$

    We now check $x \in \Omega_e$. Recall that, outside the boundary layer, the simulation
    scheme makes a standard Euler--Maruyama jump. Since
    $$x-\sigma(x)\sqrt{3h}+b(x)h>0,$$
    the absolute value in \eqref{eqn:diffusion_cont_scheme_outlayer} is inactive, and therefore
    $$\Delta X_k=\sigma(x)Z_{k+1}+b(x)h.$$
    Hence
    $$\mathbb{E}_x(\Delta X_k)=b(x)h,\qquad
    \mathbb{E}_x(\Delta X_k^2)=\sigma^2(x)h+b^2(x)h^2,$$
    $$\mathbb{E}_x(\Delta X_k^3)=3b(x)\sigma^2(x)h^2+b^3(x)h^3,$$
    and
    \begin{equation}
    \label{eq:EM_fourth_moment_diffusion}
    \begin{split}
        \mathbb{E}_x(\Delta X_k^4)
        &=
        \mathbb{E}_x\left(\sigma(x)Z_{k+1}+b(x)h\right)^4\\
        &=
        \frac{9}{5}\sigma^4(x)h^2
        +6b^2(x)\sigma^2(x)h^3
        +b^4(x)h^4.
    \end{split}
    \end{equation}
    The Taylor expansion upto fourth order is
    \begin{equation}\label{eq:taylor4_diffusion_ext}
    \begin{split}
        u_{k+1} - u_k
        &= \partial_t u_k \cdot h  +  \partial_xu_k\cdot \Delta X_k + \frac{1}{2}\partial_{xx}u_k\cdot \Delta X_k^2\\
        &\quad + \frac{1}{2}\partial_{tt} u(s_k, X_k) \cdot h^2 + \frac{1}{6}\partial_{xxx}u_k\cdot \Delta X_k^3 + \frac{1}{24}\partial_{xxxx}u(t_k, Y_{k})\cdot \Delta X_k^4.
    \end{split}
    \end{equation}
    Using the fact that $u$ solves the backward equation \eqref{eqn:diffusion_bckwd_eqn}, we know
    \begin{equation*}
    \begin{split}
       & \mathbb{E}_x\left[\partial_t u_k\cdot h + \partial_xu_k\cdot \Delta X_k + \frac{1}{2}\partial_{xx}u_k\cdot (\Delta X_k^2) \right]
        \\
        &\quad = -b(x)\partial_xu(t_k,x)\cdot h -\frac{\sigma^2(x)}{2}\partial_{xx}u(t_k,x)\cdot h
        + \partial_xu_k \cdot b(x)h
        + \frac{1}{2}\partial_{xx}u(t_k,x)\cdot \bigl(\sigma^2(x)h+b^2(x)h^2\bigr)\\
        &\quad = \frac{1}{2}b^2(x)\partial_{xx}u(t_k,x)\,h^2.
    \end{split}
    \end{equation*}
    Taking conditional expectation in \eqref{eq:taylor4_diffusion_ext}, and using the previous calculation, gives
    \begin{align*}
         \left|\mathbb{E}\left(u_{k+1}-u_k\mid X_k=x\right)\right|
         &\le \frac{1}{2}|b(x)|^2\|\partial_{xx}u\|_\infty h^2
         + \frac{h^2}{2}\|\partial_{tt}u\|_{\infty}\\
         &\quad
         + \frac{1}{6}\|u^{(3)}\|_\infty
         \left|3b(x)\sigma^2(x)h^2+b^3(x)h^3\right|\\
         &\quad
         +\frac{1}{24}\|u^{(4)}\|_\infty\,\left|\frac{9}{5}\sigma^4(x)h^2
        +6b^2(x)\sigma^2(x)h^3
        +b^4(x)h^4\right|.
    \end{align*}
    By Assumption~\ref{assump:linear_growth_coefficients}, there exists $L>0$ such that
    $$|b(x)|+|\sigma(x)|\le L(1+x),\qquad x\ge 0.$$
    Hence, using $h<1$ and $(1+x)^4\le C(1+x^4)$ for $x\ge0$, we have
    $$|b(x)|^2+\left|b(x)\right|\sigma^2(x)+|b(x)|^3h
    \le C(1+x^3)\le C(1+x^4),$$
    and
    $$\sigma^4(x)+b^2(x)\sigma^2(x)h+b^4(x)h^2
    \le C(1+x^4).$$
    Substituting these bounds and \eqref{eq:EM_fourth_moment_diffusion} into the previous estimate yields
    $$\left|\mathbb{E}\left(u_{k+1}-u_k\mid X_k=x\right)\right|
    \le M_eh^2(1+x^4),\qquad \forall x\in\Omega_e,\ \forall h<\delta,$$
    for some constant $M_e$ independent of $x$ and $h$.
\end{proof}

\noindent\textbf{Lemma~\ref{lemma:lambda_approx}.}
    Denote $\Tilde{\lambda}(\cdot)$ as the jump probability we propose in algorithm \ref{alg:cont_scheme} and $\lambda(\cdot)$ as the one we propose in algorithm \ref{alg:diffusion_cont_scheme}, if for the diffusion $\sigma(0)= 1$ for a sticky diffusion, then
    $$\left|\lambda(x) - \Tilde{\lambda}(x)\right| \leq K_{b,\sigma} \sqrt{h},\quad \forall x\in [0,\sqrt{3h}]$$ where $K_{b,\sigma} > 0$ is a constant depending on the drift term $b(\cdot)$, diffusion $\sigma(\cdot)$, and sticky parameter $\kappa$.
\begin{proof}
    Recall that
    $$
    \lambda(x)=\frac{N_\sigma(x)}{D_\sigma(x)},
    $$
    where
    $$
    N_\sigma(x):=(\sigma^2(x)-2\kappa b(x))x^2+2\kappa\sigma^2(x)x+\sigma^4(x)h
    $$
    and
    $$
    D_\sigma(x):=
    \left(\frac{\kappa\sigma(x)}{\sqrt{3h}}+\sigma^2(x)-2\kappa b(x)\right)x^2
    +\sigma^3(x)\kappa\sqrt{3h}
    +\sigma^4(x)h.
    $$
    Similarly,
    $$
    \Tilde{\lambda}(x)=\frac{N_0(x)}{D_0(x)},
    $$
    where
    $$
    N_0(x):=x^2+2\kappa x+h,
    \qquad
    D_0(x):=\left(\frac{\kappa}{\sqrt{3h}}+1\right)x^2+h+\kappa\sqrt{3h}.
    $$
    
    Using the same compactness argument as in Lemma \ref{lemma:bound_on_drift_diffusion}, and using $\sigma(0)=1$, there exist $\delta_1>0$ and constants
    $b_1,b_2,\sigma_1,\sigma_2$, with $\sigma_1>0$, such that for every $h<\delta_1$ and every $x\in[0,\sqrt{3h}]$,
    $$
    b(x)\in[b_1,b_2],
    \qquad
    \sigma(x)\in[\sigma_1,\sigma_2].
    $$
    Set
    $$
    B_M:=|b_1|\vee |b_2|.
    $$
    Moreover, since $\sigma(0)=1$ and $\sigma$ is continuously differentiable near $0$, after decreasing $\delta_1$ if necessary, there exists a constant $C_\sigma>0$, independent of $x$ and $h$, such that for $n=1,2,3,4$,
    \begin{equation}
    \label{eq:lambda_approx_sigma_power}
    |\sigma^n(x)-1|\le C_\sigma x\le C_\sigma\sqrt{3h},
    \qquad
    \forall x\in[0,\sqrt{3h}],\quad h<\delta_1.
    \end{equation}
    
    We also decrease $\delta_1$ so that $\delta_1\le 1$ and, for every $h<\delta_1$,
    $$
    \frac{\kappa\sigma_1}{\sqrt{3h}}>\sigma_2^2+2\kappa B_M.
    $$
    Then, for every $x\in[0,\sqrt{3h}]$,
    \begin{equation}
    \label{eq:lambda_approx_D_lower}
    D_\sigma(x)
    \ge \sigma^3(x)\kappa\sqrt{3h}
    \ge \kappa\sigma_1^3\sqrt{3h}.
    \end{equation}
    
    Now define
    $$
    \hat{\lambda}(x):=\frac{N_0(x)}{D_\sigma(x)}.
    $$
    We first compare $\lambda(x)$ and $\hat{\lambda}(x)$. Since
    $$
    N_\sigma(x)-N_0(x)
    =
    (\sigma^2(x)-1)x^2
    -2\kappa b(x)x^2
    +2\kappa(\sigma^2(x)-1)x
    +(\sigma^4(x)-1)h,
    $$
    using \eqref{eq:lambda_approx_sigma_power} and $x\le \sqrt{3h}$ gives
    \begin{equation}
    \label{eq:lambda_approx_N_bound}
    |N_\sigma(x)-N_0(x)|\le C_Nh
    \end{equation}
    for some constant $C_N>0$ independent of $x$ and $h$. Hence, by \eqref{eq:lambda_approx_D_lower},
    $$
    |\lambda(x)-\hat{\lambda}(x)|
    =
    \frac{|N_\sigma(x)-N_0(x)|}{D_\sigma(x)}
    \le
    \frac{C_Nh}{\kappa\sigma_1^3\sqrt{3h}}
    \le C'_N\sqrt{h},
    $$
    where $C'_N>0$ is independent of $x$ and $h$.
    
    Next, we compare $\hat{\lambda}(x)$ and $\Tilde{\lambda}(x)$. We have
    $$
    D_\sigma(x)-D_0(x)
    =
    \left(\frac{\kappa(\sigma(x)-1)}{\sqrt{3h}}+\sigma^2(x)-1-2\kappa b(x)\right)x^2
    +\kappa(\sigma^3(x)-1)\sqrt{3h}
    +(\sigma^4(x)-1)h.
    $$
    Using \eqref{eq:lambda_approx_sigma_power} again, together with $x\le \sqrt{3h}$ and $|b(x)|\le B_M$, we obtain
    \begin{equation}
    \label{eq:lambda_approx_D_bound}
    |D_\sigma(x)-D_0(x)|\le C_Dh
    \end{equation}
    for some constant $C_D>0$ independent of $x$ and $h$. Moreover,
    $$
    \hat{\lambda}(x)-\Tilde{\lambda}(x)
    =
    \frac{D_0(x)-D_\sigma(x)}{D_\sigma(x)}\cdot\frac{N_0(x)}{D_0(x)}.
    $$
    Using Lemma \ref{lemma:bound_on_lambda_SBM}, which gives $0\le \Tilde{\lambda}(x)\le 1$ on $[0,\sqrt{3h}]$, together with \eqref{eq:lambda_approx_D_lower} and \eqref{eq:lambda_approx_D_bound}, we obtain
    $$
    |\hat{\lambda}(x)-\Tilde{\lambda}(x)|
    \le
    \frac{|D_\sigma(x)-D_0(x)|}{D_\sigma(x)}|\Tilde{\lambda}(x)|
    \le
    \frac{C_Dh}{\kappa\sigma_1^3\sqrt{3h}}
    \le C'_D\sqrt{h},
    $$
    where $C'_D>0$ is independent of $x$ and $h$.
    
    Combining the two estimates, for every $h<\delta_1$ and every $x\in[0,\sqrt{3h}]$,
    $$
    |\lambda(x)-\Tilde{\lambda}(x)|
    \le
    |\lambda(x)-\hat{\lambda}(x)|+|\hat{\lambda}(x)-\Tilde{\lambda}(x)|
    \le
    K_{b,\sigma}\sqrt{h}.
    $$
    Here $K_{b,\sigma}:=C'_N+C'_D$ depends only on $\kappa$, the local bounds $b_1,b_2,\sigma_1,\sigma_2$, and the local derivative bound for $\sigma$, but is independent of $x$ and $h$. This completes the proof with $\delta=\delta_1$.
\end{proof}

\noindent\textbf{Lemma~\ref{lemma:key_diffusion_general}.}
    There exists $\delta > 0$ and constant $C$ such that for all $h < \delta$, 
    $$\mathbb{E}\left(\sum_{k=0}^{N-1} X_k \cdot \mathbf{1}_{\Omega_{br}} (X_k) \right) < C$$
\begin{proof}
    Throughout this proof we use the normalization $\sigma(0)=1$ made before the statement of the lemma. 
    By the proof of Lemma~\ref{lemma:bound_on_drift_diffusion}, there exist $\delta_B>0$ and a constant $K_B>0$, depending only on local bounds for $b$ and $\sigma$ near $0$, such that for every $h<\delta_B$,
    \begin{equation}
    \label{eq:keydiffusion_Bh}
        \Omega_{br}\subset B_h:= [0,\sqrt{3h}+K_Bh].
    \end{equation}
    Hence it is enough to prove a uniform bound with $\mathbf{1}_{B_h}$ in place of $\mathbf{1}_{\Omega_{br}}$. %insert here
    The rest of the proof is a discrete supersolution argument. We will construct a bounded barrier function $\varphi_h=e^{w_h}$ whose one-step drift is negative by an amount comparable to $x$ on the enlarged boundary layer $B_h$. More precisely, the key estimate is
    $$
    P\varphi_h(x)\leq \varphi_h(x)(1-f_h(x)+C_1h),
    \qquad
    f_h(x)\geq x\mathbf{1}_{B_h}(x).
    $$
    Once this is proved, the time factor $e^{K_T(T-t_i)}$ will absorb the harmless $C_1h$ error, giving a supersolution for the discrete occupation equation. A backward comparison then bounds the expected occupation sum.

    We now modify the test function used in Lemma~\ref{lemma:key} so that it is unchanged near the boundary but becomes constant away from the boundary. This avoids requiring $b$ and $\sigma$ to be globally bounded. Fix $s>16$ and $M>0$. Choose a fixed number $R>1$, independent of $h$. Later we decrease $\delta_B$ if necessary so that
    $$B_h\subset [0,R/4].$$

    Let $\chi\in C^2(\mathbb{R})$ satisfy
    $$0\leq \chi\leq 1,\qquad \chi(r)=1\text{ for }r\leq 0,\qquad \chi(r)=0\text{ for }r\geq 1.$$
    Define
    $$\eta_R(x):=\int_0^x \chi\left(\frac{y-R}{R}\right)\,dy.$$
    Then $\eta_R(x)=x$ for $0\leq x\leq R$, while $\eta_R(x)=A_R$ for $x\geq 2R$, where
    $$A_R:=\int_0^{2R}\chi\left(\frac{y-R}{R}\right)\,dy.$$
    Moreover, $0\leq \eta_R'(x)\leq 1$ and $|\eta_R''(x)|\leq C_R$ for a constant $C_R$ depending only on $R$ and the cutoff $\chi$.

    Define
    \begin{equation}
    \label{eq:keydiffusion_wh}
        w_h(x):=
        \begin{cases}
        M, & 0\leq x\leq \sqrt{3h}/2,\\
        M-s\left(\eta_R(x)-\frac{\sqrt{3h}}{2}\right), & x>\sqrt{3h}/2.
        \end{cases}
    \end{equation}
    Then $w_h$ agrees with the function used in Lemma~\ref{lemma:key} on $[0,R]$:
    $$
        w_h(x)=
        \begin{cases}
        M, & 0\leq x\leq \sqrt{3h}/2,\\
        M-s\left(x-\frac{\sqrt{3h}}{2}\right), & \sqrt{3h}/2<x\leq R.
        \end{cases}
    $$
    Also, $w_h$ is constant on $[2R,\infty)$. Let $\varphi_h(x):=e^{w_h(x)}.$ Since $w_h(x)\leq M$ for all $x\geq 0$, $\varphi_h(x)\leq e^M$ uniformly in $x$ and $h$.
    We first prove that there exist constants $C_1>0$ and $\delta_1>0$, independent of $x$ and $h$, such that for all $h<\delta_1$,
    \begin{equation}
    \label{eq:keydiffusion_Pphi_goal}
        P\varphi_h(x)\leq \varphi_h(x)\left(1-f_h(x)+C_1h\right),
        \qquad \forall x\geq 0,
    \end{equation}
    where $f_h$ satisfies
    \begin{equation}
    \label{eq:keydiffusion_fh_goal}
        f_h(x)\geq x\mathbf{1}_{B_h}(x).
    \end{equation}% Insert here

    We introduce an auxiliary transition operator $P_0$. 
    The role of $P_0$ is to separate the genuinely new diffusion effects from the SBM calculation. On $B_h$, the normalization $\sigma(0)=1$ implies $\sigma(x)-1=O(\sqrt h)$ and the scheme makes uniformly distributed jump of size $O(\sqrt h)$, so the continuous update in the diffusion scheme differs from the SBM reflected update by only $O(h)$.
    The remaining complication is that $P$ and $P_0$ correspond to different boundary layer 
    %($\Omega_{br}$ for $P_0$ and $\Omega_{br}$ for $P$)
    .
    Given $Z\sim \mathcal{U}[-\sqrt{3h},\sqrt{3h}]$, 
    $$
    P_0\varphi_h(x)=
    \begin{cases}
    \lambda(x)\mathbb{E}\left[\varphi_h(|x+Z|)\right]+(1-\lambda(x))\varphi_h(0), & 0\leq x\leq \sqrt{3h},\\
    \mathbb{E}\left[\varphi_h(|x+Z|)\right], & x>\sqrt{3h}.
    \end{cases}
    $$
    Thus $P_0$ is the same as the SBM transition operator, except that it uses the diffusion jump probability $\lambda(x)$ instead of $\Tilde{\lambda}(x)$.
    We claim that for $x\in B_h$,
\begin{equation}
\label{eq:keydiffusion_P_P0}
    P\varphi_h(x)\leq P_0\varphi_h(x)+C_2h\varphi_h(x),
\end{equation}
where $C_2$ is independent of $x$ and $h$. We prove this by coupling the true diffusion update and the auxiliary update using the same random variable
$Z\sim \mathcal{U}[-\sqrt{3h},\sqrt{3h}]$.

First, after decreasing $\delta_1$ if necessary, $B_h\subset[0,R/4]$. Since $\sigma(0)=1$ and $\sigma$ is locally Lipschitz near $0$, there exists a constant $C_\sigma>0$, independent of $x$ and $h$, such that
$$
|\sigma(x)-1|\leq C_\sigma x\leq C_\sigma(\sqrt{3h}+K_Bh)\leq C\sqrt h,
\qquad x\in B_h.
$$
Similarly, $b$ is locally bounded on a fixed compact neighborhood of $0$, so
$$
|b(x)|\leq C,\qquad x\in B_h.
$$
Thus, for all possible values of $Z$ and all $x\in B_h$,
$$
|x+Z|\leq x+\sqrt{3h}\leq C\sqrt h,
$$
and
$$
\left||x+\sigma(x)Z|+b(x)h\right|
\leq x+|\sigma(x)Z|+|b(x)|h
\leq C\sqrt h.
$$
After decreasing $\delta_1$ once more so that $C\sqrt h<R$, all one-step positions appearing in the comparison below lie in $[0,R]$.
We also record a bound on the Lipschitz constant of $\varphi_h$. Since $0\leq \eta_R'\leq 1$, the function $\eta_R$ is $1$-Lipschitz. For $h$ small enough that $\sqrt{3h}/2<R$, we can write
$$
w_h(x)=M-s\left(\eta_R(x)-\frac{\sqrt{3h}}{2}\right)_+,
$$
Hence $w_h$ is $s$-Lipschitz, uniformly in $h$. Since $w_h\leq M$,
\begin{equation}
\label{eq:phi_lipschitz}
    |\varphi_h(y)-\varphi_h(z)|
    =
    |e^{w_h(y)}-e^{w_h(z)}|
    \leq e^M |w_h(y)-w_h(z)|
    \leq se^M|y-z|.
\end{equation}
Therefore $\varphi_h$ is Lipschitz with constant $L_\varphi:=se^M$, independent of $x$ and $h$.
Finally, on $B_h$ we also have a uniform lower bound on $\varphi_h(x)$. Indeed, after decreasing $\delta_1$ if necessary,
\begin{equation}
\label{eq:lower_bound_phi}
    \varphi_h(x)=\exp(w_h(x))
    \geq
    \exp\left(M-s\left(\frac{\sqrt{3h}}{2}+K_Bh\right)\right)
    \geq e^{M/2},
    \qquad x\in B_h.
\end{equation}
Consequently, any estimate of the form $Ch$ on $B_h$ may be rewritten as $Ch\varphi_h(x)$ after increasing the constant $C$.
We now compare the continuous branches. Let
$$
Y_0:=|x+Z|
$$
denote the continuous update in the auxiliary operator $P_0$. For the true diffusion scheme, define the continuous update
$$
Y_1:=
\left||x+\sigma(x)Z|+b(x)h\right|
$$
Using $||a|-|b||\leq |a-b|$ gives
\begin{equation}
\label{eq:keydiffusion_cont_branch_inlayer}
\begin{split}
    |Y_1-Y_0|
    &=
    \left|\left||x+\sigma(x)Z|+b(x)h\right|-|x+Z|\right|\\
    &\leq
    \left||x+\sigma(x)Z|+b(x)h-|x+Z|\right|\\
    &\leq
    \left||x+\sigma(x)Z|-|x+Z|\right|+|b(x)|h\\
    &\leq
    |(\sigma(x)-1)Z|+|b(x)|h\\
    &\leq Ch.
\end{split}
\end{equation}
Therefore,
\begin{equation}
\label{eq:continuous_branch_comparison}
    |\varphi_h(Y_1)-\varphi_h(Y_0)|
    \leq L_\varphi |Y_1-Y_0|
    \leq Ch
    \leq Ch\varphi_h(x).
\end{equation}

It remains to control the possible mismatch between the two boundary thresholds $\sigma(x)\sqrt{3h}$ and $\sqrt{3h}$. A mismatch occurs only when one of the inequalities
$$
x\leq \sigma(x)\sqrt{3h},\qquad x\leq \sqrt{3h}
$$
holds and the other fails. Equivalently, $x$ lies between $\sqrt{3h}$ and $\sigma(x)\sqrt{3h}$. Hence
\begin{equation}
\label{eq:keydiffusion_threshold_mismatch}
    |x-\sigma(x)\sqrt{3h}|
    \leq |\sigma(x)\sqrt{3h}-\sqrt{3h}|
    =
    |\sigma(x)-1|\sqrt{3h}
    \leq Ch.
\end{equation}
In this mismatch region, we also need that $1-\lambda(x)$ is of order $h$. Using the identity
\begin{equation}
\label{eq:keydiffusion_one_minus_lambda}
1-\lambda(x)
=
\frac{\frac{\kappa\sigma(x)}{\sqrt{3h}}\left(x-\sigma(x)\sqrt{3h}\right)^2}
{\left(\frac{\kappa\sigma(x)}{\sqrt{3h}}+\sigma^2(x)-2\kappa b(x)\right)x^2+\sigma^3(x)\kappa\sqrt{3h}+\sigma^4(x)h},
\end{equation}
and using the same denominator lower-bound argument as in Lemma~\ref{lemma:lambda_approx}, after decreasing $\delta_1$ if necessary, there exists $c>0$ such that
$$
\left(\frac{\kappa\sigma(x)}{\sqrt{3h}}+\sigma^2(x)-2\kappa b(x)\right)x^2+\sigma^3(x)\kappa\sqrt{3h}+\sigma^4(x)h
\geq c\sqrt h,
\qquad x\in B_h.
$$
Combining this lower bound with \eqref{eq:keydiffusion_threshold_mismatch} and \eqref{eq:keydiffusion_one_minus_lambda}, we obtain
\begin{equation}
\label{eq:lambda_complement_estimate}
    1-\lambda(x)
    \leq
    \frac{C\frac{1}{\sqrt h}\cdot h^2}{c\sqrt h}
    \leq Ch.
\end{equation}
% Insert here
We now consider the possible threshold configurations.
The following four cases only account for which operator regards $x$ as being in the sticky boundary layer. If both operators make the same decision, the comparison follows from the $O(h)$ continuous-branch estimate. If their decisions differ, then $x$ lies in the thin mismatch region, where $1-\lambda(x)=O(h)$, so the extra sticky-jump contribution is also only $O(h)$.
\begin{enumerate}
    \item Suppose $x\in\Omega_{b}$ and $x\leq \sqrt{3h}.$
    Then
    $$
    P\varphi_h(x)
    =
    \lambda(x)\mathbb{E}\left[\varphi_h(Y_1)\right]+(1-\lambda(x))\varphi_h(0),
    $$
    and
    \begin{equation}
    \label{eq:P0_boundary_layer}
        P_0\varphi_h(x)
        =
        \lambda(x)\mathbb{E}\left[\varphi_h(Y_0)\right]+(1-\lambda(x))\varphi_h(0).
    \end{equation}
    Thus, by \eqref{eq:continuous_branch_comparison}
    $$
    P\varphi_h(x)-P_0\varphi_h(x)
    =
    \lambda(x)\mathbb{E}\left[\varphi_h(Y_1)-\varphi_h(Y_0)\right]
    \leq Ch\varphi_h(x).
    $$ 
    \item Suppose that $x\in\Omega_{b}$ and $x>\sqrt{3h}$,
    then $P_0$ uses only the continuous branch,
    \begin{equation*}
    \begin{split}
        P\varphi_h(x)-P_0\varphi_h(x)
        &=
        \lambda(x)\mathbb{E}\left[\varphi_h(Y_1)-\varphi_h(Y_0)\right]
        +(1-\lambda(x))\left(\varphi_h(0)-\mathbb{E}\left[\varphi_h(Y_0)\right]\right)\\
        &\leq Ch\varphi_h(x)
        +(1-\lambda(x))\cdot C\varphi_h(x)\\
        &\leq Ch\varphi_h(x).
    \end{split}
    \end{equation*}
    where we use \eqref{eq:continuous_branch_comparison}, \eqref{eq:lambda_complement_estimate} and the fact that
    $$
    |\varphi_h(0)-\varphi_h(Y_0)|\leq 2e^M \leq 2e^{M/2}\varphi_h(x)\leq C\varphi_h(x),
    \qquad x\in B_h,
    $$
    which follows from the upper bound $\varphi_h\leq e^M$ and the lower bound \eqref{eq:lower_bound_phi}.
    \item Suppose that $x\notin\Omega_{b}$ and $x\leq\sqrt{3h}$,
    then the true diffusion scheme uses only the continuous branch, while $P_0$ is in its sticky layer. Hence
    \begin{equation*}
    \begin{split}
        P\varphi_h(x)-P_0\varphi_h(x)
        &=
        \mathbb{E}\left[\varphi_h(Y_1)-\varphi_h(Y_0)\right]
        +(1-\lambda(x))\left(\mathbb{E}\left[\varphi_h(Y_0)\right]-\varphi_h(0)\right)\\
        &\leq Ch\varphi_h(x)
        +(1-\lambda(x))\cdot C\varphi_h(x)\\
        &\leq Ch\varphi_h(x).
    \end{split}
    \end{equation*}
    \item Suppose that $x\notin\Omega_{b}$ and $x\geq\sqrt{3h}$,
    then both operators use only their continuous branch, and therefore
    $$
    P\varphi_h(x)-P_0\varphi_h(x)
    =
    \mathbb{E}\left[\varphi_h(Y_1)-\varphi_h(Y_0)\right]
    \leq Ch\varphi_h(x).
    $$
    Taking $C_2$ large enough proves \eqref{eq:keydiffusion_P_P0}.
\end{enumerate}
    
    We now estimate $P_0\varphi_h$ on $[0,\sqrt{3h}]$. On this interval $w_h$ agrees with the test function in Lemma~\ref{lemma:key}. The only difference between $P_0$ and the SBM transition operator is the replacement of $\Tilde{\lambda}(x)$ by $\lambda(x)$.
    Let $\Tilde P_0$ denote the same auxiliary operator as $P_0$, but with $\lambda$ replaced by $\Tilde{\lambda}$. 
    We wish to show  $P_0\varphi_h(x) \approx \Tilde P_0\varphi_h(x)$ so that we can reuse calculation done in Lemma~\ref{lemma:key}.
    For $x\in[0,\sqrt{3h}]$, together with \eqref{eq:P0_boundary_layer},
    $$
    P_0\varphi_h(x)-\Tilde P_0\varphi_h(x)
    =
    \left(\lambda(x)-\Tilde{\lambda}(x)\right)
    \left(
    \mathbb{E}\left[\varphi_h(Y_0)\right]-\varphi_h(0)
    \right),
    $$
    where $Y_0=|x+Z| \leq 2\sqrt{3h}$.
    Using the Lipschitz bound for $\varphi_h$ \eqref{eq:phi_lipschitz} and the lower bound \eqref{eq:lower_bound_phi}, 
    $$
    \left|
    \mathbb{E}\left[\varphi_h(Y_0)\right]-\varphi_h(0)
    \right|
    \leq
    \mathbb{E}\left|\varphi_h(Y_0)-\varphi_h(0)\right|
    \leq C\sqrt h
    \leq C\sqrt h\,\varphi_h(x).
    $$
    Apply Lemma~\ref{lemma:lambda_approx},
    $$
    |P_0\varphi_h(x)-\Tilde P_0\varphi_h(x)|
    \leq Ch\varphi_h(x),
    \qquad x\in[0,\sqrt{3h}].
    $$
    We now estimate $P_0\varphi_h(x)-\varphi_h(x)$  by the case-by-case calculation in Lemma~\ref{lemma:key}.
    \begin{enumerate}
        \item
        For $x\in[0,\sqrt{3h}/2]$,
        $$
        P_0\varphi_h(x)\leq \varphi_h(x)\left(1-f_h(x)+C_3h\right),
        $$
        where
        $$
        f_h(x):=
        \frac{s}{4}\cdot\frac{\Tilde{\lambda}(x)}{2}\sqrt{3h}
        +s\cdot\frac{\Tilde{\lambda}(x)x^2}{2\sqrt{3h}}.
        $$
        The proof of Lemma~\ref{lemma:key} shows that, for $s>16$,
        $$
        f_h(x)\geq x,\qquad x\in[0,\sqrt{3h}/2].
        $$
        \item 
        Likewise, for $x\in[\sqrt{3h}/2,\sqrt{3h}]$,
        $$
        P_0\varphi_h(x)\leq \varphi_h(x)\left(1-f_h(x)+C_3h\right),
        $$
        where
        $$
        f_h(x):=
        \frac{s\Tilde{\lambda}(x)}{16}\sqrt{3h}
        +\frac{s\Tilde{\lambda}(x)x^2}{4\sqrt{3h}}
        +\frac{s\Tilde{\lambda}(x)x}{4}
        +\frac{s\sqrt{3h}}{2}
        -sx.
        $$
        Again, the proof of Lemma~\ref{lemma:key} shows that, for $s>16$,
        $$
        f_h(x)\geq x,\qquad x\in[\sqrt{3h}/2,\sqrt{3h}].
        $$
        \item 
        The preceding SBM comparison covers $[0,\sqrt{3h}]$. The enlarged set $B_h$ contains one additional strip of width $O(h)$, namely $[\sqrt{3h},\sqrt{3h}+K_Bh]$. This strip appears only because the diffusion boundary layer is not exactly the SBM boundary layer, so it must be checked directly.
        It remains to treat the thin strip $[\sqrt{3h},\sqrt{3h}+K_Bh]$. On this interval $P_0$ performs standard EM jump without reflection because $x\geq \sqrt{3h}$. Also $x+Z\in[0,R]$ for $h$ sufficiently small, so $w_h$ agrees with the original function used in Lemma~\ref{lemma:key}. A direct calculation gives, uniformly for $x\in[\sqrt{3h},\sqrt{3h}+K_Bh]$,
        \begin{equation}
        \label{eq:keydiffusion_thin_strip_calc}
        \begin{split}
            P_0\varphi_h(x)-\varphi_h(x)
            &=
            \left(\frac{3}{4}-\frac{x}{2\sqrt{3h}}\right)e^M
            +\frac{e^M}{2\sqrt{3h}}\int_{\sqrt{3h}/2}^{\sqrt{3h}+x}e^{-s(z-\sqrt{3h}/2)}\,dz
            -e^{w_h(x)}\\
            &=
            -e^{w_h(x)}
            \left[
            \frac{s}{4}
            \left(
            \frac{x^2}{\sqrt{3h}}-3x+\frac{9}{4}\sqrt{3h}
            \right)
            +\mathcal{O}(h)
            \right].
        \end{split}
        \end{equation}
        Thus on this thin strip define
        $$
        f_h(x):=\frac{s}{4}
        \left(
        \frac{x^2}{\sqrt{3h}}-3x+\frac{9}{4}\sqrt{3h}
        \right).
        $$
        For $x\in[\sqrt{3h},\sqrt{3h}+K_Bh]$,
        $$
        f_h'(x)=\frac{s}{4}\left(\frac{2x}{\sqrt{3h}}-3\right)\leq -\frac{s}{8}
        $$
        after decreasing $\delta_1$ so that $2K_Bh/\sqrt{3h}\leq 1/2$. Hence $f_h$ is decreasing on the thin strip. Therefore
        \begin{equation}
        \label{eq:keydiffusion_thin_strip_lower}
        \begin{split}
            f_h(x)
            &\geq f_h(\sqrt{3h}+K_Bh)\\
            &=\frac{s}{16}\sqrt{3h}-\frac{sK_B}{4}h+\frac{sK_B^2}{4}\frac{h^2}{\sqrt{3h}}.
        \end{split}
        \end{equation}
        Since $s>16$, after decreasing $\delta_1$ if necessary,
        $$
        f_h(\sqrt{3h}+K_Bh)>\sqrt{3h}+K_Bh.
        $$
        Hence
        $$
        f_h(x)\geq x,\qquad x\in[\sqrt{3h},\sqrt{3h}+K_Bh].
        $$
    \end{enumerate}
    % Insert here
    Combining the estimates for $P_0\varphi_h(x)-\varphi_h(x)$ on $B_h$ with \eqref{eq:keydiffusion_P_P0}, we obtain
    $$
    P\varphi_h(x)\leq \varphi_h(x)(1-f_h(x)+Ch),
    \qquad x\in B_h.
    $$
    
    We now prove the corresponding estimate outside $B_h$. 
    For $x\notin B_h$, set
        $$f_h(x):=0.$$
    Since $\Omega_{br}\subset B_h$, if $x\notin B_h$, then the true scheme is outside the boundary layer and the standard Euler--Maruyama update is nonnegative for every possible $Z$. Thus
    $$X_{k+1}=x+\sigma(x)Z+b(x)h.$$
    By Assumption~\ref{assump:linear_growth_coefficients}, there exists $L>0$ such that
    $$
    |b(x)|+|\sigma(x)|\leq L(1+x),\qquad x\geq 0.
    $$
    If we can assume global boundedness on drift and diffusion, this part is straightforward because the usual Taylor expansion gives an order-$h$ one-step bound for the test function under bounded coefficients.
    Without global boundedness assumption, we analyze $P\varphi_h(x)$ in $$
    B_h^c\cap[0,R/2],\qquad [R/2,4R],\qquad [4R,\infty).
    $$
    The first region exploits the piecewise linear structure of $w(x)$ and boundedness of $\sigma(x),\, b(x)$ in the compact domain. 
    The second region exploits the smoothness of $\varphi$ and the boundedness of drift and diffusion.
    The third region, $\varphi$ is constant by construction and we can make sure the next step is within the constant region by linear growth assumption.
    \begin{enumerate}
        \item 
        Consider $x\in B_h^c\cap[0,R/2]$. On $[0,R]$, $b$ and $\sigma$ are bounded. For $h$ sufficiently small, every possible value of $x+\sigma(x)Z+b(x)h$ lies in $[0,R]$. Since $x\notin B_h$, we also have $x>\sqrt{3h}/2$. On $[0,R]$, the function $w_h$ is flat-linear, and for $y\in[0,R]$ and $x>\sqrt{3h}/2$,
        $$w_h(y)-w_h(x)\leq -s(y-x).$$
        Therefore, by Taylor expansion around $0$, using first and second moment of uniform random variable and the boundedness of $\sigma$, $b$ in $[0,R]$, there exists $C_4$ is independent of $x$ and $h$ such that
        $$
        \frac{P\varphi_h(x)}{\varphi_h(x)}
        \leq
        \mathbb{E}\left[e^{-s(\sigma(x)Z+b(x)h)}\right]
        \leq 1+C_4h.
        $$
        \item 
        Consider $x\in[R/2,4R]$. For $h$ sufficiently small, every possible value of
        $$Y:=x+\sigma(x)Z+b(x)h$$
        lies in $[R/4,5R]$. Also $\sqrt{3h}/2<R/4$, so $\varphi_h$ is $C^2$ on $[R/4,5R]$. 
        There exist constants $c_{\varphi},C_{\varphi}>0$, depending on $s,M,R$ and $\chi$ but not on $x$ or $h$, such that
        $$
        0<c_{\varphi}\leq \varphi_h(y),
        \qquad
        |\varphi_h'(y)|+|\varphi_h''(y)|\leq C_{\varphi},
        \qquad y\in[R/4,5R].
        $$
        Taylor expansion gives
        $$
        \varphi_h(Y)
        =
        \varphi_h(x)+\varphi_h'(x)(Y-x)+\frac{1}{2}\varphi_h''(\xi)(Y-x)^2
        $$
        for some $\xi$ between $x$ and $Y$. Taking expectations and using $\mathbb{E}Z=0$ and $\mathbb{E}Z^2=h$,
        $$
        |P\varphi_h(x)-\varphi_h(x)|
        \leq C h.
        $$
        Since $\varphi_h(x)\geq c_{\varphi}$ on this compact interval, this implies
        $$
        P\varphi_h(x)\leq \varphi_h(x)(1+C_5h),
        \qquad x\in[R/2,4R],
        $$
        where $C_5$ is independent of $x$ and $h$.
        \item 
        Consider $x\geq 4R$. By the linear-growth assumption, for all possible $Z$,
        $$
        x+\sigma(x)Z+b(x)h
        \geq x-L(1+x)(\sqrt{3h}+h).
        $$
        Since $R\geq 1$, $1+x\leq \frac{5}{4}x$ for $x\geq 4R$. After decreasing $\delta_1$ so that
        $$\frac{5}{4}L(\sqrt{3h}+h)\leq \frac{1}{2},$$
        we get
        $$x+\sigma(x)Z+b(x)h\geq \frac{x}{2}\geq 2R.$$
        Thus both $x$ and $X_{k+1}$ lie in the region where $w_h$ is constant. Therefore
        $$P\varphi_h(x)=\varphi_h(x),\qquad x\geq 4R.$$
    \end{enumerate}

    Taking $C_1$ to be larger than the constants appearing in the estimates above, and decreasing $\delta_1$ if necessary, we obtain \eqref{eq:keydiffusion_Pphi_goal} for all $x\geq 0$ and all $h<\delta_1$. Moreover, by construction,
    $$f_h(x)\geq x\mathbf{1}_{B_h}(x).$$
    Since $f_h$ is supported on $B_h$ and the formulas above are all of order $\sqrt h$, there exists $C_6>0$, independent of $x$ and $h$, such that
    \begin{equation}
    \label{eq:keydiffusion_fh_upper}
        0\leq f_h(x)\leq C_6\sqrt h,\qquad \forall x\geq 0,\quad h<\delta_1.
    \end{equation}
    % Insert here
    We now introduce the time-dependent supersolution
    \begin{equation}
    \label{eq:keydiffusion_V_define}
        V(t_i,x):=e^{K_T(T-t_i)}\varphi_h(x),
        \qquad i=0,\dots,N.
    \end{equation}
    The constant $K_T>0$ will be chosen below. Using \eqref{eq:keydiffusion_Pphi_goal},
    \begin{equation}
    \label{eq:keydiffusion_PV_calc}
    \begin{split}
        PV(t_i,x)-V(t_i,x)
        &=e^{K_T(T-t_i-h)}P\varphi_h(x)-e^{K_T(T-t_i)}\varphi_h(x)\\
        &\leq
        e^{K_T(T-t_i)}\varphi_h(x)
        \left[
        e^{-K_Th}(1-f_h(x)+C_1h)-1
        \right].
    \end{split}
    \end{equation}
    After decreasing $\delta_1$ if necessary,
    $$e^{-K_Th}=1-K_Th+r_{K_T}(h),\qquad |r_{K_T}(h)|\leq C_7(K_T)h^2.$$
    Using this expansion in \eqref{eq:keydiffusion_PV_calc}, together with \eqref{eq:keydiffusion_fh_upper}, gives
    $$
    PV(t_i,x)-V(t_i,x)
    \leq
    -e^{K_T(T-t_i)}\varphi_h(x)
    \left(
    f_h(x)+(K_T-C_1)h-C_8(K_T)h^{3/2}
    \right).
    $$
    Choose
    $$K_T:=C_1+1.$$
    After decreasing $\delta_1$ again, we may assume
    $$
    (K_T-C_1)h-C_8(K_T)h^{3/2}\geq 0.
    $$
    Hence
    \begin{equation}
    \label{eq:keydiffusion_PV_super}
        PV(t_i,x)-V(t_i,x)
        \leq
        -e^{K_T(T-t_i)}\varphi_h(x)f_h(x).
    \end{equation}
    Finally, decrease $\delta_1$ once more so that
    $$M-s\left(\frac{\sqrt{3h}}{2}+K_Bh\right)\geq 0.$$
    Then $w_h(x)\geq 0$ on $B_h$, so $e^{K_T(T-t_i)}\varphi_h(x)\geq 1$ on $B_h$. Combining this with \eqref{eq:keydiffusion_fh_goal} and \eqref{eq:keydiffusion_PV_super}, we conclude that
    \begin{equation}
    \label{eq:keydiffusion_PV_final}
        PV(t_i,x)-V(t_i,x)
        \leq
        -x\mathbf{1}_{B_h}(x),
        \qquad i=0,\dots,N-1,\quad x\geq 0.
    \end{equation}
    Define
    $$
    U_B(t_i,x):=
    \mathbb{E}\left(
    \sum_{k=i}^{N-1}X_k\mathbf{1}_{B_h}(X_k)
    \,\middle|\, X_i=x
    \right).
    $$
    By Lemma~\ref{lem:V}, applied with $g(t_i,x)=x\mathbf{1}_{B_h}(x)$, we have
    $$
    PU_B(t_i,x)-U_B(t_i,x)=-x\mathbf{1}_{B_h}(x),
    \qquad i=0,\dots,N-1,
    $$
    with terminal condition $U_B(T,x)=0$. On the other hand, $V(T,x)=\varphi_h(x)\geq 0=U_B(T,x)$, and \eqref{eq:keydiffusion_PV_final} shows that $V$ is a supersolution. A backward induction, exactly as in the proof of Lemma~\ref{lemma:key}, gives
    $$
    U_B(t_i,x)\leq V(t_i,x),
    \qquad i=0,\dots,N,\quad x\geq 0.
    $$
    Since $w_h(x)\leq M$ for all $x\geq 0$,
    $$
    V(t_i,x)\leq e^{K_TT+M},
    \qquad i=0,\dots,N,\quad x\geq 0.
    $$
    Therefore
    $$
    U_B(t_0,x)\leq e^{K_TT+M},
    \qquad x\geq 0.
    $$
    Using $\Omega_{br}\subset B_h$, we get
    $$
    \mathbb{E}\left(
    \sum_{k=0}^{N-1}X_k\mathbf{1}_{\Omega_{br}}(X_k)
    \,\middle|\, X_0=x
    \right)
    \leq
    U_B(t_0,x)
    \leq e^{K_TT+M}.
    $$
    If $X_0$ is random, taking expectation and using the tower property gives
    $$
    \mathbb{E}\left(
    \sum_{k=0}^{N-1}X_k\mathbf{1}_{\Omega_{br}}(X_k)
    \right)
    \leq e^{K_TT+M}.
    $$
    Thus the desired bound holds with $C=e^{K_TT+M}$, which is independent of $h$.
\end{proof}

\begin{proof}[Proof of Lemma \ref{lemma:fourth_moment_bound_diffusion_scheme}]
    It suffices to prove the one-step estimate
    $$
    \mathbb{E}_x(X_{k+1}^4)\le (1+Ch)x^4+Ch,
    $$
    for a constant $C$ independent of $x$ and $h$.
    
    First consider $x\in\Omega_e$. In this region the scheme makes a standard Euler--Maruyama jump, and the absolute value is inactive, so
    $$
    X_{k+1}=x+b(x)h+\sigma(x)Z_{k+1},
    \qquad
    Z_{k+1}\sim \mathrm{Unif}[-\sqrt{3h},\sqrt{3h}].
    $$
    Using $\mathbb{E}Z_{k+1}=0$, $\mathbb{E}Z_{k+1}^3=0$, $\mathbb{E}Z_{k+1}^2=h$, and $\mathbb{E}Z_{k+1}^4=\frac95h^2$, we have
    \begin{equation}
    \begin{split}
        \mathbb{E}_x(X_{k+1}^4)
        &=
        (x+b(x)h)^4
        +6(x+b(x)h)^2\sigma^2(x)h
        +\frac95\sigma^4(x)h^2\\
        &=
        x^4
        +4x^3b(x)h
        +6x^2b^2(x)h^2
        +4xb^3(x)h^3
        +b^4(x)h^4\\
        &\quad
        +6x^2\sigma^2(x)h
        +12xb(x)\sigma^2(x)h^2
        +6b^2(x)\sigma^2(x)h^3
        +\frac95\sigma^4(x)h^2.
    \end{split}
    \end{equation}
    By the linear-growth assumption, $|b(x)|+|\sigma(x)|\le L(1+x)$. Hence, using $h<1$ and $(1+x)^4\le C(1+x^4)$ for $x\ge0$, each term after $x^4$ is bounded by $Ch(1+x^4)$. For example,
    $$|x^3b(x)h|\le Chx^3(1+x) \le 2Ch(1+x^4),$$
    $$x^2b^2(x)h^2+x^2\sigma^2(x)h\le Chx^2(1+x)^2\le 3Ch(1+x^4),$$
    and the remaining terms are bounded similarly by $Ch(1+x)^4\le Ch(1+x^4)$. Therefore,
    $$\mathbb{E}_x(X_{k+1}^4)\le x^4+Ch(1+x^4)\le (1+Ch)x^4+Ch,\qquad x\in\Omega_e.$$

    Next consider $x\in\Omega_{br}$. By Lemma~\ref{lemma:bound_on_drift_diffusion}, after decreasing $\delta$ if necessary, $x\le C\sqrt h$ and $b,\sigma$ are bounded on the boundary layer. The scheme either jumps to $0$ or makes a reflected Euler--Maruyama-type update, so in either case
    $$
    X_{k+1}\le x+\sigma_2\sqrt{3h}+B_Mh\le C\sqrt h.
    $$
    Hence
    $$
    \mathbb{E}_x(X_{k+1}^4)\le Ch^2\le Ch
    \le (1+Ch)x^4+Ch,
    \qquad x\in\Omega_{br}.
    $$
    Combining the two cases gives
    $$
    \mathbb{E}_x(X_{k+1}^4)\le (1+Ch)x^4+Ch,
    \qquad x\ge 0.
    $$
    Taking expectation with respect to $X_k$ yields
    $$
    \mathbb{E}(X_{k+1}^4)\le (1+Ch)\mathbb{E}(X_k^4)+Ch.
    $$
    By the discrete Gronwall inequality, for $0\le k\le N$ and $Nh=T$,
    $$
    \mathbb{E}(X_k^4)\le e^{CT}\left(\mathbb{E}(X_0^4)+CT\right).
    $$
    Therefore,
    $$
    \sup_{0\le k\le N}\mathbb{E}(1+X_k^4)\le C_T,
    $$
    where $C_T$ is independent of $h$.
\end{proof}

\end{appendix}

\end{document}